\documentclass{amsart}

\usepackage{amssymb, verbatim, multirow, xcolor, graphicx, caption, subcaption}
\usepackage{hyperref}
\usepackage{amsmath,wrapfig}

\def\dist{\mathop{\rm dist}}

\def\diam{\mathop{\rm diam}}

\def\XXint#1#2#3{{\setbox0=\hbox{$#1{#2#3}{\int}$ }
\vcenter{\hbox{$#2#3$ }}\kern-.6\wd0}}

\newcounter{Examp}
\newcounter{Defs}
\newtheorem{theorem}{Theorem}[section]
\newtheorem{proposition}[theorem]{Proposition}
\newtheorem{lemma}[theorem]{Lemma}

\newtheorem{corollary}[theorem]{Corollary}

\begin{document}

\title[A Liouville theorem]{Regularity and a Liouville theorem for a class of boundary-degenerate second order equations}
\date{December 2019}
\author{Brian Weber}
\maketitle

\begin{abstract}
	We study a class of second-order boundary-degenerate elliptic equations in two dimensions with minimal regularity assumptions.
	We prove a maximum principle and a Harnack inequality at the degenerate boundary, and assuming local boundedness, we prove continuity.
	On globally defined non-negative solutions we provide strong constraints on behavior at infinity, and prove a Liouville-type theorem for entire solutions on the closed half-plane.
	The class of PDE in question includes many from mathematical finance, Keldysh- and Tricomi-type PDE, and the 2nd order reduction of the fully non-linear 4th order Abreu equation from K\"ahler geometry.
\end{abstract}

\section{Introduction}

We study solutions $f\ge0$ of $L(f)=0$ where $L$ is the operator
\begin{eqnarray}
	L(f)\;=\;y^2(f_{xx}+f_{yy})
	+y\left(b_1f_x\,+\,b_2f_y\right)
	+cf
	\label{EqnOpMain}
\end{eqnarray}
on the open half-plane $H^2=\{(x,y)\,|\,x\in\mathbb{R},\,y\in(0,\infty)\}$ and its closure $\overline{H}{}^2$, assuming bounded and measurable coefficients $b_1$, $b_2$, $c$.
Of the solutions we require interior local boundedness, and for some of our results such as continuity at the boundary, local boundedness at $\{y=0\}$ as well.
None of our results assume global boundedness, higher regularity, or growth constraints on solutions.
Most of our results require $c\ge0$ and $b_2\ge1$.

Solutions of $L(f)=0$ show major qualitative changes in behavior between $b_2\ge1$ and $b_2<1$.
We explore the phenomenon that, for $b_2\in(0,1)$, the boundary-value problem is both well posed and not well posed at the degenerate boundary, in different ways.
For $b_2\in(0,1)$ we show boundary data can be specified at the degenerate boundary, and the boundary value problem is well-posed in the Hadamard sense.
Nevertheless the boundary value problem remains well-posed if boundary data is specified everywhere {\it except} the degenerate boundary, as shown in \cite{FP}.
See below for a discussion.

Our primary concern is the $b_2\ge1$ case, for which the boundary value problem at $\{y=0\}$ is never well-posed.
We prove a Harnack inequality and a maximum principle at $\{y=0\}$, and, assuming local boundedness of solutions, we prove continuity.
Globally on $H^2$ we discover strong controls on the behavior of solutions for large $y$.
On the closed half-plane $\overline{H}{}^2$ when $c=0$, we prove an analogue of the classical Liouville theorem: any non-negative entire solution is constant.

We shall refer to the kind of boundary degeneracy in (\ref{EqnOpMain})---where the $k^{th}$ order terms are multiplied by $y^k$---as {\it Euler-type} degeneracy, in reference to the classic Euler-type differential equation $y^2f_{yy}+Byf_y+C=0$.
Operators of Euler-type that have $L^\infty$-bounded coefficients are effectively invariant under simultaneous scaling of coordinates, which makes possible this paper's point-picking and blow-up style arguments.
Scale-invariance also suggests that $L^\infty$ bounds on $b_1$, $b_2$, and $c$ should perhaps be the right bounds to consider---as opposed to $L_{loc}^p$ bounds for instance---as the $L^\infty$ norm is the only $L^p$ norm invariant under coordinate scaling, and lower semi-continuous under the taking of limits after coordinate scaling.

Operators of this form have been studied in K\"ahler geometry, manifold embedding, stochastic PDE, financial modeling, population dynamics, diffusion-transport phenomena, nonlinear elliptic and fractional-Laplacian boundary problems, Keldysh- and Tricomi-type boundary problems of elliptic and mixed elliptic/hyperbolic type, fluid propagation through porous media, cold plasmas, Prandtl boundary-layer problems, and in other applications.
In the case of certain operators, such as those of SABR, Keldysh, or Tricomi type, the fact that they can be transformed into (\ref{EqnOpMain}) and vice-versa seems possibly under-appreciated to date.

\subsection{Summary}

Our most general result is a gradient bound for solutions on the open half-plane $H^2$.
This requires bounded, measurable coefficients, but no sign restrictions on $c$ or $b_2$, and solutions need not be bounded at $\{y=0\}$.
\begin{proposition}[Interior gradient estimate, {\it cf.} Proposition \ref{PropInteriorInitial}] \label{PropIntGradFirst}
	Assume we have bounds $|b_1|,|b_2|,|c|\le\Lambda$, and assume $f\ge0$ satisfies $|L(f)|\le\Lambda$ weakly on the open half-plane $H^2$.
	Then a constant $D=D(\Lambda)$ exists so that
	\begin{eqnarray}
	y|\nabla\log{}f|\;\le\;D.
	\end{eqnarray}
\end{proposition}
This immediately provides growth/decay bounds on solutions $f\ge0$, specifically exponential growth/decay in $x$ and polynomial growth/decay in $y$.
Certainly $f$ may be unbounded near $\{y=0\}$, but this proposition shows its growth cannot be any worse than $y^{-D}$.
This proposition definitely fails if the 2-sided bound $|L(f)|\le\Lambda$ is replaced with a 1-sided bound, say $L(f)\le0$; see Example \ref{ExSupers}.

We prove a localized version of the gradient estimate near the boundary $\{y=0\}$.
This is required later for our continuity result.
\begin{proposition}[Localized interior gradient estimate, {\it cf.} Proposition \ref{PropLocDomVer}] \label{PropFirstLocalGradEst}
	Assume $|b_1|,|b_2|,|c|\le\Lambda$.
	There exists a constant $D=D(\Lambda)$ so that the following holds.
	If $p\in\{y=0\}$, $\Omega$ is a neighborhood of $p$, $f\ge0$, and $|L(f)|\le\Lambda$ on $\Omega^{Int}$, then there exists some neighborhood $\Omega'\subset\Omega$ of $p$ so that on $\Omega'{}^{,Int}$,
	\begin{eqnarray}
		y|\nabla\log{}f|\;\le\;D.
	\end{eqnarray}
\end{proposition}

In addition to the localized gradient estimate we prove three other results on the behavior of solutions locally near the boundary: unspecifiability, continuity, and the Harnack inequality.
First we make precise what we mean by a solution of $L(f)=g$ on a region that includes some part of the degenerate boundary.

\refstepcounter{Defs}
{\bf Definition \arabic{Defs}}. (Neighborhoods and open sets at the degenerate boundary.)  \label{DefOpens}
If $p\in\{y=0\}$ then we call $\Omega$ a neighborhood of $p$ if $p\in\Omega$ and there is some open set $\Omega'\subset\mathbb{R}^2$ such that $\Omega=\Omega'\cap\overline{H}{}^2$.
An open set in $\overline{H}{}^2$ will be any set $\Omega$ such that there is some open set $\Omega'\subset\mathbb{R}^2$ with $\Omega=\Omega'\cap\overline{H}{}^2$.

\refstepcounter{Defs}
{\bf Definition \arabic{Defs}}. (Degenerate and non-degenerate boundary components.) \label{DefBoundComps}
Assume $\Omega$ is a set in $\overline{H}{}^2$.
Then its boundary $\partial\Omega$ is divided into its {\it degenerate} and {\it {\it non-degenerate}} boundary components
\begin{eqnarray}
	\partial_0\Omega\;=\;\partial\Omega\cap\{y=0\} \quad \text{and} \quad
	\partial_1\Omega\;=\;\partial\Omega\cap\{y>0\}.
\end{eqnarray}
We define the interior of $\Omega$ to be $\Omega^{Int}=\Omega\setminus(\partial_0\Omega\cup\partial_1\Omega)$.

\refstepcounter{Defs}
{\bf Definition \arabic{Defs}}. (Solutions, subsolutions, supersolutions in the interior.) \label{DefSols}
Assuming an open set $\Omega$ contains no points of $\{y=0\}$, then we say that $L(f)=g$ (or $L(f)\ge{}g$ or $L(f)\le{}g$, etc) on $\Omega$ provided $L(f)=g$ (or $L(f)\ge{}g$ etc) holds weakly or in the viscosity sense on $\Omega$.

\refstepcounter{Defs}
{\bf Definition \arabic{Defs}}. (Solutions, subsolutions, supersolutions at the degenerate boundary.) \label{DefSolsAtDegBound}
If $\Omega$ is an open set in the sense of Definition \ref{DefOpens} that contains points of $\{y=0\}$, then we say that $L(f)=g$ (or $L(f)\ge{}g$, $L(f)\le{}g$, etc) provided $L(f)=g$ (or $L(f)\ge{}g$ etc) holds weakly or in the viscosity sense on all interior points of $\Omega$, and also if whenever $p\in\Omega\cap\{y=0\}$, then if $\{p_i\}$ is a sequence of points in the interior of $\Omega$ converging to $p$, $\liminf_{p_i}f(p_i)$ and $\limsup_{p_i}f(p_i)$ are both finite.

In other words, for us to consider a solution to exist at a degenerate boundary point it need only be {\it locally finite} on nearby interior points.
This very weak constraint is actually enough to force continuity and other strong restrictions on the behavior of solutions at $\{y=0\}$.

\begin{proposition}[Non-specifiability at the degenerate boundary; {\it cf.} Proposition \ref{PropUnspecifiability}] \label{PropUnspecifiabilityIntro}
	Let $p=(x_0,0)$ be a point on the boundary line $\{y=0\}$ and assume $b_2\ge1$ and $c\le\inf\frac14(b_1-1)^2$ in some pre-compact neighborhood $\Omega$ of $p$ (see Definition \ref{DefOpens}).
	Assume two functions $f_1,f_2$ satisfy $L(f_i)=g$ on $\Omega$ (Definition \ref{DefSolsAtDegBound}) and that $f_1=f_2$ on $\partial_1\Omega$.
	Then $f_1=f_2$.
\end{proposition}

Non-specifiability definitely fails when $b_2<1$; see Example \ref{ExMaxFailure}.
For $b_2\ge1$ the degenerate boundary portion $\partial_0\Omega$ is sometimes also called the {\it interior boundary}, for the reason that the values of $f$ on $\partial_0\Omega$ are completely determined by the values of $f$ on $\partial_1\Omega$, as with actual interior points.

\begin{theorem}[Harnack inequality at the degenerate boundary, local version; {\it cf.} Theorem \ref{ThmBoundHarnackLocal}] \label{ThmBoundHarnackLocalIntro}
	Assume $|b_1|,|b_2|,|c|\le\Lambda$ and $b_2\ge1$, $c\ge0$.
	Assume $f\ge0$ satisfies $L(f)\le0$ and $L(f)\ge-\Lambda$ on the semi-closed rectangle
	\begin{eqnarray}
	\mathcal{R}_{y_0}\;=\;\left\{
	\,(x,y)\;\;\big|\;\;x\in[-4\Lambda{}y_0,4\Lambda{}y_0],\;y\in(0,y_0]
	\right\}.
	\end{eqnarray}
	Then $f(0,0)\;\ge\;\frac19\inf_{x\in[-4\Lambda{}y_0,4\Lambda{}y_0]}f(x,y_0)$.
\end{theorem}

\begin{theorem}[Continuity at the degenerate boundary, {\it cf.} Theorem \ref{ThmContinuity}] \label{ThmContinuityIntro}
	Let $p\in\{y=0\}$ and assume $\Omega$ is a neighborhood of $p$ (see Definition \ref{DefOpens}).
	Assume $|b_1|,|b_2|,|c|\le\Lambda$ and $b_2\ge1$, $c\ge0$.
	If $L(f)=0$ on $\Omega$ (Definition \ref{DefSolsAtDegBound}), then $f$ is continuous at $p$.
\end{theorem}

Continuity definitely fails if $b_2<1$---see Example \ref{ExHomogAndSteps}---or if the requirement of local finiteness of solutions fails at even a single point of the degenerate boundary; see Example \ref{ExImps}.
In the case $b_2\ge1$ we believe the regularity can be strengthened to $C^{1,\alpha}$, or to $C^{k+2,\alpha}$ if the coefficients are in $C^{k,\alpha}$.
Since this would require techniques beyond those we consider, and since we wish to keep the present study restricted in scope to a few core techniques (those being scaling/blowup methods and barrier methods), we leave the question of optimality for the future.
See Conjecture 1.

\begin{theorem}[A maximum principle at the degenerate boundary, {\it cf.} Theorem \ref{ThmWeakMax}] \label{ThmWeakMaxIntro}
	Assume $p\in\{y=0\}$ and $\Omega$ is some neighborhood of $p$ (see Definition \ref{DefOpens}).
	Assume $|b_1|,|b_2|,|c|\le\Lambda$ and $b_2\ge1$, $c\ge0$.
	
	If $f$ solves $L(f)\le0$ on $\Omega$ (Definition \ref{DefSolsAtDegBound}), then $f(p)$ is not a strict local minimum on $\Omega$.
	If $f$ solves $L(f)\ge0$ on $\Omega$ (Definition \ref{DefSolsAtDegBound}), then $f(p)$ is not a strict local maximum on $\Omega$.
\end{theorem}

The maximum principle at $\{y=0\}$ definitely fails when $b_2<1$; see Example \ref{ExMaxFailure}.
However, compare to \cite{Fee}, where a maximum principle is recovered even for $b_2\in(0,1)$, after apriori differentiability assumptions are made.
See also \cite{LP} where a maximum principle exists for a certain Tricomi operator that is equivalent to an Euler-type operator with $b_2=\frac13$ (see Section \ref{SubSecPopDyn}), and which also requires a strong differentiability condition, as it must.

With these local theorems in hand we move on to our global theorems, which require $f\ge0$ to exist on the entire half-plane $H^2$ or, in the case of the Liouville theorem, on its closure.
\begin{theorem}[Global version of the Harnack inequality, {\it cf.} Theorem \ref{ThmBoundHarnackGlobal}]
	There exists a number $\delta=\delta(\Lambda)$ so that the following holds.
	Assume $|b_1|,|b_2|,|c|\le\Lambda$ and $b_2\ge1$, $c\ge0$, and assume $f\ge0$ satisfies $L(f)\le0$ and $L(f)\ge-\Lambda$ on the open half-plane.
	
	Then $f(x_0,y')\;\ge\;\delta\cdot{}f(x_0,y)$ where $(x_0,y_0)$ is any point in the open half-plane and $y'\in[0,y]$.
\end{theorem}

Next, our ``almost monotonicity'' result states (for $c\ge0$) that $f$ is bounded at infinity.
When $c=0$ the result is far stronger, stating that $f$ takes its global minimum at infinity, and $f$ limits to this global minimum along any ray $y\mapsto(x_0,y)$.
The term ``almost monotonicity'' refers to the fact that $y\mapsto{}f(x_0,y)$, while perhaps not strictly decreasing to the minimum, can never increase by very much as $y$ gets larger.
Proposition \ref{PropAlmostMonoIntro} does not require any local boundedness at $\{y=0\}$.
\begin{proposition}[Almost Monotonicity, {\it cf.} Proposition \ref{PropAlmostMonot}] \label{PropAlmostMonoIntro}
	Assume $|b_1|,|b_2|,|c|\le\Lambda$ and $b_2\ge1$, $c\ge0$, and assume $f\ge0$ satisfies $L(f)\le0$ and $L(f)\ge-\Lambda$ on the open half-plane $H^2$.
	Let $x_0\in\mathbb{R}$ and consider the function $y\mapsto{}f(x_0,y)$.
	
	A number $\delta=\delta(\Lambda)>0$ exists so that
	\begin{eqnarray}
	\limsup_{y\rightarrow\infty}f(x_0,y)\;\le\;\delta^{-1}\inf_{(x,y)\in{}H^2}f(x,y).
	\end{eqnarray}
	Further, $y\mapsto{}f(x_0,y)$ has the ``almost monotonicity'' property, namely that
	\begin{eqnarray}
	f(x_0,y_2)<\delta^{-1}f(x_0,y_1)
	\end{eqnarray}
	whenever $y_2>y_1$.
	
	Additionally, in the case $c=0$, for any fixed $x_0$ we have that $\lim_{y\rightarrow\infty}f(x_0,y)$ exists and equals $\inf_{H^2}f$.
\end{proposition}

\begin{proposition}[Polynomial bounds in $x$, {\it cf.} Proposition \ref{PropPolyXBounds}] \label{PropPolyXBoundsIntro}
	Assume $|b_1|,|b_2|,|c|\le\Lambda$ and $b_2\ge1$, $c\ge0$, and assume $f\ge0$ satisfies $L(f)\le0$ and $L(f)\ge-\Lambda$ on the open half-plane $H^2$.
	There exists a constant $\delta=\delta(\Lambda)>0$ so that for any two values $x,x'\in\mathbb{R}$ we have the growth/decay bounds
	\begin{eqnarray}
		\begin{aligned}
		&f(x,y)
		\;\le\;f(x',y)\frac{1}{\delta}\left(\frac{|x-x'|}{y}+1\right)^{D} \\
		&f(x,y)
		\;\ge\;f(x',y)\delta\left(\frac{|x-x'|}{y}+1\right)^{-D}
		\end{aligned}
	\end{eqnarray}
	where $D=D(\Lambda)$ is the constant from Proposition \ref{PropIntGradFirst}.
\end{proposition}

\begin{theorem}[The Liouville theorem, {\it cf.} Theorem \ref{ThmLiouvilleActual}] \label{ThmLiouvilleFirstStatement}
	Assume $|b_1|,|b_2|\le\Lambda$, $b_2\ge1$, and $c=0$.
	Assume $f\ge0$ satisfies $L(f)=0$ on the closed half-plane $\overline{H}{}^2$.
	Then $f$ is constant.
\end{theorem}
Theorem \ref{ThmLiouvilleFirstStatement} makes no assumption on the growth of solutions, and no form of regularity outside of local finiteness.
This theorem is quite sharp, as we demonstrate in the examples of Section \ref{SecExamps}.
Restricting, say, to the strip $\{0\le{}y\le1\}$ any uniqueness of solutions is definitely false, even after specifying boundary values; see example \ref{ExNonUniqStrips}.
If $f$ violates local finiteness, even at a single point of $\{y=0\}$, Example \ref{ExImps} shows the Liouville theorem fails.
Examples \ref{ExHomogAndSteps}, \ref{ExHestonGradFail}, \ref{ExNegC}, and \ref{ExPosC} show how the Liouville theorem fails under other forms of weakened hypotheses, for example allowing $b_2<1$ or $c\ne0$.

We remark that Theorem 1.5 of \cite{FP} says something similar to our Liouville theorem, except there the coefficients are assumed constant, $c$ has a definite sign, and solutions are apriori assumed to be bounded on two sides (or must at a minimum have something like polynomial growth constraints or else the Fourier methods of \cite{FP} won't apply).
We point out that the hypotheses of Theorem 1.5 of \cite{FP} are unclear, as apriori differentiability and boundedness assumptions are left unstated but are certainly necessary there\footnote{Without these assumptions, a counterexample is $f(x,y)=I_0(\sqrt{y})$, which solves $y\triangle{}f+f_y-\frac14f=0$ and is $C^\infty$, entire, and non-negative.
A non-smooth but still $C^{0,1/2}$ counterexample is $f(x,y)=Exp(-\sqrt{y})$, which solves $y\triangle{}f+\frac12f_y-\frac14f=0$ and is both entire and bounded.} (this is undoubtedly just an oversight as these missing hypotheses are present in other theorems of \cite{FP}).

We provide a single result in the case $b_2<1$.
If $\lambda$ is constant and $f(x,y)$ solves $y^2\triangle{}f+(1+\lambda)yf_y=0$ then $\tilde{f}(x,y)=y^{\lambda}f(x,y)$ solves 
\begin{eqnarray}
	y^2\triangle\tilde{f}+(1-\lambda)y\tilde{f}_y=0. \label{EqnSolModEqn}
\end{eqnarray}
Using this simple trick we obtain the following corollary of the Liouville theorem which we record mainly due to its applicability in K\"ahler geometry (see \S\ref{SubSecAbreu}).
\begin{corollary}[{\it cf.} Corollary \ref{CorB2Less}] \label{CorB2LessIntro}
		Assume $\lambda>0$ is a constant and $f\ge0$ solves
		\begin{eqnarray}
			y^2\triangle{}f\,+\,(1-\lambda)yf_y\;=\;0
		\end{eqnarray}
		on the upper half-plane.
		Assume $f$ is continuous at $\{y=0\}$, and $f(x,0)=0$.
		Then $f$ is a positive multiple of the power function $y^\lambda$:
		\begin{eqnarray}
			f(x,y)\;=\;C_1y^\lambda.
		\end{eqnarray}
\end{corollary}
Our results have implications for certain Keldysh-type operators.
We record just two results: the first is a restatement of almost-monotonicity and the second is a restatement of the Liouville theorem.
\begin{theorem} \label{ThmAMKeldysh}
	Assume $L$ is the Keldysh operator
	\begin{eqnarray}
		L(f)\;=\;f_{xx}\,+\,y^kf_{yy}
	\end{eqnarray}
	with $k>2$.
	Assume $f\ge0$ is locally bounded and $L(f)=0$ in the open half-plane.

	Then $f$ is continuous at the degenerate boundary $\{y=0\}$, and, at this boundary, the function $x\mapsto{}f(0,x)$ is constant and equal to $\inf_{H^2}f$.
\end{theorem}
\begin{theorem}[Liouville theorem for certain Keldysh operators] \label{ThmLiouvilleKeldysh}
	Assume $L$ is the Keldysh operator
	\begin{eqnarray}
	L(f)\;=\;f_{xx}\,+\,y^kf_{yy}
	\end{eqnarray}
	with $k>2$.
	Assume $f\ge0$ is locally bounded and $L(f)=0$ in the open half-plane.
	Further assume that along any ray $y\mapsto{}f(y,x_0)$ where $x_0$ is fixed, $f$ is finite.
	Then $f$ is constant.
\end{theorem}
This Liouville theorem is false if $k=2$, as $f(x,y)=y^{\frac13}\cosh(\frac{\sqrt{2}}{3}x)$ shows.
We do not at the moment have either a counterexample or a proof of Theorem \ref{ThmAMKeldysh} when $k=2$.
The proofs of these two theorems are given in Section \ref{SubSecPopDyn}.

\subsection{The significance of $b_2\ge1$} \label{SubSecImportanceOfB2}
Between $b_2\ge1$ and $b_2<1$ major changes occur in the nature of solutions.
If $b_2\in(0,1)$ the degenerate boundary takes on some characteristics of a non-degenerate boundary, and all four of our ``local'' theorems \ref{PropUnspecifiabilityIntro}, \ref{ThmBoundHarnackLocalIntro}, \ref{ThmContinuityIntro}, and \ref{ThmWeakMaxIntro} are false---there is no maximum principle, no Harnack inequality, no differentiability in general, and one {\it can} specify boundary values at $\{y=0\}$, as demonstrated in Example \ref{ExImps}.
That said, in \cite{Fee} and \cite{FP} we see maximum principles and non-specifiability at the degenerate boundary for all $b_2>0$, even $b_2\in(0,1)$.
What's going on?

The results of \cite{Fee} and \cite{FP} require, apriori, that solutions possess a strong form of differentiability at $\{y=0\}$; after assuming this differentiability, \cite{FP} then controls it quantitatively.
What we are seeing is that, for $b_2\in(0,1)$, maximum principles and uniqueness are false under mere local boundedness or even continuity, but becomes true again if twice differentiability is assumed\footnote{A $C^2$ assumption is slightly too strong; see \cite{FP} for the weakest known requirement, and Conjecture 1 for what we believe is the optimal requirement.}.
A result of \cite{FP} is that for $b_2>0$ boundary specifications on $\partial_1\Omega$ {\it automatically} produce boundary values at $\partial_0\Omega$, and these automatic boundary values have good differentiability.
By contrast, the method of Example \ref{ExImps} shows that when $b_2\in(0,1)$ one may specify {\it arbitrary} boundary values at $\{y=0\}$.
This apparent conflict is resolved by the fact that, for $b_2\in(0,1)$, only $C^{0,\alpha}$ regularity can be expected no matter how smooth the boundary data is, and only for certain $\alpha$.
Smoothness occurs only for the highly exceptional boundary values at $\partial_0\Omega$ found by \cite{FP}.

A different qualitative change also occurs, this time at infinity.
When $b_2\ge1$ we have almost-monotonicity, Proposition \ref{PropAlmostMonoIntro}, which tells us $f$ is bounded at infinity by a definite multiple of $\inf_{H^2}f$---indeed almost-monotonicity can be thought of as a kind of Harnack inequality at infinity, although this is not fully accurate.
But almost-monotonicity fails when $b_2<1$ and we no longer see such highly constrained behavior at infinity.
For example the function $f(x,y)=y^{1/4}$ is non-negative on the half-plane, is unbounded, and solves $y^2\triangle{}f+(3/4)yf_y=0$---this function has poor regularity at the boundary (only $C^{0,1/4}$) and grows unboundedly as $y\rightarrow\infty$.
Obviously then the Liouville theorem, too, is false for $b_2=1/4$.

We remark that between $b_2>0$ and $b_2\le0$ an entirely separate ``phase change'' occurs in the nature of solutions at $\{y=0\}$.
When $b_2\le0$ the apparently degenerate boundary $\{y=0\}$ becomes fully non-degenerate, with all the regularity and non-regularity phenomena that one would expect at any other boundary.

Another expression of this stark division in behaviors lies in the not inconsiderable distinction between Keldysh and Tricomi operators.
As discussed in Section \ref{SubSecPopDyn}, if $b_2\in(0,1)$ the Euler-type operator (\ref{EqnOpMain}) transforms into a Tricomi operator, whereas if $b_2\in(-\infty,0]\cup[1,\infty)$ then it transforms into a Keldysh operator.

\subsection{Motivation}
Much of our motivation comes from K\"ahler geometry and in particular the study of the Abreu equation, discussed in \S\ref{SubSecAbreu} below.
Our Liouville theorem has strong implications for broad classes of canonical metrics on K\"ahler 4-manifolds.

For certain reasons, works already in the literature can be difficult to apply.
The works \cite{Fee}, \cite{FP} contain some of the same results as those found here, but require extrinsic differentiability assumptions at $\{y=0\}$.
The uniqueness results of \cite{FP} also require a uniform boundedness assumption on solutions\footnote{The existence-uniqueness statements of \cite{FP}, Theorems 1.6 and 1.11, have a (surely inadvertent) misstated hypothesis. The statements assert existence/uniqueness of solutions $u$ under the supposition $u\in{}C^{k,2+\alpha}_s$, when surely they mean to suppose $\|u\|_{C^{k,2+\alpha}_s}<\infty$. If not then Example \ref{ExNonUniqStrips} is a counterexample. In Theorem 1.6 they also neglected to mention a sign restriction on $c$, which is necessary for uniqueness even under strong differentiability and boundedness assumptions.}.
See Example \ref{ExNonUniqStrips} for non-uniqueness under conditions of local but not global boundedness.

The authors of \cite{FP} impose such strong apriori conditions because their aim is to {\it create} solutions of $L(f)=0$, given only boundary values on $\partial_1\Omega$, thereby showing well-posedness when boundary values are only specified on $\partial_1\Omega$.
However when solutions are simply found, already existing, in some naturalistic setting, there may be no reason to assume apriori boundedness or any regularity at $\partial_0\Omega$.
This work is motivated by the need to address this type of situation.

\subsection{Organization}
Section 2 outlines a few of the situations where our results may find use, from differential geometry to financial market modeling.
By means of some coordinate transformations that seem absent from the literature, we show that several well-known equations such as the SABR equation actually have a very orthodox form of Euler-type degeneracy.

In Section 3 we prove the all-important interior gradient estimate, using a scaling/blow-up style argument of the sort found frequently in differential geometry.
Section 4 deploys the interior gradient estimate in combination with the lower barrier of Lemma \ref{LemmaISubfunction} to enforce powerful constraints on the behavior of solutions at both $\{y=0\}$ and infinity.
Section 5 uses a different lower barrier, created in Proposition \ref{PropPolyXBounds}, to improve the exponential growth/decay estimate of Section 3 to polynomial growth/decay.
Then the Liouville theorem is proved with a combination point-picking and upper barrier argument.
We close the paper with a set of examples that demonstrate the sharpness of our theorems.

The literature on boundary-degenerate equations is vast, and probably intractable.
The foundational results are contained in probably several dozen works, and an accounting of the most valuable theory-based papers probably numbers in the low hundreds.
Many hundreds more papers make significant contributions to the mathematics, physical science, engineering, and financial modeling aspects of these equations.
A tiny sampling can begin with the field's origins in papers of Tricomi \cite{Tric}, Keldysh \cite{Keld}, and Fichera \cite{Fich1} \cite{Fich2} \cite{Fich3} (unfortunately some of these papers have never been translated) where it was first noticed that well-posedness sometimes requires exclusion of boundary data on certain boundary portions.
The book by Oleinik-Radkevic \cite{OR} contains this prior work and much more, and one can perhaps follow this with the Kohn-Nirenburg paper \cite{KN}.
The book \cite{Otw} contains a great deal of information on Keldysh and Tricomi operators.
The material in the papers \cite{Fee} \cite{FP}, touched on above, is probably closest in subject matter to ours.

The avalanche of papers has not abated in recent years, and numerous recent works explore themes closely adjacent to ours.
A variety of Liouville and Harnack theorems involving boundary-degenerate equations, sometimes in the fractional Laplacian setting, are now available; for a tiny sampling see \cite{ST} \cite{CT} \cite{HL} \cite{JS}.
We believe our paper addresses a considerable gap in the literature, and possesses an attractive breadth and simplicity in its assumptions.

Finally, examination of our examples leads us to offer two conjectures.

{\bf Conjecture 1.} (Optimal regularity threshold at the degenerate boundary.)
{\it Let $\alpha\in(0,1)$ and assume $b_2>1-\alpha$ in some neighborhood $\Omega$ of a point $p$ of a degenerate boundary component (see Definition \ref{DefOpens}).
Assume the coefficients $b_1,b_2,c$ are measurable.
Assume $f$ solves $L(f)=0$ and $f$ is locally finite in $\Omega$.

If $f\in{}C^{0,\alpha}(\partial_0\Omega\cup\Omega^{Int})$, then there is some neighborhood $\Omega'$ of $p$ with $\Omega'\subset\Omega$, so that $f\in{}C^{1,\beta}(\partial_0\Omega'\cup\Omega'{}^{Int})$ for all $\beta\in[0,1)$.
Similarly, if $b_1,b_2,c\in{}C^{k,\beta}(\partial_0\Omega\cup\Omega^{Int})$, then $f\in{}C^{k+2,\beta}(\partial_0\Omega'\cup\Omega'{}^{Int})$) .

In the case that $b_2\ge1$, if $f$ is locally bounded then $f\in{}C^{1,\beta}(\partial_0\Omega'\cup\Omega'{}^{Int})$ for all $\beta\in(0,1)$ (or $f\in{}C^{k+2,\beta}(\partial_0\Omega'\cup\Omega'{}^{Int})$ if $b_1,b_2,c\in{}C^{k,\beta}(\partial_0\Omega\cup\Omega^{Int})$).}

{\bf Conjecture 2.} (Liouville theorem for the steady-state Heston equation.)
{\it Consider the Heston operator $L_H$, given by
\begin{eqnarray}
	L_H
	\;=\;y\triangle
	+\left(b_1+B_1{}y\right)\frac{\partial}{\partial{x}}
	+\left(b_2+B_2{}y\right)\frac{\partial}{\partial{y}}
	-rU,
\end{eqnarray}
with measurable, bounded coefficients $b_1$, $B_1$, $b_2$, $B_2$, and $r$.
Let $\epsilon$ be nonzero.
A non-negative, locally finite solution $L_H(f)=0$ on the closed half-plane $\overline{H}{}^2$ with $b_2>0$, $B_2<-\epsilon^2$, and $r\le0$ is necessarily constant.
The solution is zero if $r<0$.} \\

See \S\ref{SubSecHestonDescription} and Example \ref{ExHestonGradFail} for more discussion about the Heston operator.

\section{Equations with Euler-type degeneracy} \label{SecVarEqns}

We give a sampling of operators with Euler-type degeneracy and their applications.
The prototype is the homogeneous Euler ordinary differential equation
\begin{eqnarray}
	y^2f_{yy}+Byf_y+Cf\;=\;0. \label{EqnEulerModel}
\end{eqnarray}
If we demand solutions remain non-negative, it is necessary that $C\le\frac14(1-B)^2$.
Most solutions have the form $y^{1-B}$ so when $B<1$ solutions are bounded at $0$ and unbounded at infinity, and when $B>1$ solutions are unbounded at $0$ and are zero at infinity.
As expected, solutions show major qualitative changes at $B=1$.
Equation (\ref{EqnEulerModel}) is our model ODE, and the behavior of the model solutions $y^{1-B}$ helps us build barriers for solutions $f\ge0$ of our PDE $L(f)=0$ when $b_2\ge1$, but not when $b_2<1$.

\subsection{Transport-Diffusion in a Hyperbolic metric}

An extremely natural appearance of the operator (\ref{EqnOpMain}) is in the diffusion-transport problem in the hyperbolic metric on the half-plane.
Using the familiar $g_{ij}=y^{-2}\delta_{ij}$ and letting $\vec{B}$ be the vector field $\vec{B}=y\vec{b}=y\left(b_1\partial_x+b_2\partial_y\right),$ then the norms of the fields $\vec{B}$, $\vec{b}$ in their respective metrics are identical: $|\vec{B}|^2_g=|\vec{b}|^2_{Eucl.}=(b_1)^2+(b_2)^2$.
Then (\ref{EqnOpMain}) is a diffusion-transport operator with a bounded transport field:
\begin{eqnarray}
	\begin{aligned}
	L(f)
	&\;=\;y^2\left(f_{xx}+f_{yy}\right)+y(b_1f_x+b_2f_y)+cf \\
	&\;=\;\triangle_gf
	\,+\,\big<\vec{B},\,\nabla{}f\big>_g
	\,+\,cf.
	\end{aligned}
\end{eqnarray}
When the ``catalysis'' coefficient $c$ is zero, our Liouville theorem states that, provided $b_2\ge1$, the only steady-state solutions are the constant solutions.
We remark that a qualitative change in behavior still occurs when $b_2<1$.
One wonders what the invariant meaning behind this change in behavior might be.

Daskalopoulos-Hamilton \cite{DH} present a similar interpretation of the operator $L$, except instead of working in the hyperbolic metric, they interpreted the slightly different operator $L=\frac12y\triangle{}y+\nu\partial_y$ as a diffusion-transport operator on the metric $g_{ij}=(2y)^{-1}\delta_{ij}$, which they term the {\it cycloidal} metric.
Daskalopoulos-Hamilton employed this metric to great success, but we remark that the cycloidal metric is incomplete and has unbounded Gaussian curvature, and the transport field $\vec{b}=\nu\partial_y$ has unbounded norm.

\subsection{Population Dynamics, and Keldysh and Tricomi operators} \label{SubSecPopDyn}

Keldysh operators in two variables take the form $L(u)=u_{xx}+K(y)u_{yy}$ and Tricomi operators take the form $L(u)=K(y)u_{xx}+u_{yy}$, modulo lower order terms, where it is required that $K=0$ along a ``parabolic curve'' that separates the elliptic from the hyperbolic regime.
The associated boundary value problem goes back to \cite{Keld}; see \cite{Otw} for a thorough treatment.
On the ``elliptic side,'' where $K\ge0$, Keldysh and Tricomi operators can be transformed into operators with Euler-type degeneracy.

One place this type of operator appears is in population dynamics.
Epstein-Mazzeo studied diffusion processes in population dynamics in the extended work \cite{EM}, with Keldysh-type operators of the form
\begin{eqnarray}
	\begin{aligned}
		L
		\;=\;
		\sum_{i=1}^m\left(\frac{\partial}{\partial{}x_i}\right)^2
		+\sum_{i=1}^n\left(y_i\left(\frac{\partial}{\partial{}y_i}\right)^2+b_i\frac{\partial}{\partial{}y_i}\right)
	\end{aligned}
\end{eqnarray}
on $\mathbb{R}^m\times\mathbb{R}^n_+$.
The case $m=1$, $n=1$ gives
\begin{eqnarray}
	L\;=\;
	\left(\frac{\partial}{\partial{}x_1}\right)^2
	+y_1\left(\frac{\partial}{\partial{}y_1}\right)^2
	+b_1\frac{\partial}{\partial{}y_1}
\end{eqnarray}
which appears to be different from the kind of operator studied in this paper; but after the change of variables $x=x_1$, $y=2\sqrt{y_1}$ we obtain
\begin{eqnarray}
	L\;=\;
	\frac{\partial^2}{\partial{y}^2}+\frac{\partial^2}{\partial{y}^2}
	+\left(2b_1-1\right)\frac1y\frac{\partial}{\partial{y}}
\end{eqnarray}
which is precisely the kind of operator we study, after multiplying through by $y^2$.

Any Keldysh- or Tricomi-type degenerate-elliptic operator of the form
\begin{eqnarray}
	L_k(f)\;=\;f_{tt}\,+\,u^kf_{uu}
\end{eqnarray}
has Euler-type degeneracy after substituting $x=t$, $y=\frac{2}{2-k}u^{\frac{2-k}{2}}$ (making a logarithmic change when $k=2$ actually does {\it not} give Euler-type degeneracy).
We find $b_2=\frac{k}{k-2}$ and one notices the exceptional values $b_2=0$ and $b_2=1$ correspond to $k=0$ and $k=\pm\infty$, and the extraordinary range $b_2\in(0,1)$ corresponds to $k$ negative.
Therefore the $b_2\in(-\infty,0]\cup[1,\infty)$ versus the $b_2\in(0,1)$ cases precisely distinguish, respectively, the Keldysh-type from the Tricomi-type operators.

These coordinate changes easily allow us to prove Theorems \ref{ThmLiouvilleKeldysh} and \ref{ThmLiouvilleKeldysh}.

\begin{proof}[Proof of Theorems \ref{ThmAMKeldysh} and \ref{ThmLiouvilleKeldysh}]
	We consider the operator $L=\left(\frac{\partial}{\partial{t}}\right)^2+u^k\left(\frac{\partial}{\partial{u}}\right)^2$ for $k>2$.
	With $x=t$, $y=\frac{2}{2-k}u^{\frac{2-k}{2}}$ elementary computations give
	$$
	L=\left(\frac{\partial}{\partial{x}}\right)^2+\left(\frac{\partial}{\partial{y}}\right)^2\,+\,\frac{k}{k-2}\frac{1}{y}\frac{\partial}{\partial{y}}
	$$
	and we notice that $b_2=\frac{k}{k-2}>1$.
	The coordinate transformation takes the half-plane to the half-plane, but $u=\infty$ is exchanged with $y=0$ and vice-versa.
	Therefore almost-monotonicity, Theorem \ref{PropAlmostMonoIntro}, precisely states that when $u=0$---which is $y=\infty$---the function $f$ is constant and equals its infimum.
	
	Similarly, the hypotheses for the Liouville theorem, Theorem \ref{ThmLiouvilleActual}, are satisfied after the coordinate transformation.
\end{proof}

\subsection{The Heston model and other financial models} \label{SubSecHestonDescription}

Heston \cite{Hest} extended the Black-Scholes model to the situation where the underlying asset's price volatility is itself a stochastic variable.
Heston's stochastic system is
\begin{eqnarray*}
	\begin{aligned}
		&dS=\mu{}Sdt+\sqrt{v}dW_t^1, \;\;
		dv=\kappa\left(\theta-v\right)dt+\sigma\sqrt{v}dW_t^2, \;\;
		\left<dW_t^1,dW_t^2\right>=\rho\,dt
	\end{aligned}
\end{eqnarray*}
for asset price $S$ and its stochastic volatility $v$ as functions of time, where Weiner processes $dW_t^1$, $dW_t^2$ are correlated by $\rho$, and $\mu,\kappa,\theta,\sigma$ are constants known as the asset drift rate, the volatility mean-reversion rate, average volatility, and the volatility of volatility.
Standard techniques produce a backward heat equation for a European-style option price $U$, given by $U_t\;=\;-L_H(U)$ \cite{Hest} where
\begin{eqnarray}
	\begin{aligned}
		&L_H(U)
		\;=\;
		\frac12vS^2\frac{\partial^2U}{\partial{S}^2}
		+\rho\sigma{}vS\frac{\partial^2U}{\partial{S}\partial{v}}+\frac12\sigma^2v\frac{\partial^2U}{\partial{}v^2} \\
		&\quad\quad\quad
		+rS\frac{\partial{U}}{\partial{S}}
		+\left(\kappa(\theta-v)-\lambda\right)\frac{\partial{U}}{\partial{v}}
		-rU,
		\;\;\text{and} \;\; \kappa,\theta>0,
		\;\rho\in[0,1).
	\end{aligned} \label{EqnHestonOpDefFirst}
\end{eqnarray}
In this model the interest rate $r$ is assumed constant---in the past it was even typical to assume $r\ge0$.
The price of volatility, $\lambda=\lambda(v,S,t)$, is often $0$, at least in simple models.
The operator $L_H$ is the {\it Heston operator}.
On the change of variables
\begin{eqnarray}
	\begin{aligned}
		&x\;=\;\frac{\sqrt{2}}{\sqrt{1-\rho^2}}\log\,S
		\;-\;\frac{\sqrt{2}\rho/\sigma}{\sqrt{1-\rho^2}}v,
		\quad\quad
		y\;=\;\frac{\sqrt{2}}{\sigma}v
	\end{aligned}
\end{eqnarray}
we find
\begin{eqnarray}
	\begin{aligned}
		&L_H(U)
		\;=\;y\left(U_{xx}+U_{yy}\right)
		+\left(b_1+B_1{}y\right)U_x
		+\left(b_2+B_2{}y\right)U_y-rU,
	\end{aligned}\label{EqnHestonModified}
\end{eqnarray}
where $b_1=\frac{\sqrt{2}\sigma}{\sqrt{1-\rho^2}}\left(r\sigma+(\lambda-\kappa\theta)\rho\right)$, $B_1=\frac{\sqrt{2}\sigma}{\sqrt{1-\rho^2}}\left(\rho\kappa-\frac{1}{2}\right),$ $b_2=\frac{\sqrt{2}}{\sigma}\left(\kappa\theta-\lambda\right)$, and $B_2=-\frac{\sqrt{2}\kappa}{\sigma}.$
Multiplying through by $y$ we do indeed see Euler-type degeneracy at the boundary $\{y=0\}$.
We typically have $b_2>0$ but no guarantee that $b_2\ge1$.

Due to the unbounded coefficients about half of this paper does not apply to the Heston equation.
We mention it because helps illustrate the necessity of the assumptions in our theorems.
See Example \ref{ExHestonGradFail} to see some solutions that violate the interior gradient estimate Proposition \ref{PropIntGradFirst} and the Liouville Theorem \ref{ThmLiouvilleFirstStatement}.

A large number of financial models display boundary-degeneracy; examples are the SABR model \cite{HKLW}, the Cox-Ingersoll-Ross process \cite{CIR}, and the Fernholz-Karatzas equation \cite{FK}.
Most of these models can indeed be transformed into equations with Euler-type degeneracy.
For example the SABR model uses the stochastic process
\begin{eqnarray}
	\begin{aligned}
		&dF=aF^\beta{}dW_t^1, \;
		da=\nu{}a\,dW_t^1, \;
		dW_t^1dW_t^2=\rho{}dt
	\end{aligned}
\end{eqnarray}
where $\nu{}>0$ and $\rho,\beta\in[0,1]$.
The options pricing equation is $P_t=-\frac12L(P)$ where
\begin{eqnarray}
	L(P)
	\;=\;a^2\left[F^{2\beta}P_{FF}+2\rho\nu{}F^\beta{}P_{aF}+\nu^2P_{aa} \right]
	\label{OpRoughSABR}
\end{eqnarray}
(Equation (A.10a) of \cite{HKLW}) on the quarter-plane $F,a\in[0,\infty)$.
Changing variables to $w=\nu^{-1}\alpha$, $z=\frac{1}{1-\beta}F^{1-\beta}$ gives
\begin{eqnarray}
	L
	\;=\;
	\nu^2w^2
	\left[
	\left(
	\left(\frac{\partial}{\partial{z}}\right)^2
	+2\rho\frac{\partial}{\partial{z}}\frac{\partial}{\partial{w}}
	+\left(\frac{\partial}{\partial{w}}\right)^2\right)
	\,-\,\left(\frac{\beta}{1-\beta}\right)\frac{1}{z}\frac{\partial}{\partial{}z}
	\right]
	\label{OpEasierSABR}
\end{eqnarray}
on the quarter-plane $w,z\in[0,\infty)$.
This type of operator, where simultaneous scaling in the two variables leaves the operator unchanged, can {\it always} be transformed, up to a homogeneous factor, into an equation of the form (\ref{EqnOpMain}), after making an affine and then a conformal transformation.
For (\ref{OpEasierSABR}) one sets $z'=z$, $w'=\frac{1}{\sqrt{1-\rho}}w-\frac{\rho}{\sqrt{1-\rho}}z$ and then $x+\sqrt{-1}y=(w'+\sqrt{-1}z')^{\pi/\theta}$, where $\theta=-\cot^{-1}\rho(1-\rho^2)^{-1/2}$.
Usually this is too tough for by-hand computation, but in the simplest case, $\rho=0$, we have $\theta=\pi/2$ and $x=w^2-z^2$, $y=2zw$.
The SABR operator is therefore
\begin{eqnarray*}
	\begin{aligned}
		L=
		G(x,y)
		\left[
		y^2\left(\partial_x^2+\partial_y^2\right)
		-y\cdot\frac{\beta/2}{1-\beta}
		\left(\frac{y}{\sqrt{x^2+y^2}}\partial_{x}
		+\frac{x+\sqrt{x^2+y^2}}{\sqrt{x^2+y^2}}\partial_{y}
		 \right)
		\right]
	\end{aligned}
\end{eqnarray*}
on the half-plane $\{y\ge0\}$, which is precisely the kind of equation we study here, up to the homogeneous multiple $G(x,y)=G(x/y)$ whose exact form is unimportant when examining $L(f)=0$.
The coefficients are indeed bounded, in the sense we require.
The sign of $b_2$ however is negative for $\beta\in(0,1)$.

%

\subsection{The Abreu equation} \label{SubSecAbreu}

If an $n$-torus acts on a K\"ahler manifold $(M^{2n},J,\omega)$ isometrically and symplectomorphically, the K\"ahler condition allows us to combine the Arnold-Liouville dimensional reduction from symplectic geometry with Riemannian geometry to produce the attractive theory of toric K\"ahler geometry. See, for example, \cite{Guil}, \cite{Ab1}, \cite{Don1} and references therein.

The Arnold-Liouville construction creates specially adapted coordinates
\begin{eqnarray}
	\begin{aligned}
		\varphi^1,\,\theta_1,\,\dots,\,\varphi^n,\,\theta_n
	\end{aligned}
\end{eqnarray}
on the K\"ahler manifold, known as {\it action-angle} coordinates, or simply {\it symplectic coordinates}, where the ``angle'' fields $\frac{\partial}{\partial\theta_i}$ generate the torus action, and the ``action'' variables $\varphi^i$, which satisfy $\nabla\varphi^i=-J\frac{\partial}{\partial\theta_i}$, parameterize the leaf-space.
The map $\Phi:M^{2n}\rightarrow\mathbb{R}^n$ sending $p$ to $(\varphi^1(p),\dots,\varphi^n(p))$ is called the Arnold-Liouville reduction, or in a slight abuse of terminology, the {\it moment map}.
If $M^{2n}$ is compact then its image under $\Phi$ is a compact polytope $\Sigma^n\subset\mathbb{R}^n$ called its {\it Delzant polytope}.

The Arnold-Liouville reduction $M^{2n}\rightarrow\Sigma^n$ is the expression, in coordinates, of the Riemannian quotient of $M^{2n}$ by the isometric action of the torus.
Thus the reduced manifold-with-boundary $\Sigma^n$ must contain, in some fashion, all of the metric, symplectic, and complex-analytic data present in the original K\"ahler manifold.
Indeed there is a convex function $U:\Sigma^n\rightarrow\mathbb{R}$, called the manifold's {\it symplectic potential}, with $g_{ij}=U_{ij}d\varphi^i\otimes{}d\varphi^j$ on $\Sigma^n$ and $g_{ij}=U_{ij}d\varphi^i\otimes{}d\varphi^j+U^{ij}d\theta_i\otimes{}d\theta_j$ on $M^{2n}$.
Here $U_{ij}$ is the coordinate Hessian $(U_{ij})=\left(\frac{\partial^2U}{\partial\varphi^i\partial\varphi^j}\right)$ with respect to the action coordinates, and $(U^{ij})$ is the inverse matrix of $(U_{ij})$.
If $R$ is the scalar curvature of $M^{2n}$ then its expression on the reduced manifold $\Sigma^n$ is given by the {\it Abreu Equation}, the fully non-linear 4th order elliptic equation
\begin{eqnarray}
	\frac{\partial^2\,U^{ij}}{\partial\varphi^i\partial\varphi^j}\;=\;-2R \label{EqnAbreu}
\end{eqnarray}
from Theorem 4.1 of \cite{Ab1}.
The Abreu equation bears the relationship to the biharmonic equation $\triangle\triangle{}U=2R$ that the Monge-Ampere equation $det(U_{ij})=2R$ bears to the Poisson equation $\triangle{}U=2R$.

The theory of 4th order elliptic equations, in comparison to the 2nd order theory, is a bit threadbare.
For example there is no maximum principle.
But we have the Trudinger-Wang reduction \cite{TW}, and its application to the Abreu equation by Donaldson \cite{Don1}.
Starting from the homogeneous nonlinear equation $U^{ij}{}_{,ij}=0$, Trudinger-Wang introduced an auxiliary linear second order equation, solutions of which provide information on the original 4th order equation.
Donaldson sharpened this construction with a hodographic transformation, and fully reduced $U^{ij}{}_{,ij}=0$ to a pair of second order linear equations.

To explain, we work the 2-dimensional setting.
Under a natural convexity requirement $g=U_{ij}d\varphi^i\otimes{}d\varphi^j$ is a Riemannian metric on a (potentially hypothetical) 2-dimensional manifold.
Then one notices that, assuming $U$ solves the homogeneous Abreu equation, the function $y=\sqrt{det(U_{ij})}$ is {\it harmonic} in this metric---this is the original Trudinger-Wang observation---and therefore has an {\it harmonic conjugate} $x$ obtained by solving $dx=*dy$ where $*:\bigwedge{}^1\rightarrow\bigwedge{}^1$ is the Hodge-$*$ operator of $g$.
Going further, Donaldson noticed that with $(x,y)$ being isothermal coordinates on our (possibly hypothetical) Riemannian 2-manifold, one can compute the coordinate transitions from the $(x,y)$ back to the $(\varphi^1,\varphi^2)$ coordinates, and find
\begin{eqnarray}
	\begin{aligned}
	&y\left(\varphi^1{}_{xx}+\varphi^1{}_{yy}\right)
	-\varphi^1{}_y=0
	\quad \text{and} \\
	&y\left(\varphi^2{}_{xx}+\varphi^2{}_{yy}\right)
	-\varphi^2{}_y=0. \label{EqnsAbreuReduction}
	\end{aligned}
\end{eqnarray}
Notice the roles of the dependent and independent variables have completely switched.
Working backwards, if one can solve the decoupled linear system (\ref{EqnsAbreuReduction}), one can solve the Abreu equation.

The two equations (\ref{EqnsAbreuReduction}) have the Euler-type degeneracy that we study in this paper---after multiplying everything by $y$, that is---but the value of the transport term is wrong: it is not just less than $1$, it is negative.
This is remedied by replacing the functions $\varphi^i$ by $\tilde\varphi^i=y^{-2}\varphi^i$, whereupon we obtain the two equations
\begin{eqnarray}
	y\left(\tilde\varphi^i{}_{xx}+\tilde\varphi^i{}_{yy}\right)
	+3\tilde\varphi^i{}_y=0.
\end{eqnarray}
The theory developed in this paper, particularly the Liouville theorem \ref{ThmLiouvilleFirstStatement}, has strong consequences for the geometry of toric scalar-flat K\"ahler 4-manifolds.

\section{The interior gradient estimates}

Our broadest result is Proposition \ref{PropInteriorInitial}, which states that a complete solution of $L(f)=g$, $f\ge0$ always satisfies $y|\nabla\log{}f|\le{}D$.
This result requires boundedness but no sign constraints on the coefficients, and notably does not require local boundedness of $f$ at $\{y=0\}$.
The method of proof is by a point-picking improvement argument and then a scale/blowup argument.
This style of argument sees frequent use in differential geometry---it was largely popularized in its present form by Perelman---and is made possible here by the fact that the operators $y^2\triangle$ and $y\nabla$ are invariant under simultaneous scaling of $x$ and $y$.
The reliance on coordinate scaling makes the $L^\infty$ bounds on the coefficients, as opposed to, say, $L^p_{loc}$ bounds, completely indispensable.

\begin{proposition}[Interior gradient estimate] \label{PropInteriorInitial}
	Assume that on the open half-plane $H^2$ the functions $b_1,b_2,c,g$ are measurable, bounded $|b_1|,|b_2|,|c|,|g|<\Lambda$, and that $f>0$ satisfies $L(f)=g$ weakly.
	Then a constant $D=D(\Lambda)$ exists so
	\begin{eqnarray}
		y|\nabla\log{}f|\;\le\;D.
	\end{eqnarray}
\end{proposition}
{\bf Remark.}
The method of this proof easily extends to provide uniform bounds of the form $y^k|\nabla^k\log{}f|<D=D(\Lambda,k)$, provided the coefficients also have improved regularity.
But optimal regularity is not a large concern of this paper, and since we shall not need this for any other results we leave this for the future. \\
\indent{\bf Remark.} In ``satisfies $L(f)=g$ weakly'' one may replace ``weakly'' with ``in the distributional sense'' or ``in the viscosity sense.''
The only place where any notion of ``$=$'' in ``$L(f)=g$'' occurs is when we obtain $C^{1,\alpha}$ convergence of solutions $f_i$ of the classical Poisson equation of the form $\triangle{}f_i=h_i$ where the $h_i$ are gradients of measurable functions; this is in the argument just before (\ref{EqnTraingleZero}). \\
\indent{\bf Remark.}
This proposition is definitely false if $|L(f)|\le\Lambda$ is replaced with a one-sided bound, say $L(f)\le0$.
See Example \ref{ExSupers}.

\begin{proof}
	For an argument by contradiction, assume there is no such $D$.
	Then there exists a sequence of operators $L_i$ and functions $b_{1,i},b_{2,i},c_i,g_i$ satisfying the hypotheses, but for which a sequence of functions $f_i>0$ with $L_i(f_i)=g_i$ on $H^2$ exists, along with a sequence of points $p_i=(x_i,y_i)\in{}H^2$, where $y_i|\nabla\log{}f|_{p_i}\rightarrow\infty$.
	Passing to a subsequence if necessary we may assume
	\begin{eqnarray}
		y_i|\nabla\log{}f|_{p_i}>i. \label{EqnGrowthOfGrad}
	\end{eqnarray}
		
	The first step is to execute a ``point-picking'' scheme to improve the choice of the points $p_i$, to wit, nearby $p_i$ we select a better point, where ``better'' means a point with substantially larger value of $|\nabla\log{}f_i|$, should such a point exist.
	Specifically, denoting $p_i$ by $p_{i,1}$, let $p_{i,2}$ be any point in the disk of radius $i/(4|\nabla\log{}f_i|_{p_{i,1}})$ around $p_{i,1}$ that has quadruple the gradient, $|\nabla\log{}f_i|_{p_{i,2}}>4|\nabla\log{}f_i|_{p_{i,1}}$, if any such point exists.
	We remark that from $y_i|\nabla\log{}f|_{p_{i,1}}>i$, the $y$-value of $p_{i,2}$ satisfies
	\begin{eqnarray}
		y(p_{i,2})
		\ge{}y_i-\frac{i}{4|\nabla\log{}f_i|_{p_{i,1}}}
		>y_i\left(1-\frac14\right) \label{EqnNewPointYEst}
	\end{eqnarray}
	implying that the disk of radius $i/(4|\nabla\log{}f_i|_{p_{i,1}})$ around $p_{i,1}$, which is the search space for the better point $p_{i,2}$, remains well within the upper half-plane.
	Due to (\ref{EqnNewPointYEst}) we also have
	\begin{eqnarray}
		y(p_{i,2})|\nabla\log{}f_i|_{p_{i,2}}\;>\;y_i\left(1-\frac14\right)4|\nabla\log{}f_i|_{p_{i,1}}\;\ge\;3i
	\end{eqnarray}
	and so we retain (and even improve) the hypothesis (\ref{EqnGrowthOfGrad}).
	
	But $p_{i,2}$ may still be inadequate.
	Possibly there is a point $p_{i,3}$ in the ball of radius $i/(4|\nabla\log{}f_i|_{p_{i,2}})$ with still larger gradient: $|\nabla\log{}f_i|_{p_{i,3}}>4|\nabla\log{}f_i|_{p_{i,2}}$.
	If such a point $p_{i,3}$ exists, we have
	\begin{eqnarray}
		\begin{aligned}
		y(p_{i,3})
		&\ge{}y_i-\frac{i}{4|\nabla\log{}f_i|_{p_{i,1}}}-\frac{i}{16|\nabla\log{}f_i|_{p_{i,1}}}
		\ge{}y_i\left(1-\frac14-\frac{1}{16}\right)
		\end{aligned}
	\end{eqnarray}
	and consequently also
	\begin{eqnarray}
		\begin{aligned}
		y(p_{i,3})|\nabla\log{}f_i|_{p_{i,3}}
		\;\ge\;y_i\left(1-\frac14-\frac{1}{16}\right)\cdot4^2|\nabla\log{}f_i|_{p_{i,1}}
		\;\ge\;11i.
		\end{aligned}
	\end{eqnarray}
	Continuing this process, we obtain, at the $k^{th}$ iteration, a point
	$p_{i,k}$ with:
	\begin{eqnarray}
		&&|\nabla\log{}f_i|_{p_{i,k}}>4^{k-1}|\nabla\log{}f_i|_{p_{i,1}}, \label{IneqGradDoubling} \\
		&&y(p_{i,k})
		\;\ge\;\left(1-\frac14-\dots-\frac{1}{4^{k-1}}\right)y_i
		\;=\;\frac13\left(2+4^{-k+1}\right)y_i, \label{IneqYRemains}
	\end{eqnarray}
	and consequently also the following improvement on (\ref{EqnGrowthOfGrad})
	\begin{eqnarray}
		&&y(q_{i,k})|\nabla\log{}f_i|_{p_{i,k}}
		\;>\;\frac13\left(2\cdot{}4^{k-1}+1\right)i.
	\end{eqnarray}
	Because of (\ref{IneqYRemains}), we see that the sequence of re-chosen points $p_{i,1},p_{i,2},\dots,p_{i,k}$ remains within the interior of the upper half-plane, and more than that, remains within the closed ball of radius $\frac23y_i$ around the original point $p_i$.
	In particular the choices all remain within the fixed compact set $\overline{B_{p_i}(\frac23y_i)}\subset{}H^2$.

	To prove that this process {\it must} terminate after only finitely many steps, note elliptic regularity excludes the possibility that $|\nabla{}f_i|$ is infinite on any compact set in the open half-plane---even though the coefficients on $L_i(f_i)=g_i$ might be large in the disk $\overline{B_{p_i}(\frac23y_i)}$, no coefficient is ever infinite in this disk.
	Therefore (\ref{IneqGradDoubling}) guarantees that our point-reselection process must terminate at some finite stage.
	Letting $p_{i,k}$ be the terminal point of this process, we replace the old point $p_i$ with the now re-selected point $p_{i,k}$.
	
	Upon reselection of an improved point $p_i$, we still have the same functions $f_i$, operators $L_i$, and functions $g_i$.
	We now have the following conditions:
	\begin{itemize}
		\item[{\it{a}})] $f_i>0$ and weakly satisfies $L_i(f_i)=g_i$ on the open half-plane $H^2$, and $|b_1|,|b_2|,|c|,|g_i|\le\Lambda$.
		\item[{\it{b}})] We have points in the open half-plane $p_i\in{}H^2$ with $y(p_i)|\nabla\log{}f_i|_{p_i}>i$.
		\item[{\it{c}})] In the ball of radius $i/(4|\nabla\log{}f_i|_{p_i})$ around $p_i$, $|\nabla\log{}f_i|\le4|\nabla\log{}f_i|_{p_i}$.
	\end{itemize}
	Item ({\it{c}}) was ensured by the point-picking process; ({\it{a}}) and ({\it{b}}) already held.
	
	With the reselection process done, the second step is to scale the functions $f_i$ and to scale the coordinate system.
	To scale $f_i$, simply multiply it by a constant so $f_i(p_i)=1$; this clearly does not affect conditions ({\it{a}})-({\it{c}}).
	To scale the coordinates system, set $\alpha_i=|\nabla\log{}f_i|_{p_i}$ and for each $i$ create the linear diffeomorphism
	\begin{eqnarray}
		\begin{aligned}
		&\bar{x}\;=\;\alpha_i\left(x-x(p_i)\right) \\
		&\bar{y}\;=\;\alpha_i\left(y-y(p_i)\right)
		\end{aligned}
	\end{eqnarray}
	from the half-plane $\{y>0\}$ in the $(x,y)$ system to the half-plane $\{\bar{y}>-\alpha_iy_i\}$ in the new $(\bar{x},\bar{y})$ system.
	
	The coordinates of $p_i$ in the new system are $(0,0)$, and at this point the new choice of coordinates gives $|\nabla\log{f}_i|_{(0,0)}=1$.
	As measured in the new coordinate system, conditions ({\it{a}}), ({\it{b}}) and ({\it{c}}) now read
	\begin{itemize}
		\item[{\it{a}})${}^\prime$] $f_i>0$ satisfies
		\begin{eqnarray*}
			\begin{aligned}
			&\left(\bar{y}+\alpha_iy(p_i)\right)^2\left(\frac{\partial^2f_i}{\partial\bar{x}^2}+\frac{\partial^2f_i}{\partial\bar{y}^2}\right) \\
			&\quad+\left(\bar{y}+\alpha_iy(p_i)\right)\left(b_1\frac{\partial{}f_i}{\partial\bar{x}}+b_2\frac{\partial{}f_i}{\partial\bar{y}}\right)
			+cf_i
			\;=\;g_i
			\end{aligned}
		\end{eqnarray*}
		 on the half-plane $\{(\bar{x},\bar{y})\,|\,\bar{y}>-\alpha_i{}y(p_i)\}$.
		 The measurable functions $b_1$, $b_2$, $c$, and $g$ are all uniformly bounded by $\Lambda$ on this half-plane.
		\item[{\it{b}})${}^\prime$] At the origin we have $|\nabla\log{}f_i|_{(0,0)}=1$ and $f_i(0,0)=1$
		\item[{\it{c}})${}^\prime$] In the ball of radius $i/4$ around the origin, we have $|\nabla\log{}f_i|\le4$.
	\end{itemize}
	A consequence of ({\it{c}})${}^{\prime}$ is that $f_i$ is bounded from above and below exponentially: $e^{-4\dist(p,o)}\le{}f_i(p)\le{}e^{4\dist(p,o)}$ for all $p$ within the ball of radius $i/4$ about the origin $o$.
	Because $|\bar{y}|\le{}i/4$ in the ball about the origin and because the half-plane $\{\bar{y}>-\alpha_iy(p_i)\}$ contains the half-plane $\{\bar{y}>-i\}$ (as a consequence of ({\it{a}}${}^\prime$)), we have that $\bar{y}+\alpha_iy(p_i)>\frac34i$ on the ball of radius $i/4$.
	Thus within this ball we have the estimate
	\begin{eqnarray}
		\begin{aligned}
		\left|\frac{\partial^2f_i}{\partial\bar{x}^2}
		+\frac{\partial^2f_i}{\partial\bar{y}^2}\right|
		&=
		\left|
		\frac{1}{\bar{y}+\alpha_iy(p_i)}
		\left(b_1\frac{\partial{}f_i}{\partial\bar{x}}+b_2\frac{\partial{}f_i}{\partial\bar{y}}\right)
		+\frac{cf_i-g_i}{\left(\bar{y}+\alpha_iy(p_i)\right)^2}
		\right| \\
		&\le\left(
		\frac{4}{3i}
		\left|b_1\frac{\partial\log{}f_i}{\partial\bar{x}}+b_2\frac{\partial\log{}f_i}{\partial\bar{y}}\right|
		+\frac{16}{9i^2}\left|c\right|\right)f_i
		+\frac{16}{9i^2}g_i \\
		&\le
		\left(\frac{16\Lambda}{3i}+\frac{16\Lambda}{9i^2}\right)f_i
		+\frac{16\Lambda}{i^2}
		\end{aligned} \label{IneqBallElliptic}
	\end{eqnarray}
	where we used $|\nabla\log{}f_i|<4$ and $|b_1|,|b_2|,|c|,|g_i|\le\Lambda$ in the last line.
	On any fixed pre-compact domain $\Omega$ containing the origin in the $(\bar{x},\bar{y})$ system, the value of $f_i$ is bounded by $e^{4\diam\Omega}$.
	
	We conclude that, on any fixed pre-compact domain $\Omega$, we have $|\triangle{}f_i|=O(i^{-1})$, and we also have that $f_i$ is bounded above and below by fixed exponential functions.
	By ({\it{a}})${}^\prime$ the Laplacian $\triangle{}f_i$ is at least measurable and by (\ref{IneqBallElliptic}) it is bounded, so the usual theory implies that $f_i$ has uniform $C^{1,\alpha}$ bounds within $\Omega$.
	Taking the limit as $i\rightarrow\infty$ and passing to a subsequence if necessary, we obtain $C^{1,\alpha}$ convergence $f_i\rightarrow{}f_\infty$ to some function $f_\infty$ that weakly (and therefore strongly) satisfies
	\begin{eqnarray}
		\triangle{}f_\infty\;=\;0. \label{EqnTraingleZero}
	\end{eqnarray}
	Because $\alpha_iy_i>i$, the half-planes $\{\bar{y}>-\alpha_iy_i\}$ converge to the entire plane $\mathbb{R}^2$ as $i\rightarrow\infty$ and so the $C^{1,\alpha}$ convergence $f_i\rightarrow{}f_\infty$ occurs on {\it every} pre-compact set.
	
	Thus $\triangle{}f_\infty=0$ on all of $\mathbb{R}^2$.
	Because the convergence was uniformly $C^{1,\alpha}$ on compact sets and because $f_i(0,0)=1$ and $|\nabla{}f_i|_{(0,0)}=1$, in the limit we retain $f_\infty(0,0)=1$ and $|\nabla{}f_{\infty}|_{(0,0)}=1$.
	
	Finally recall that the classical Liouville theorem states that any non-negative harmonic function on $\mathbb{R}^2$ is constant.
	This contradicts $|\nabla{}f_{\infty}|=1$ at $(0,0)$, and establishes the theorem.
\end{proof}

An immediate consequence of Proposition \ref{PropInteriorInitial} is polynomial bounds on $f$ in the $y$-direction: for any fixed $x_0$ and $0<y_1<y_2$ we have
\begin{eqnarray}
	\left(\frac{y_2}{y_1}\right)^{-D}
	\;\le\;\frac{f(x_0,y_1)}{f(x_0,y_2)}
	\;\le\;\left(\frac{y_2}{y_1}\right)^D. \label{IneqIneqGlobalYGrowth}
\end{eqnarray}
In the $x$-direction Proposition \ref{PropInteriorInitial} provides only exponential bounds: for fixed $y$ and any $x_1$, $x_2$ we have
\begin{eqnarray}
	Exp\left(-\frac{D}{y}|x_2-x_1|\right)
	\;\le\;\frac{f(x_2,y)}{f(x_1,y)}
	\;\le\;Exp\left(\frac{D}{y}|x_2-x_1|\right). \label{IneqGlobalXGrowth}
\end{eqnarray}

\begin{proposition}[The localized gradient estimate; sequential version] \label{PropGradEstSequential}
	There exists a constant $D=D(\Lambda)$ so that the following holds.
	Assume $b_1,b_2,c,g$ are measurable, $|b_1|,|b_2|,|c|,|g|<\Lambda$, and that $f>0$ solves $L(f)=g$ on $\Omega^{Int}$, where $\Omega$ is a neighborhood of a point $p\in\{y=0\}$ (see Definition \ref{DefOpens}).
	Then if $\{p_i\}$ is a sequence of points in $\Omega^{Int}$ converging to $p$, we have
	\begin{eqnarray}
		\lim_{i\rightarrow\infty}y(p_i)|\nabla\log{}f|_{p_i}\;\le\;D.
	\end{eqnarray}
\end{proposition}
{\bf Remark.} Notice we do not require solution be locally bounded at $p$.
\begin{proof}
	By shifting in the $x$-coordinate, we may assume $p$ has coordinates $(0,0)$.
	
	For a proof by contradiction, assume there is a sequence of points $p_i$ along with operators $L_i$ satisfying the hypotheses, and solutions $f_i$ to $L_i(f_i)=g_i$ so that $y(p_i)|\nabla{}f_i|_{p_i}^2\rightarrow\infty$.
	Passing to a subsequence if necessary we may assume both
	\begin{eqnarray}
		y(p_i)|\nabla\log{}f_i|_{p_i}\;>\;i, \quad\text{and}\quad 
		dist(p_i,p)\,<\,1/i. \label{IneqsDefiningSEquenceIneqs}
	\end{eqnarray}
	It might be objected that we must also vary the domain $\Omega$, or else the constant $D$ might depend on the domain of definition $\Omega$.
	But to see that actually $D$ is independent of the domain $\Omega$, notice that the operator $L$ and the expression $y|\nabla\log{}f|$ are invariant under simultaneous rescaling of both coordinates.
	For this reason, given any neighborhood $\Omega$ of $p$, we may simply rescale the coordinates so that $\Omega$ contains, say, the open set $\{(x,y)\,|\,x^2+y^2<1,\,y\ge0\}$.
	
	Following the proof of Proposition \ref{PropInteriorInitial}, the first step is to improve the choice of the points $p_i$.
	For convenience let denote $p_{i}$ by $p_{i,1}$.
	Note that the ball of radius $i/4|\nabla\log{}f_i|_{p_{i,1}}$ around $p_{i,1}$ is still in the upper half-plane, since $i/4|\nabla\log{}f_i|_{p_{i,1}}<y(p_{i,1})/4$.
	Let $p_{i,2}$ be any point in this ball where $|\nabla\log{}f_i|_{p_{i,2}}>4|\nabla\log{}f_i|_{p_{i,1}}$.
	We retain $y(p_{i,2})|\nabla\log{}f_i|_{p_{i,2}}^2>i$ since
	\begin{eqnarray}
		\begin{aligned}
		y(p_{i,2})|\nabla\log{}f_i|_{p_{i,2}}^2
		&\;\ge\;
		\left(y(p_{i,1})-\frac{y(p_{i,1})}{4}\right)\cdot4|\nabla\log{}f_i|_{p_{i,1}} \\
		&\;\ge\;3y(p_{i,1})|\nabla\log{}f_i|_{p_{i,1}}
		\;\ge\;3i
		\end{aligned}
	\end{eqnarray}
	We do not necessarily retain $dist(p,p_{i,2})<1/i$, but we come close: using the fact that $p_{i,2}$ is in the ball of radius $i/4|\nabla\log{}f_i|_{p_{i,1}}<y(p_{i,1})/4$ around $p_{i,1}$ we see
	\begin{eqnarray}
		\begin{aligned}
		dist(p,p_{i,2})&\;\le\;\dist(p,p_{i,1})+\dist(p_i,p_{i,2}) \\
		&\;\le\;\frac1i+\frac14y(p_{i,1})
		\;\le\;\frac1i+\frac{1}{4i}\;=\;\frac54\frac1i
		\end{aligned}
	\end{eqnarray}
	and we also have an estimate on the y-coordinate of $p_{i,2}$:
	\begin{eqnarray}
		y(p_{i,2})
		\;\ge\;y(p_{i,1})-\frac{i}{4|\nabla\log{}f_i|_{p_{i,1}}}
		\;\ge\;y(p_{i,1})-\frac14y(p_{i,1})\;=\;\frac34y(p_{i,1}).
	\end{eqnarray}
	But possibly $p_{i,2}$ can also be improved.
	For an inductive process, assume $p_{i,1},\dots,p_{i,k-1}$ have been chosen in such a way that
	\begin{eqnarray}
		\begin{aligned}
		&|\nabla\log{}f_i|_{p_{i,j}}
		>4^j|\nabla\log{}f_i|_{p_{i,1}}, \\
		&y(p_{i,j})
		\ge\left(1-\sum_{n=1}^j\left(\frac14\right)^n\right)y(p_i), \\ &y(p_{i,j})|\nabla\log{}f_i|_{p_{i,j}}
		>4^j\left(1-\sum_{n=1}^j\left(\frac14\right)^n\right)i, \;\;\text{and} \\
		&dist(p,p_{i,j})
		\;\le\;\frac1i\sum_{n=0}^j\left(\frac14\right)^n.
		\end{aligned} \label{EqnsTheFourConditions}
	\end{eqnarray}
	Now choose the next point $p_{i,k}$ to be any point in the ball about $p_{i,k-1}$ of radius $i/(4|\nabla\log{}f_i|_{p_{i,k-1}})$ that has $|\nabla\log{}f_i|_{p_{i,k}}\ge4|\nabla\log{}f_i|_{p_{i,k-1}}$, should such a point exist.
	Assuming it exists, we verify the four conditions (\ref{EqnsTheFourConditions}).
	The first condition is immediate from the choice of $p_{i,k}$.
	The second condition follows from the choice of $p_{i,k}$ being within the ball of radius $i/4|\nabla\log{}f_i|_{p_{i,k-1}}$, and therefore $y(p_{i,k})>y(p_{i,k-1})-i/4|\nabla\log{}f_i|_{p_{i,k-1}}$, which we estimate to be
	\begin{eqnarray*}
		\begin{array}{rll}
		y(p_{i,k})
		&>y(p_{i,k-1})-\frac{i}{4|\nabla\log{}f_i|_{p_{i,k-1}}} \\
		&\ge
		y(p_{i,k-1})-
		\frac{1}{4}\frac{i}{4^{k-1}|\nabla\log{}f_i|_{p_{i,1}}}
		& \text{(by the inductive hypothesis)} \\
		&\ge
		\left(1-\sum_{n=1}^{k-1}\left(\frac14\right)^n\right)y(p_{i,1})
		-\frac{1}{4^{k}}y(p_{i,1})
		& \text{(by}\;y(p_{i,1})|\nabla\log{}f_i|_{p_{i,1}}>i \text{)}\\
		&=\left(1-\sum_{n=1}^{k}\left(\frac14\right)^n\right)y(p_{i,1}).
		\end{array}
	\end{eqnarray*}
	Verifying the inequalities of (\ref{EqnsTheFourConditions}) for $k$, the third inequality is immediate, and the fourth inequality follows from the fact that $p_{i,k}$ is in the ball of radius $i/4|\nabla\log{}f_i|_{p_{i,1}}<\frac{1}{4^k}y(p_{i,1})\le\frac{1}{4^k}dist(p,p_{i,1})$ and using the inductive hypothesis we have
	\begin{eqnarray}
		\begin{aligned}
		dist(p,p_{i,k})
		&\;\le\;dist(p,p_{i,k-1})+dist(p_{i,k-1},p_i^k) \\
		&\;\le\;
		\frac1i\sum_{n=1}^{k-1}\left(\frac14\right)^n
		+\frac{1}{4^k}dist(p,p_{i,1})
		\;\le\;\frac1i\sum_{n=1}^{k}\left(\frac14\right)^n
		\end{aligned}
	\end{eqnarray}
	as desired.
	
	This process cannot continue indefinitely and must terminate at a finite stage; the reason is that this process finds points $p_{i,k}$ for which the gradient grows unboundedly, even though the points $p_{i,k}$ remain {\it bounded away} from the line $\{y=0\}$---by (\ref{EqnsTheFourConditions}) we certainly have $y(p_{i,k})>\frac23y(p_{i,1})$ so the $y$-values of the $p_{i,k}$ remain bounded away from 0.
	But the operator $L_i$ remains uniformly elliptic away from $\{y=0\}$, and so it is impossible that $\lim_k|\nabla\log{}f|_{p_{i,k}}$ be infinite.
	
	Replace $p_i$ with the newly-reselected terminal point of this point-picking process.
	This means that within the ball of radius $i/4|\nabla\log{}f_i|_{p_i}$ we have bounded gradient: $|\nabla\log{}f_i|<4|\nabla\log{}f_i|_{p_i}$.
	Indeed, these improved points $\{p_i\}$ now satisfy the three conditions:
	\begin{itemize}
		\item[{\it{a}})] $f_i>0$ satisfies $L(f_i)=g_i$ on $\Omega$, where $b_1,b_2,c,g_i$ are measurable and bounded by $\Lambda$
		\item[{\it{b}})] We have points $p_i$ in $\Omega$ with $dist(p,p_i)\le\frac43i^{-1}$ and $y(p_i)|\nabla\log{}f_i|_{p_i}\ge{}i$
		\item[{\it{c}})] On the ball of radius $i/(4|\nabla\log{}f_i|_{p_i})$ about $p_i$ we have the gradient bound $|\nabla\log{}f_i|<4|\nabla\log{}f_i|_{p_i}$
	\end{itemize}
	Now we scale the coordinate system: set $\alpha_i=|\nabla\log{}f_i|_{p_i}$ and note that $\alpha_iy(p_i)\ge{}i$.
	For each $i$ create coordinates
	\begin{eqnarray}
		\begin{aligned}
		&\bar{x}\;=\;\alpha_i(x-x(p_i)) \\
		&\bar{y}\;=\;\alpha_i(y-y(p_i)).
		\end{aligned}
	\end{eqnarray}
	Then the point $p_i$ has coordinates $(0,0)$ and after scaling $f_i$ so $f_i(0,0)=1$ and transforming the coordinate system, we have $|\nabla\log{}f_i|_{(0,0)}=|\nabla{}f_i|_{(0,0)}=1$.
	In the new coordinates condition ({\it{c}}) becomes the condition that $|\nabla\log{}f_i|<4$ on the ball of radius $i/4$.
	In fact conditions ({\it{a}})-({\it{c}}) now read
	\begin{itemize}
		\item[{\it{a}})${}^\prime$] $f_i>0$ satisfies
		\begin{eqnarray*}
			\begin{aligned}
				&\left(\bar{y}+\alpha_iy(p_i)\right)^2\left(\frac{\partial^2f_i}{\partial\bar{x}^2}+\frac{\partial^2f_i}{\partial\bar{y}^2}\right) \\
				&\quad+\left(\bar{y}+\alpha_iy(p_i)\right)\left(b_1\frac{\partial{}f_i}{\partial\bar{x}}+b_2\frac{\partial{}f_i}{\partial\bar{y}}\right)
				+cf_i
				\;=\;g_i
			\end{aligned}
		\end{eqnarray*}
		on the half-plane $\{(\bar{x},\bar{y})|\bar{y}>-\alpha_i{}y(p_i)\}$, where $b_1,b_2,c,g_i$ are measurable and bounded by $\Lambda$
		\item[{\it{b}})${}^\prime$] At the origin, $|\nabla\log{}f_i|_{(0,0)}=1$ and $f_i(0,0)=1$
		\item[{\it{c}})${}^\prime$] On the ball of radius $i/4$ around the origin we have $|\nabla\log{}f_i|<4$.
	\end{itemize}
	Condition ({\it{c}})${}^\prime$, the statement of uniform $C^1$ bounds on large balls, gives in particular the pointwise bounds $e^{-4dist(p,o)}<f_i(p)<e^{4dist(p,o)}$.
	From the equation $L_i(f_i)=g_i$, we estimate within the ball of radius $i/4$ that
	\begin{eqnarray}
		\begin{aligned}
		\left|\frac{\partial^2f_i}{\partial\bar{x}^2}
		+\frac{\partial^2f_i}{\partial\bar{y}^2}\right|
		&=
		\left|
		\frac{1}{\bar{y}+\alpha_iy(p_i)}
		\left(b_{i,1}\frac{\partial{}f_i}{\partial\bar{x}}+b_{2,i}\frac{\partial{}f_i}{\partial\bar{y}}\right)
		+\frac{c_if_i-g_i}{\left(\bar{y}+\alpha_iy(p_i)\right)^2}
		\right| \\
		&\le\left(
		\frac{4}{3i}
		\left|b_{i,1}\frac{\partial\log{}f_i}{\partial\bar{x}}+b_{i,2}\frac{\partial\log{}f_i}{\partial\bar{y}}\right|
		+\frac{16}{9i^2}\left|c_i\right|\right)f_i
		+\frac{16}{9i^2}g_i \\
		&\le
		\left(\frac{16\Lambda}{3i}+\frac{16\Lambda}{i^2}\right)f_i
		+\frac{16\Lambda}{i^2}
		\end{aligned}
	\end{eqnarray}
	where we used that $|\nabla\log{}f|<4$ and $\alpha_iy(p_i)>i$, as noted above.
	On any pre-compact domain $\Omega$ that contains the origin, we have the bound $f_i<e^{4diam(\Omega)}$.
	
	We conclude that, on any fixed compact set in the $(\bar{x},\bar{y})$-plane, we see that $|\triangle{}f_i|\searrow0$.
	We therefore obtain $C^{1,\alpha}$ convergence of $f_i$ to some limiting function $f_\infty$ that exists on the entire $(\bar{x},\bar{y})$-plane.
	
	Because the convergence is $C^{1,\alpha}$ and because $f_i>0$, $f_i(0,0)=1$, $|\nabla{}f_i|_{(0,0)}=1$, we obtain in the limit an entire function $f_\infty$ with
	\begin{eqnarray}
		\triangle{}f_\infty=0, \quad f_\infty>0,
		\quad f_\infty(0,0)=1, \quad |\nabla{}f_\infty|_{(0,0)}=1.
	\end{eqnarray}
	But with $f_\infty$ harmonic and non-negative on $\mathbb{R}^2$, the classical Liouville theorem says $f_\infty$ is constant.
	This contradiction establishes the proposition.
\end{proof}

\begin{proposition}[The localized gradient estimate; domain version, {\it cf.} Proposition \ref{PropFirstLocalGradEst}] \label{PropLocDomVer}
	Assume $|b_1|,|b_2|,|c|\le\Lambda$.
	There exists a constant $D=D(\Lambda)$ so that the following holds.
	
	Assume $p\in\{y=0\}$, $\Omega$ is a neighborhood of $p$, $f\ge0$, and $|L(f)|\le\Lambda$ on $\Omega^{Int}$.
	Then there exists some neighborhood $\Omega'\subset\Omega$ of $p$ so that
	\begin{eqnarray}
		y|\nabla\log{}f|\;\le\;D
	\end{eqnarray}
	on $\Omega'{}^{Int}$.
\end{proposition}
\begin{proof}
	If not, then for every neighborhood $B_i=B_p(1/i)\cap\{y\ge0\}$ we can pick a non-negative solution $f_i$ on $\Omega^{Int}$ and a point $p_i\in{}B_i$ so that $y(p_i)|\nabla\log{}f_i|_{p_i}\rightarrow\infty$.
	This contradicts Proposition \ref{PropGradEstSequential}.	
\end{proof}
{\bf Remark.} The sub-neighborhood $\Omega'\subset\Omega$ depends on the original domain $\Omega$ and the value of $\Lambda$, but does not depend on the function $f$---this is because in the proof we allow the function $f_i$ to vary as the point $p_i$ approaches the limit.
This type of uniformity is necessary in the proof of continuity, Proposition \ref{ThmContinuity} below.

\section{Local theorems at the boundary, and almost-monotonicity}

Most of results of this section stem from use of the lower barrier constructed in Lemma \ref{LemmaISubfunction}.
Depicted in Figure \ref{SubFigBarrierWithLines}, this barrier forces uniform positive bounds on $f(x,0)$ at the degenerate boundary $\{y=0\}$, assuming only uniform positive bounds on $f$ at some point $(x,y)$ of the interior.

\subsection{Unspecifiability at Interior Boundary Points} \label{SubSecUnspecifiability}

The behavior of solutions $f$ of $L(f)=g$ at the degenerate boundary $\{y=0\}$ is a central issue.
Under the assumption that $f$ is locally finite and $b_2\ge1$ we show that assigning boundary values at the degenerate boundary $\{y=0\}$ is impossible.

For more on the issue of specifiability and non-specifiability, see Example \ref{ExImps} where, in the constant-coefficient case, we obtain explicit expressions for Kernels in the case $b_2<1$.
These Kernels allow specification of boundary values at $\{y=0\}$.

\begin{proposition}[Unspecifiability of boundary values at $\{y=0\}$] \label{PropUnspecifiability}
	Assume the functions $b_1,b_2,c,g$ are measurable and locally finite.	
	Let $p=(x_0,0)$ be a point on the boundary line $\{y=0\}$ and assume $b_2\ge1$ and $c\le\inf\frac14(b_2-1)^2$ in some pre-compact neighborhood $\Omega$ of $p$.
	Assume two functions $f_1,f_2$ satisfy $L(f_i)=g$ on $\Omega$ in the sense of Definition \ref{DefSolsAtDegBound} and that $f_1=f_2$ on $\partial_1\Omega$ (see Definition \ref{DefBoundComps}).
	Then $f_1=f_2$.
\end{proposition}
\begin{proof}
	The function $f=f_1-f_2$ is continuous on $\overline{\Omega}$, and zero on $\partial\Omega\setminus\{y=0\}$.
	
	Set $\lambda=\inf{}b_2-1$ and consider the function $\psi_\epsilon=\epsilon{}y^{-\lambda/2}$, or $\psi_\epsilon=\epsilon(-\log(y)+\log(\diam\Omega))$ when $\inf{}b_2=1$.
	For $\lambda>0$ we compute
	\begin{eqnarray}
		\begin{aligned}
			L(\psi_\epsilon)
			&\;=\;y^2\left(\psi_\epsilon\right)_{yy}
			+yb_2(\psi_\epsilon)_y
			+c\psi_\epsilon \\
			&\;=\;\epsilon\left(\frac{\lambda^2}{4}+\frac{\lambda}{2}-b_2\frac{\lambda}{2}+c \right)y^{-\lambda/2} \\
			&\;\le\;\epsilon\left(\frac{\lambda^2}{4}-\frac{\lambda^2}{2}+c \right)y^{-\lambda/2}
			\;\le\;0
		\end{aligned}
	\end{eqnarray}
	and similarly for the case $\lambda=0$.

	Next we show that $\psi_\epsilon$ dominates $f$ as long as $\epsilon>0$.
	Note that, by the boundedness of $f$, if $\epsilon$ is very large then certainly $\psi_\epsilon<f$.
	Then we may lower the value of $\epsilon$ until we find the first $\epsilon>0$ with $\psi_\epsilon(x,y)=f(x,y)$ at some point $(x,y)\in\overline{\Omega}$.
	By boundedness of $f$ certainly $y>0$ at this point, and by the fact that $f=0$ but $\psi_\epsilon>0$ on $\partial\Omega\setminus\{y=0\}$ certainly also $(x,y)\notin\partial\Omega\setminus\{y=0\}$ because $f$ is zero there.
	
	Thus we have an interior point $(x,y)\in\Omega^{Int}$ at which $f=\psi_\epsilon$, even though $f\ge\psi_\epsilon$ and $\psi_\epsilon$ is a supersolution.
	But at all points of $\Omega^{Int}$ the operator $L$ is uniformly elliptic, and so we have a contradiction with the maximum principle.

	This contradiction forces $f<\psi_\epsilon$ for all positive $\epsilon$, so $f\le0$.
	Replacing $f$ with $-f$ we see also $f\ge0$.
	Thus $f\equiv0$ and so $f_1\equiv{}f_2$.
\end{proof}

{\bf Remark}.
Definition \ref{DefSolsAtDegBound} demands $f$ only be $L_{loc}^\infty$ near degenerate boundary points.
This modest demand is crucial; see Example \ref{ExImps} for a counterexample when local finiteness fails at a single point of the degenerate boundary.

\subsection{The lower barrier. The Harnack inequality at the boundary}

We first produce a function $\psi_{y_0}$ that has compact support in the strip $y\in[0,y_0]$, which solves $L(\psi_{y_0})\ge0$, and which has $\psi_{y_0}(0,0)=1/9$.

What gives this subfunction so much power is that for it to be a lower barrier we must only check that $\psi_{y_0}\le{}f$ {\it on the line $\{y=y_0\}$}.
Once this is done we retain $\psi_{y_0}\le{}f$ on the whole strip $y\in[0,y_0]$, and most particularly at the line $\{y=0\}$ itself, where $\psi_{y_0}(0,y_0)=C_\Lambda$.

\begin{lemma}[The lower barrier $\psi_{y_0}$] \label{LemmaISubfunction}
	Assume $c\ge0$, $b_2\ge1$, and $|b_2|<\Lambda$.
	Choose any number $y_0>0$, and let $\mathcal{R}_{y_0}$ be the rectangle
	\begin{eqnarray}
	\mathcal{R}_{y_0}=\left\{(x,y)\;\big|\;y\in[0,y_0],\,x\in[-4y_0\Lambda,4y_0\Lambda]\right\}.
	\end{eqnarray}
	Define the function $\psi_{y_0}=\psi_{y_0}(x,y)$ to be
	\begin{eqnarray}
		\begin{aligned}
		\psi_{y_0}
		=
		\frac{10}{9}e^{-\frac{\pi}{8y_0\Lambda}y}\,
		{}_1F_1\left(\frac{1+\sqrt{3}\Lambda}{2};\,1;\,\frac{\pi}{4y_0\Lambda}y\right)
		\cos\left(\frac{\pi}{8y_0\Lambda}x\right)
		-1
		\end{aligned}
	\end{eqnarray}
	on $\mathcal{R}_{y_0}$, where the ${}_1F_1$ is confluent hypergeometric function of the first kind (as depicted in Figure \ref{FigBarrierOrig}).
	Then $\psi_{y_0}$ has the following properties:
	\begin{itemize}
		\item[{\it{i}})] $L(\psi_{y_0})\ge0$ at all points $(x,y)\in\mathcal{R}_{y_0}$ for which $\psi_{y_0}(x,y)\ge0$.
		\item[{\it{ii}})] On the edge $y=y_0$, $x\in[-4y_0\Lambda,4y_0\Lambda]$ we have $\psi_{y_0}\le2\cos\left(\frac{\pi}{8y_0\Lambda}x\right)-1$ and in particular $\psi_{y_0}\le1$
		\item[{\it{iii}})] On the edges $x=\pm4y_0\Lambda$, $y\in[0,y_0]$ we have $\psi_{y_0}\le0$
		\item[{\it{iv}})] On the edge $y=0$, $x\in[-4y_0\Lambda,4y_0\Lambda]$, $\psi_{y_0}(x,0)=\frac{10}{9}\cos\left(\frac{\pi}{8y_0\Lambda}x\right)-1$
		\item[{\it{v}})] At $(0,0)$, in particular, we have $\psi_{y_0}=\frac19$.
	\end{itemize}
	Finally $\psi_{y_0}$ is a subfunction on $\mathcal{R}_{y_0}$ provided it is a subfunction on the edge $y=y_0$, $x\in[-4y_0\Lambda,4y_0\Lambda]$.
	Specifically, assuming $f\ge0$ and $L(f)\le0$ on $\mathcal{R}_{y_0}$, then if $f\ge\psi_{y_0}$ on the edge $y=y_0$, $x\in[-4y_0\Lambda,4y_0\Lambda]$, we have $f\ge\psi_{y_0}$ on the entire closed set $\mathcal{R}_{y_0}$.
\end{lemma}
\noindent\begin{figure}[h!]
	\vspace{-0.5in}
	\hspace{-1.9in}
	\begin{minipage}[c]{0.7\textwidth}
		\caption{
			\it The barrier $\psi_{y_0}$ and its box $\mathcal{R}_{y_0}$.
			Along the line $\{y=y_0\}$ the barrier is approximately $2\cos(\frac{\pi}{8y_0\Lambda}x)-1$.
			At the origin we have $\psi_{y_0}(0,0)=\frac19$.
			\label{FigBarrierOrig}
		} 
	\end{minipage} \hspace{-0.5in}
	\begin{minipage}[c]{0.2\textwidth}
		\includegraphics[scale=0.5]{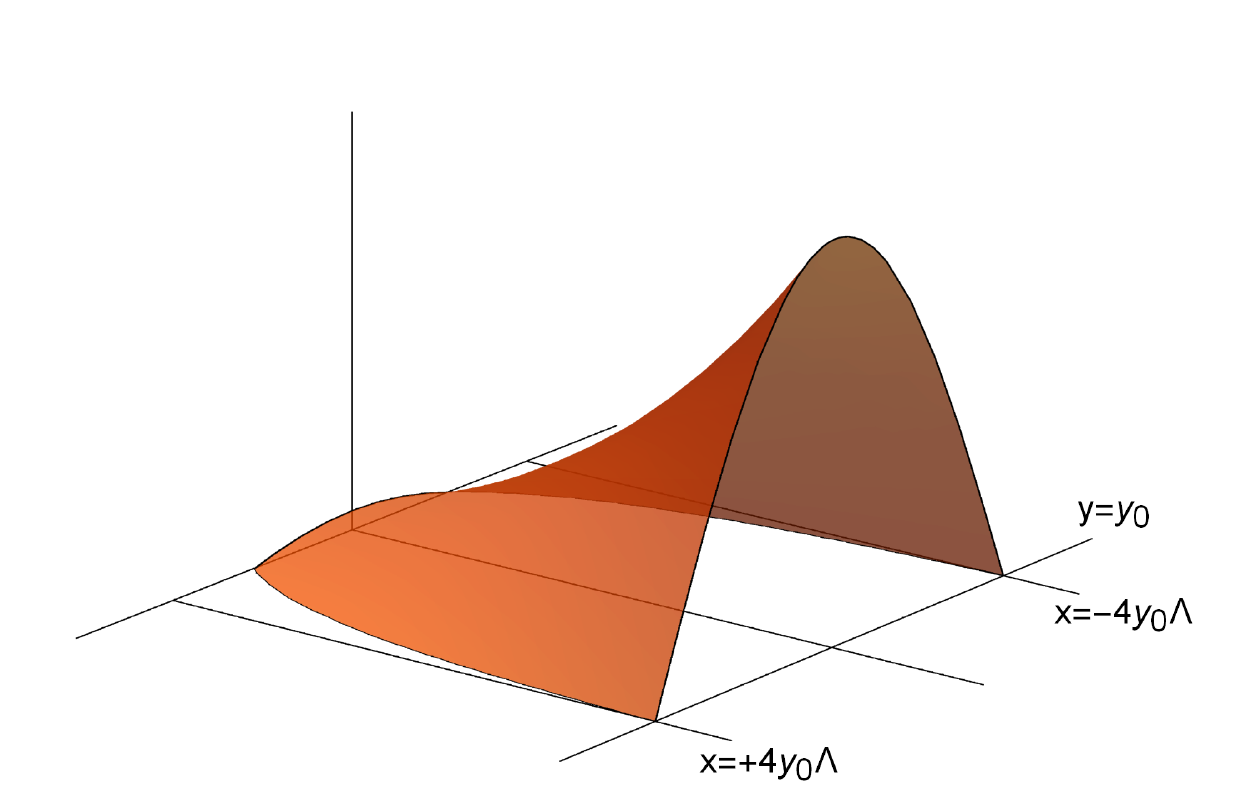}
	\end{minipage} \hfill
\end{figure}
\begin{proof}
	The claims ({\it{ii}})-({\it{v}}) require just elementary verification, perhaps with some electronic help.
	The two non-trivial claims are ({\it{i}}) that $L(\psi_{y_0})\ge0$, and the final claim that $\psi_{y_0}\le{}f$ on $\mathcal{R}_{y_0}$ whenever $\psi_{y_0}\le{}f$ on the line segment $\{y=y_0,\,x\in[-4y_0\Lambda,4y_0\Lambda]\}$.	
	
	To begin the verification of ({\it{i}}), we multiply $\psi_{y_0}$ by $\frac{9}{10}$ for convenience and consider $\frac{9}{10}\psi_{y_0}(x,y)=g(y)h(x)-\frac{9}{10}$ where
	\begin{eqnarray}
		\begin{aligned}
		&h(x)=\cos\left(\frac{\pi{}x}{8y_0\Lambda}\right) \\
		&g(y)=Exp\left(-\frac{\pi}{8y_0\Lambda}y\right)
		\,{}_1F_1\left(\frac{1+\sqrt{3}\Lambda}{2};\,1;\,\frac{\pi}{4y_0\Lambda}y\right).
		\end{aligned}
	\end{eqnarray}
	We have chosen $g(y)$ to satisfy the ODE
	\begin{eqnarray}
		y^2g_{yy}+yg_y-\left(
		\left(\frac{\pi{}y}{8y_0\Lambda}\right)^2
		+\left(\frac{\pi{}y}{8y_0}\right)\sqrt{3}
		 \right)g\;=\;0.
	\end{eqnarray}
	Evaluating $L(hg-9/10)$ we obtain
	\begin{eqnarray*}
		\begin{aligned}
			&L(\psi_{y_0})
			\;=\;y^2\left(g_{yy}h+gh_{xx}\right)
			+y(b_1gh_x+b_2g_yh)
			+c(gh-9/10) \\
			&\quad\;\ge\;\cos\left(\frac{\pi}{8y_0\Lambda}\right)
			\left[
			y^2g_{yy}+yg_y
			-\left(
			\frac{1}{\Lambda^2}\left(\frac{\pi{}y}{8y_0}\right)^2
			+\left(\frac{\pi{}y}{8y_0}\right)
			\tan\left(\frac{\pi{}x}{8y_0\Lambda}\right)
			\right)g
			\right]
		\end{aligned}
	\end{eqnarray*}
	on the domain where $h(x)g(y)-9/10\ge0$; we used $|b_1|<\Lambda$, $b_2>1$, and $c\ge0$.
	
	The function $\tan(\pi{}x/8y_0\Lambda)$ is unbounded if its domain is unrestricted.
	But on the restricted domain $h(x)g(y)-9/10\ge0$ we can bound it.
	We use the fact that $g$ is increasing on $[0,y_0]$ and $h(0)=1$ to obtain
	\begin{eqnarray}
		h(x)
		\;\ge\;\frac{9}{10}\frac{1}{g(y)}
		\;\ge\;\frac{9}{10}\frac{1}{g(y_0)}.
	\end{eqnarray}
	To estimate $g(y_0)$, one computes $g(y_0)=Exp(-\pi/8\Lambda){}_1F_1((1+\sqrt{3}\Lambda)/2;1;\pi/4\Lambda)$ is actually decreasing as a function of $\Lambda$, and decreases to a value of about $1.804$.
	Estimating
	\begin{eqnarray}
		h(x)
		\;\ge\;\frac{9}{10}\frac{1}{1.8}\;=\;0.5
	\end{eqnarray}
	we see that, on the region where $h(x)g(y)-9/10>0$, $y\in[0,y_0]$ then $h(x)=\cos(\pi{}x/8y_0\Lambda)\ge0.5$.
	This gives $|\tan(\pi{}x/8y_0\Lambda)|<\sqrt{3}$ which gives, on this region,
	\begin{eqnarray*}
		\begin{aligned}
			&L(\psi_{y_0})
			\;\ge\;\cos\left(\frac{\pi}{8y_0\Lambda}\right)
			\left[
			y^2g_{yy}+yg_y
			+\left(
			-\frac{1}{\Lambda^2}\left(\frac{\pi{}y}{8y_0}\right)^2
			-\left(\frac{\pi{}y}{8y_0}\right)\sqrt{3}
			\right)g
			\right].
		\end{aligned}
	\end{eqnarray*}
	We conclude that $L(\psi_{y_0})\ge0$ on the subset of $\mathcal{R}_{y_0}$ where $y\in[0,y_0]$ and $\psi_{y_0}\ge0$.
	
	To verify the final claim---that $\psi_{y_0}$ is a subfunction on $\mathcal{R}_{y_0}$ provided it is a subfunction on the segment $y=y_0$, $x\in[-y_0L,y_0L]$, we must show that, when $f\ge0$ solves $L(f)\le0$ on $\mathcal{R}_{y_0}$, then if we have $f(x,y_0)<\psi_{y_0}(x,y_0)$ on the segment $y=y_0$, $x\in[-y_0L,y_0L]$ we have $f<\psi_{y_0}$ on $\mathcal{R}_{y_0}$.
	
	To see this, for any $\epsilon>0$, consider
	\begin{eqnarray}
		\psi_{y_0}(x,y)\,+\,\epsilon{}\log(y/y_0).
	\end{eqnarray}
	This is also a subfunction for all those values where it is non-negative, for
	\begin{eqnarray}
		\begin{aligned}
		&L(\psi_{y_0}(x,y)\,+\,\epsilon{}\log(y/y_0)) \\
		&\;=\;y^2\triangle\psi_{y_0}+y\left(b_1\psi_{y_0,x}+b_2\psi_{y_0,y}\right) \\
		&\quad\quad
		+\epsilon\left(y^2(\partial_y)^2\log(y/y_0)
		+yb_1\partial_y\log(y/y_0)\right) \\
		&\quad\quad
		+c\left(\psi_{y_0}(x,y)\,+\,\epsilon{}\log(y/y_0)\right) \\
		&
		\;=\;b_1-1\;\ge0.
		\end{aligned}
	\end{eqnarray}
	The term $y^2\triangle\psi_{y_0}+y\left(b_1\psi_{y_0,x}+b_2\psi_{y_0,y}\right)$ is non-negative by all the work above.
	The term $\left(y^2(\partial_y)^2\log(y/y_0)
	+yb_1\partial_y\log(y/y_0)\right)$ equals $b_2-1$.
	The term $c\left(\psi_{y_0}(x,y)\,+\,\epsilon{}\log(y/y_0)\right)$ is non-negative by assumption.
	Therefore, as claimed,
	\begin{eqnarray}
		\begin{aligned}
		&L(\psi_{y_0}(x,y)\,+\,\epsilon{}\log(y/y_0))\;\ge\;b_2-1\;\ge0.
		\end{aligned}
	\end{eqnarray}
	Because $\log(y/y_0)\searrow-\infty$ near the boundary $\{y=0\}$, and because $L$ is uniformly elliptic away from $\{y=0\}$ the maximum principle applied in the interior of $\mathcal{R}_{y_0}$ guarantees that $\psi_{y_0}(x,y)+\epsilon\log(y/y_0)<f(x,y)$.
	Sending $\epsilon\searrow0$ gives the result.	
\end{proof}

\begin{theorem}[The Harnack inequality at the boundary; local version] \label{ThmBoundHarnackLocal}
	Assume $|b_1|,|b_2|,|c|<\Lambda$, and $b_2\ge1$, $c\ge0$.
	Assume $f\ge0$ satisfies $L(f)\le0$ and $L(f)\ge-\Lambda$ on the semi-closed rectangle
	\begin{eqnarray}
		\mathcal{R}_{y_0}\;=\;\left\{
		\,(x,y)\;\;\big|\;\;x\in[-4\Lambda{}y_0,4\Lambda{}y_0],\;y\in(0,y_0]
		\right\}.
	\end{eqnarray}
	
	Then $f(0,0)\;\ge\;\frac19\inf_{x\in[-4\Lambda{}y_0,4\Lambda{}y_0]}f(x,y_0)$.
\end{theorem}
\begin{proof}
	This is immediate from Lemma \ref{LemmaISubfunction}, after noticing that 
	\begin{eqnarray}
	\psi_{y_0}(x,y_0)\,\cdot\,
	\inf{}_{x'\in[-4\Lambda{}y_0,4\Lambda{}y_0]}f(x',y_0)
	\;\le\;f(x,y_0)
	\end{eqnarray}
	for all $x\in[-4\Lambda{}y_0,4\Lambda{}y_0]$.
\end{proof}

\begin{theorem}[The Harnack inequality; global version] \label{ThmBoundHarnackGlobal}
	There exists a number $\delta=\delta(\Lambda)>0$ so that the following holds.
	Assume $|b_1|,|b_2|,|c|<\Lambda$, and $b_2\ge1$, $c\ge0$, and assume $f\ge0$ satisifies $L(f)\le0$ and $L(f)\ge-\Lambda$ on the open half-plane.
	
	Then $f(x_0,y')\;\ge\;\delta\cdot{}f(x_0,y)$ whenever $(x_0,y_0)$ is a point in the open half-plane and $y'\in[0,y]$.
\end{theorem}
\begin{proof}
	After shifting the $x$-coordinate if necessary, we may assume that $x_0=0$.
	From the inequality (\ref{IneqGlobalXGrowth}) we have the exponential bounds $f(x,y_0)\ge{}f(0,y)e^{-D|x|/y_0}$ where $D=D(\Lambda)$.
	In particular along the segment $\{y=y_0,x\in[-4y_0\Lambda,4y_0\Lambda]\}$ we have $f(x,y_0)>f(0,y_0)e^{-D\cdot4\Lambda}$.
	Now we may apply Theorem \ref{ThmBoundHarnackLocal}, by using
	\begin{eqnarray}
		\inf{}_{x\in[-4y_0\Lambda,4y_0\Lambda]}f(x,y_0)
		\;\ge\;e^{-D\cdot4\Lambda}f(0,y_0).
	\end{eqnarray}
\end{proof}

\noindent\begin{figure}[h!] 
	\centering
	\noindent\begin{subfigure}[b]{0.4\textwidth} 
		\includegraphics[scale=0.5]{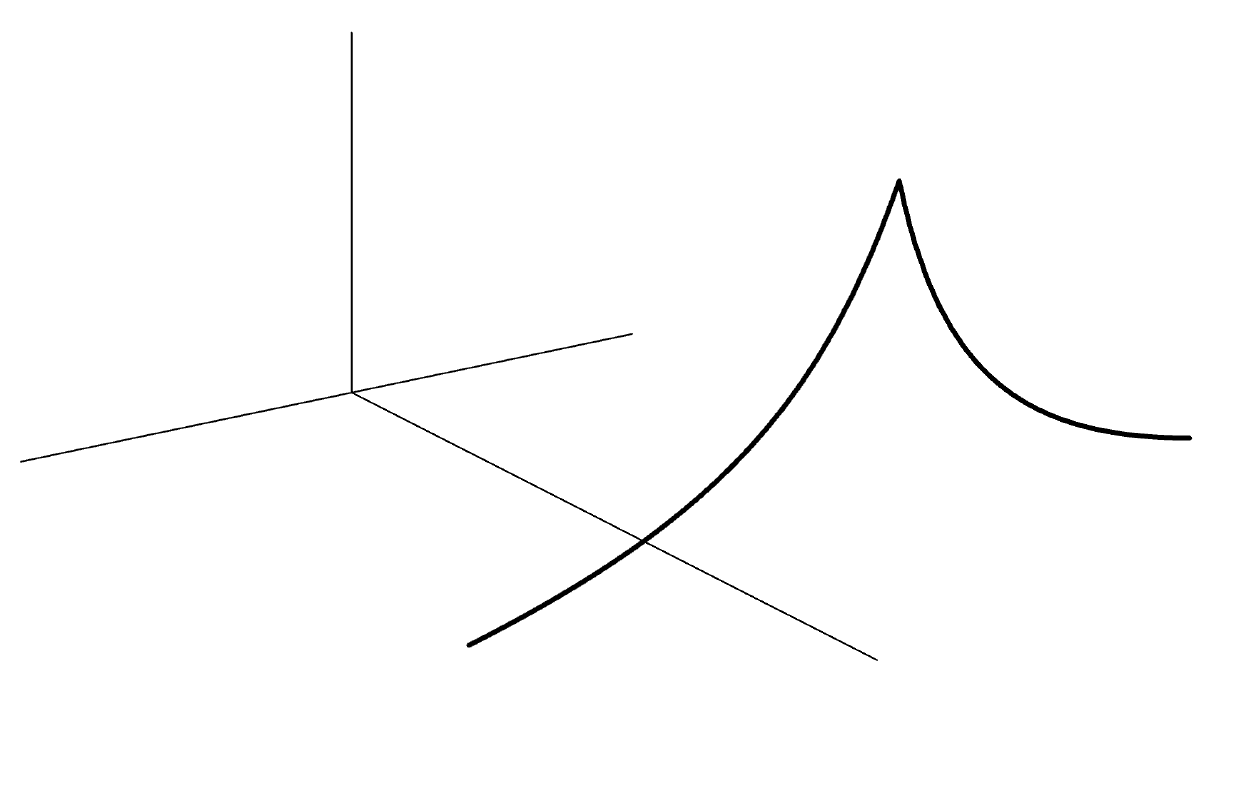}
		\caption{Curve depicting the exponential lower bounds on a solution $f$, due to Proposition \ref{PropInteriorInitial}, along $\{y=y_0\}$. \\
			\\
		\label{SubFigBarrierLines}}
	\end{subfigure}
	\hspace{0.5in}
	\begin{subfigure}[b]{0.4\textwidth} 
		\includegraphics[scale=0.5]{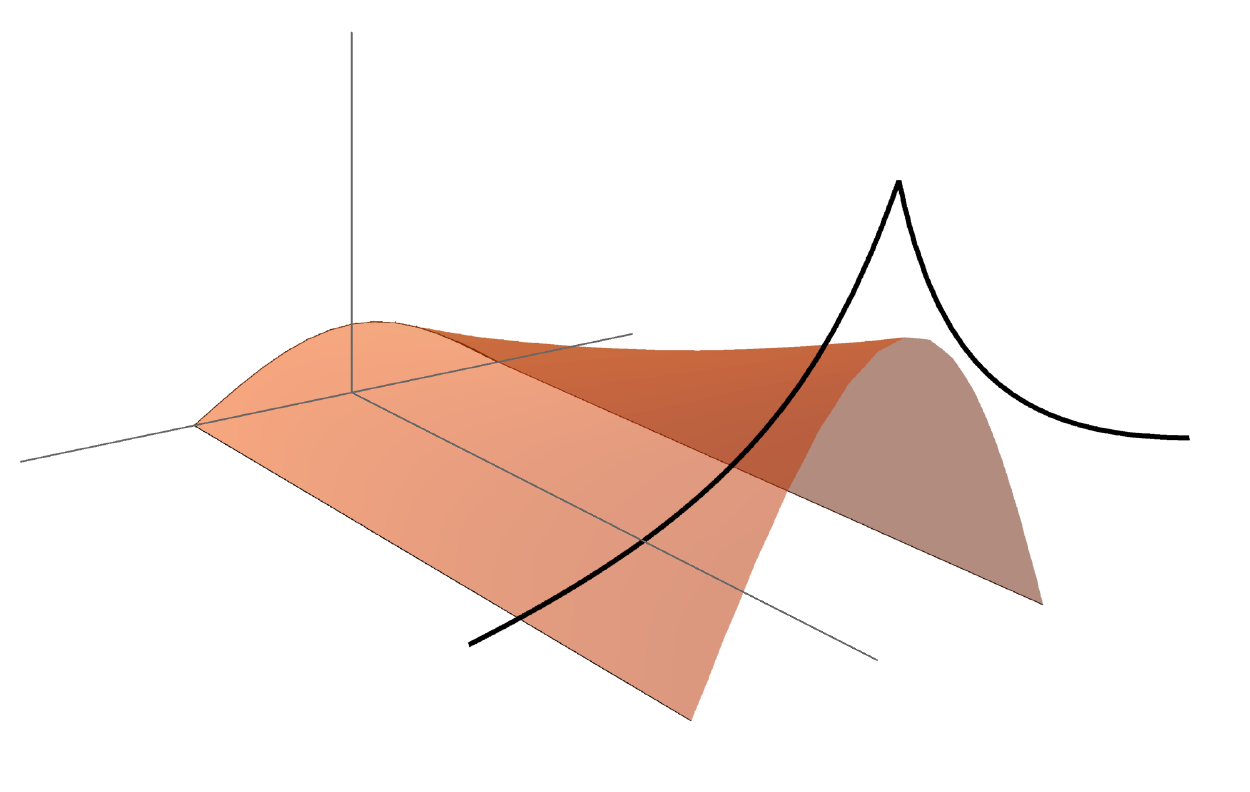}
		\caption{Schematic of the barrier of Lemma \ref{LemmaISubfunction} as it is used in Theorem \ref{ThmBoundHarnackGlobal}.
			It fits beneath the exponential lower bounds at $\{y=y_0\}$, forcing new lower bounds at $\{y=0\}$.
			\label{SubFigBarrierWithLines}}
	\end{subfigure}
	\caption{\it The barrier of Lemma \ref{LemmaISubfunction} and its use in Theorem \ref{ThmBoundHarnackGlobal}. }\label{FigSubFunction}
\end{figure}

\subsection{Continuity at the degenerate boundary, and the Maximum principle.}

The two results of this subsection constrain the behavior of solutions $L(f)=0$ at opposite ends: at $y=\infty$ we prove the function $x\mapsto{}\lim_{y\rightarrow\infty}f(x,y)$ is well-defined and is actually a constant, and at $y=0$ we prove the function $x\mapsto{}f(x,0)$ is well-defined and is actually continuous.
But both results follow from the judicious use of the subfunction $\psi_{y_0}$, developed in Lemma \ref{LemmaISubfunction} and depicted in Figure \ref{FigSubFunction}.

\begin{theorem}[Continuity at the degenerate boundary] \label{ThmContinuity}
	Assume $|b_1|,|b_2|,|c|<\Lambda$, and $b_2\ge1$, $c\ge0$.
	Let $p\in\{y=0\}$ and let $\Omega$ be a neighborhood of $p$ (see Definition \ref{DefOpens}).
	If $L(f)=0$ on $\Omega$ (Definition \ref{DefSolsAtDegBound}), then $f$ is continuous at $p$.
	
	In particular if $f\ge0$ solves $L(f)=0$ on the closed half-plane, then in fact $f\in{}C^0(\overline{H}{}^2)\cap{}C^{1,\alpha}(H^2)$.	
\end{theorem}
{\bf Remark}. As laid out in Definitions \ref{DefOpens} and \ref{DefSolsAtDegBound}, the only assumption on the solution $f$ is local boundedness.
Continuity fails if this assumption is weakened at even a single boundary point, as Example \ref{ExImps} demonstrates.
\begin{proof}
	Translating in the $x$-direction if necessary, we may assume $p=(0,0)$.
	For an argument by contradiction, after setting
	\begin{eqnarray}
		l\;=\;\liminf_{(x,y)\rightarrow(0,0)}f(x,y)
		\quad \text{and} \quad
		u\;=\;\limsup_{(x,y)\rightarrow(0,0)}f(x,y),
	\end{eqnarray}
	we assume that $l<u$.
	By hypothesis $u$ and $l$ are finite.
	
	Let us select a small constant $\epsilon>0$; below we shall see that $\epsilon=e^{-4D\Lambda}/36$ shall be sufficient, where $D=D(\Lambda)$ is the constant from Proposition \ref{PropLocDomVer}.
	Then choose $y_0$ so that on the closed rectangle $\mathcal{R}_{y_0}$ (see Lemma \ref{LemmaISubfunction}) we have $l-\epsilon(u-l)<f<u+\epsilon(u-l)$.
	
	Next we rescale both the coordinate system and the function $f$.
	For the coordinate system, rescale both $x$ and $y$ by a factor of $y_0$, so that the box under consideration is now just
	\begin{eqnarray}
		\mathcal{R}_1\;=\;\big\{y\in[0,1],\;x\in[-4\Lambda,4\Lambda]\big\}.
	\end{eqnarray}
	Change $f(x,y)$ to $\tilde{f}(x,y)$ where
	\begin{eqnarray}
		\tilde{f}(x,y)\;=\;\frac{1}{u-l}\left(f(x,y)-l\right)\,+\,\epsilon
	\end{eqnarray}
	and so that the former condition $l-\epsilon(u-l)<f<u+\epsilon(u-l)$ on $\mathcal{R}_{y_0}$ is now
	\begin{eqnarray}
		\begin{aligned}
			&0\;\le\;\tilde{f}(x,y)\;\le\;1+2\epsilon 
			\quad \text{on} \quad \mathcal{R}_1,
		\end{aligned} \label{IneqsTildeFProperties1}
	\end{eqnarray}
	and also
	\begin{eqnarray}
		\begin{aligned}
		\liminf_{(x,y)\rightarrow(0,0)}\tilde{f}(x,y)=\epsilon,
		\quad \text{and} \quad
		\limsup_{(x,y)\rightarrow(0,0)}\tilde{f}(x,y)=1+\epsilon.
		\end{aligned} \label{IneqsTildeFProperties2}
	\end{eqnarray}
	\noindent\begin{figure}[h!] 
		\centering
		\noindent\begin{subfigure}[b]{0.4\textwidth} 
			\includegraphics[scale=0.5]{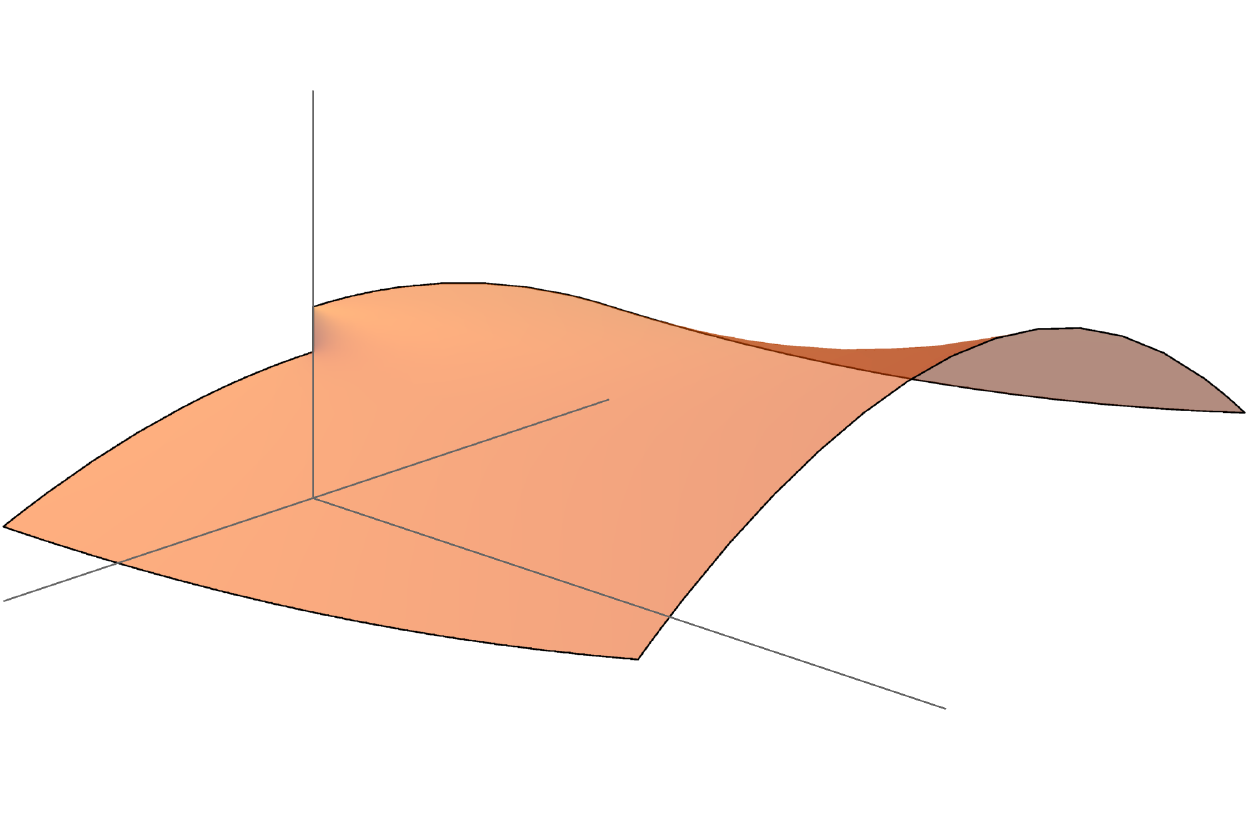}
			\caption{A function with a discontinuity at $(0,0)$, meaning $\liminf{}f<\limsup{}f$ at $(0,0)$. \\
				\\
			}\label{SubFigOrigFun}
		\end{subfigure}
		\hspace{0.5in}
		\begin{subfigure}[b]{0.4\textwidth}
			\includegraphics[scale=0.5]{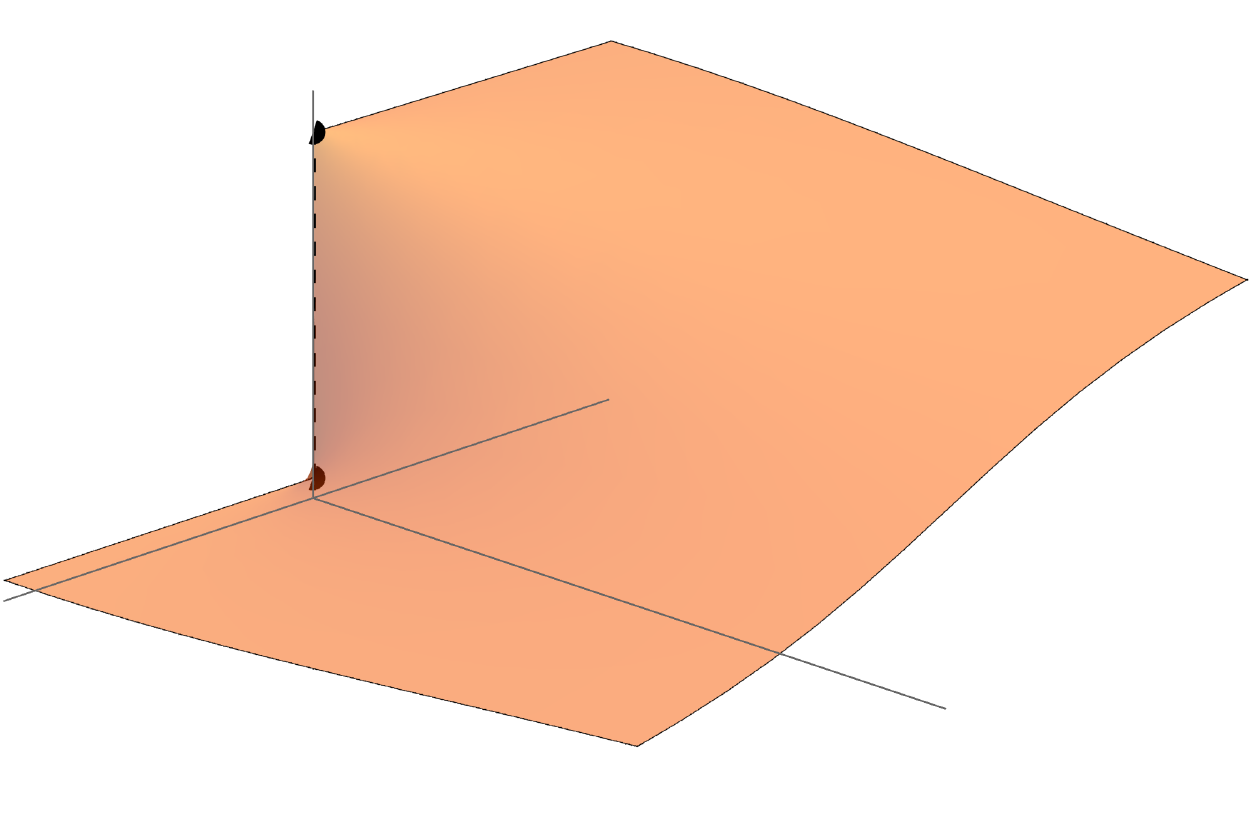}
			\caption{
				The situation after (\ref{IneqsTildeFProperties2}).
				The discontinuity is now greatly exaggerated, with $\liminf\tilde{f}=\epsilon$ and with $\limsup{}\tilde{f}=1+\epsilon$ at $(0,0)$, while also $0\le\tilde{f}\le{}1+2\epsilon$ on $\mathcal{R}_1$
			} \label{SubFigPts}
		\end{subfigure}
		\caption{\it Examining a Function with a discontinuity at the boundary}\label{FigContinuity1}
	\end{figure}
	
	It is necessary to make two further provisions.
	Passing to a smaller rectangle $\mathcal{R}_{y_0}$ if necessary---and then again rescaling the coordinates so we remain on $\mathcal{R}_1$---we may assume $y|\nabla\log\tilde{f}|<D$.
	That this is possible is due to the local version of the gradient estimate, Proposition \ref{PropLocDomVer}.
	
	The second provision is that $\tilde{f}(0,1)>1/2$.
	If on the contrary $\tilde{f}(0,1)\le{}1/2$ then simply replace $\tilde{f}(x,y)$ by $1+2\epsilon-\tilde{f}(x,y)$ to obtain $\tilde{f}(0,1)>1/2$.
	Clearly this replacement allows us to retain (\ref{IneqsTildeFProperties1}) and (\ref{IneqsTildeFProperties2}) as well as the gradient estimate.
	
	\noindent\begin{figure}[h!] 
		\centering
		\noindent\begin{subfigure}{0.4\textwidth}
			\includegraphics[scale=0.5]{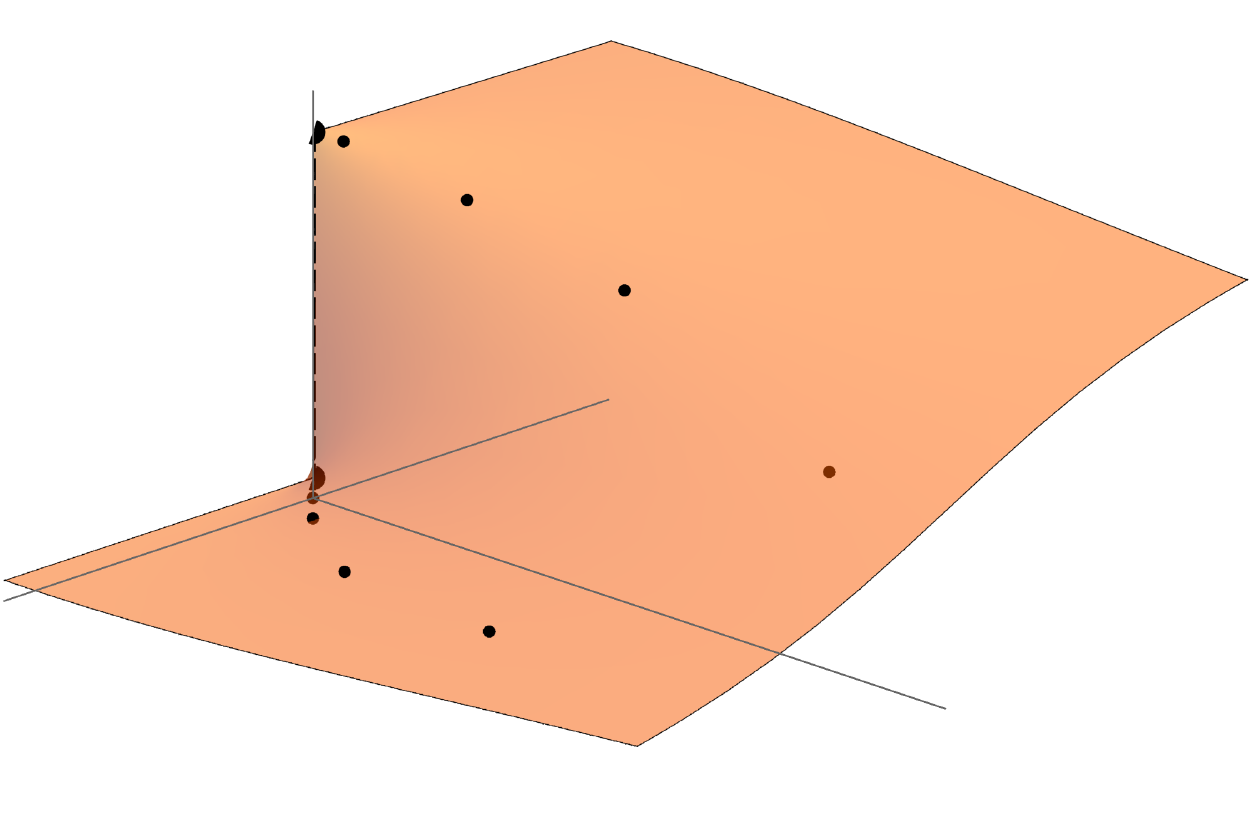}
			\caption{At least one sequence $\{p_i\}\subset\mathcal{R}_1$ has $p_i\rightarrow0$ and $\tilde{f}(p_i)\rightarrow{1}+\epsilon$ and another has $q_i\rightarrow0$ and $\tilde{f}(q_i)\rightarrow{\epsilon}$. \\
			}\label{SubFigPointsConverging}
		\end{subfigure}
		\hspace{0.5in}
		\begin{subfigure}{0.4\textwidth}
			\includegraphics[scale=0.5]{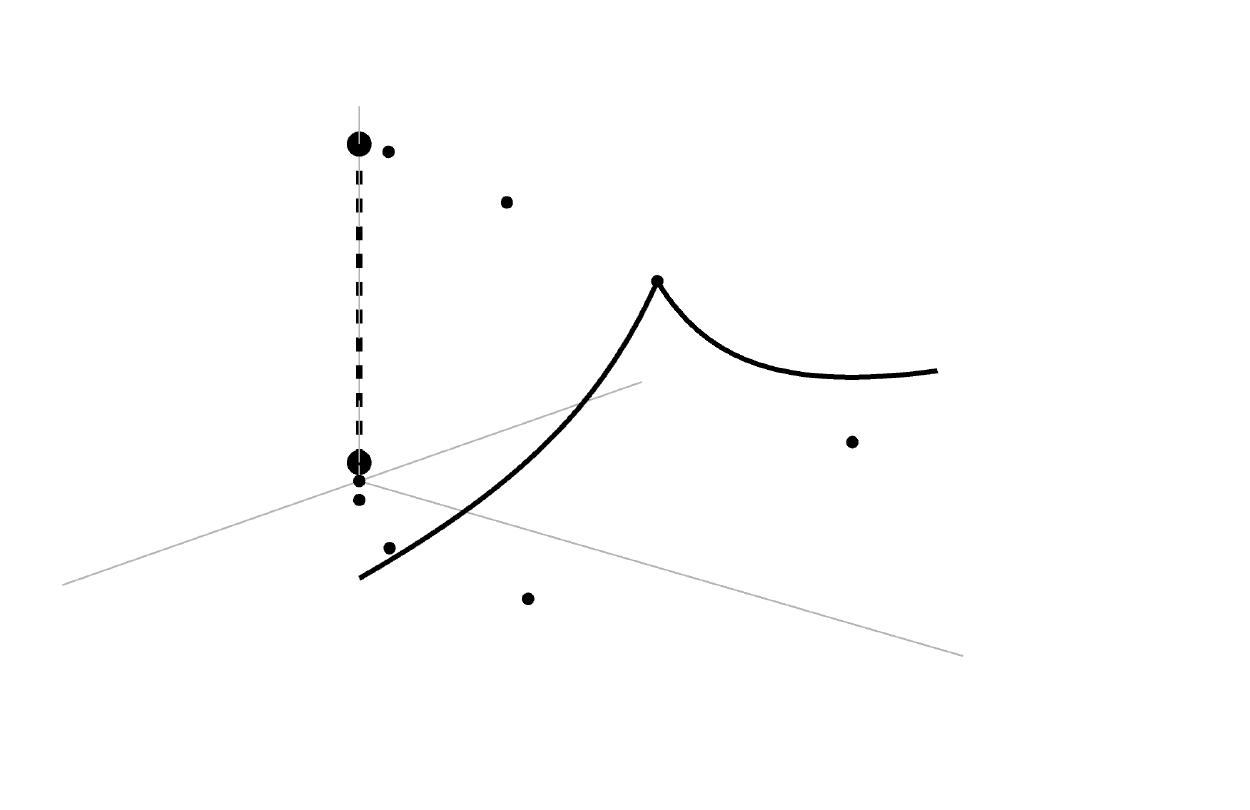}
			\caption{
				We arrange it so $\tilde{f}\ge{}1/2$ at $(0,1)$.
				Then the interior gradient estimate gives lower bounds for $\tilde{f}$ in the $x$-direction, as given by (\ref{IneqContLowerGradBounds}).
			} \label{SubFigExpLower}
		\end{subfigure}
		\caption{\it Analysis of the discontinuity point.} \label{FigContinuity2}
	\end{figure}
	
	This construction work now finished, we are able to use the lower barrier of Lemma \ref{LemmaISubfunction} to draw a contradiction.
	As stated in that lemma, we have $\tilde{f}(x,y)>C_1\psi_1(x,y)$ at all points where $\psi_1>0$, provided $\tilde{f}(x,1)>C_1\psi_1(x,1)$ on the segment $\{y=1,x\in[-4\Lambda,4\Lambda]\}$ (and $C_1$ is a constant of our choosing).
	Due to the gradient estimate $y|\nabla\log{}\tilde{f}|<D$, we have
	\begin{eqnarray}
		\begin{aligned}
		\tilde{f}(x,1)
		&\;\ge\;\tilde{f}(0,1)e^{-D|x|}
		\;>\;\frac{1}{2}e^{-D|x|}.
		\end{aligned} \label{IneqContLowerGradBounds}
	\end{eqnarray}
	Now we place a lower barrier $C_1\psi_1$ underneath the bound (\ref{IneqContLowerGradBounds}), where $C_1$ is a constant and $\psi_1$ is the function from Lemma \ref{LemmaISubfunction}; see Figure \ref{FigContinuity3}.
	Choosing $C_1=\frac12e^{-4D\Lambda}$, on the line segment $\{y=1,\,-4\Lambda\le{}x\le{}4\Lambda\}$ we have
	\begin{eqnarray}
		\begin{aligned}
			f(x,1)-C_1\psi_1(x,1)
			&\;\ge\;\frac{1}{2}e^{-D|x|}\,-\,\frac12e^{-4D\Lambda}\left(2\cos\left(\frac{\pi}{8\Lambda}x\right) -1 \right) \\
			&\;\ge\;\frac{1}{2}e^{-4D\Lambda}\,-\,\frac12e^{-4D\Lambda}\;=\;0.
		\end{aligned}
	\end{eqnarray}
	By Lemma \ref{LemmaISubfunction}, therefore $f(x,y)-C_1\psi(x,y)\ge0$ for all $(x,y)$ where $\psi(x,y)>0$ and $y<1$.
	Therefore we must have
	\begin{eqnarray}
		\begin{aligned}
			\epsilon
			\;=\;\liminf_{p\rightarrow(0,0)}f(p)
			\;\ge\;C_1\psi_1(0,0)
			\;=\;\frac{1}{18}e^{-4D\Lambda},
		\end{aligned}
	\end{eqnarray}
	where we used $\psi_1(0,0)=\frac19$.
	This contradicts $\epsilon=e^{-4D\Lambda}/36$.
	\begin{figure}
		\hspace{-2in}
		\begin{minipage}[c]{0.7\textwidth}
			\caption{
				\it A lower barrier $\psi_1$ from Lemma \ref{LemmaISubfunction} now fits beneath the exponential bounds at $\{y=1\}$.
				Near $(0,0)$ this barrier forces $\tilde{f}$ to be strictly larger than $\liminf_{p\rightarrow(0,0)}\tilde{f}$, giving the sought-for contradiction. \label{FigContinuity3}
			} 
		\end{minipage} \hspace{-0.5in}
		\begin{minipage}[c]{0.2\textwidth}
			\includegraphics[scale=0.6]{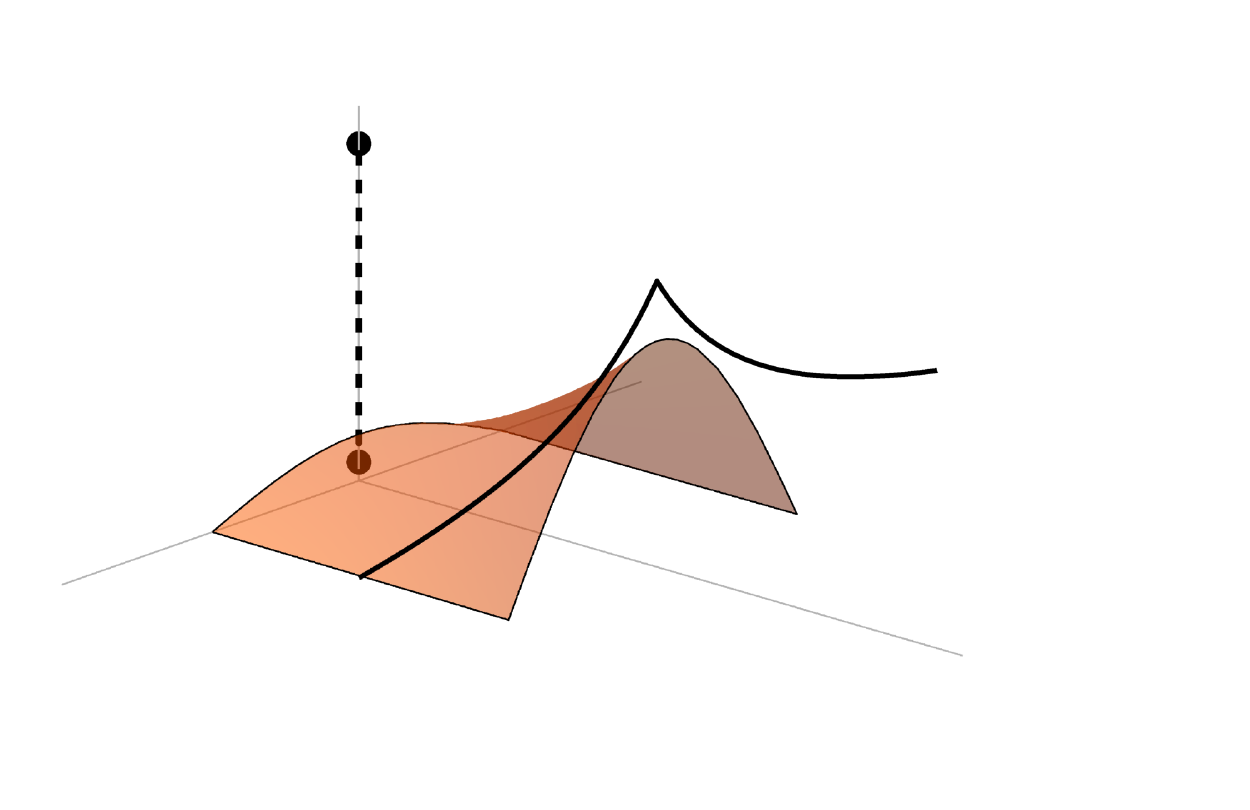}
		\end{minipage} \hfill
	\end{figure}
\end{proof}

\begin{theorem}[The Weak Maximum Principle at $\{y=0\}$] \label{ThmWeakMax}
	Assume $p\in\{y=0\}$ and $\Omega$ is some neighborhood of $p$.
	Assume $|b_1|,|b_2|,|c|\le\Lambda$, $b_2\ge1$ and $c\ge0$.
	
	If $f$ solves $L(f)\le0$ on $\Omega$ (see Definition \ref{DefSolsAtDegBound}), then $f(p)$ is not a strict local minimum on $\Omega$.
	If $f$ solves $L(f)\ge0$ on $\Omega$ (see Definition \ref{DefSolsAtDegBound}), then $f(p)$ is not a strict local maximum on $\Omega$.
\end{theorem}
\begin{proof}
	Assume $L(f)\le0$, and for a contradiction assume $f$ obtains a strict local minimum at $(x_0,0)$.
	Without loss of generality we may assume $p=(0,0)$.
	By Theorem \ref{ThmContinuity} $f$ is continuous at $p$.
	There is some $y_0$ so that on the closed rectangular region $\mathcal{R}_{y_0}$ (see Lemma \ref{LemmaISubfunction}) the function $f$ obtains its global minimum at $p$.
	
	Then the function $\tilde{f}(x,0)=f(x,y)-f(0,0)$ is zero at $(0,0)$ and is otherwise positive on the rectangle $\mathcal{R}_{y_0}$.
	Therefore there is some number $\epsilon>0$ so that $\epsilon\psi_{y_0}(x,y)<f(x,y)-f(0,0)$ on the segment $x\in[-4\Lambda{}y_0,4\Lambda{}y_0]$, $y=y_0$.
	Lemma \ref{LemmaISubfunction} then forces $\tilde{f}(0,0)\ge\epsilon/9$, contradicting $\tilde{f}(0,0)=1$.
	
	For the case $L(f)\ge0$, replace $f$ by $-f$.
\end{proof}

\subsection{Almost-Monotonicity for $y\rightarrow\infty$}

Here we prove the ``almost-monotonicity'' theorem, which strongly restrains the behavior of $f(x,y)$ along rays $y\mapsto{}f(x_0,y)$.
This theorem is global and requires $f\ge0$ on the open half-plane (although does not require local boundedness at $\{y=0\}$), and requires $b_2\ge1$ and $c\ge0$.
It states two things.
The first is that, as $y\mapsto\infty$, $f$ must ``almost'' approach its global minimum $\inf_{H^2}f$, and the second is that it does so in an ``almost'' monotonically decreasing fashion.

The result is strongest in the $c=0$ case, where $\lim_{y\rightarrow\infty}f(x_0,y)$ always exists and always equals $\inf_{H^2}f$.
\begin{proposition}[Almost Monotonicity] \label{PropAlmostMonot}
	Assume $|b_1|,|b_2|\le\Lambda$, $b_2\ge1$, $c\ge0$, and that $f\ge0$ satisfies $L(f)\le0$ and $L(f)\ge-\Lambda$ on the open half-plane $H^2$.
	Let $x_0\in\mathbb{R}$ and consider the function $y\mapsto{}f(x_0,y)$.
	
	Then there is some $\delta=\delta(\Lambda)>0$ so that
	\begin{eqnarray}
		\limsup_{y\rightarrow\infty}f(x_0,y)\;\le\;\delta^{-1}\inf_{(x,y)\in{}H^2}f(x,y). \label{IneqFunctionLimitAlmost}
	\end{eqnarray}
	Further, $y\mapsto{}f(x_0,y)$ has the ``almost monotonicity'' property, namely that
	\begin{eqnarray}
		f(x_0,y_2)<\delta^{-1}f(x_0,y_1) \label{IneqTrueAlmostMonotonicity}
	\end{eqnarray}
	whenever $y_2>y_1$.
	
	Finally in the case $c=0$, for any fixed $x_0$ we have that $\lim_{y\rightarrow\infty}f(x_0,y)$ exists and equals $\inf_{H^2}f$.
\end{proposition}
\begin{proof}
	After translating in the $x$-direction, we may assume $x_0=0$.
	The inequality (\ref{IneqTrueAlmostMonotonicity}) is simply the Harnack inequality from Theorem \ref{ThmBoundHarnackGlobal}.
	
	To prove (\ref{IneqFunctionLimitAlmost}), pick any $\epsilon>0$, and let $(\bar{x},\bar{y})$ be a point with $f(\bar{x},\bar{y})<\epsilon+\inf_{H^2}f$.
	Now scale both $x$ and $y$ coordinates by $\frac{1}{1000(|\bar{x}|+\bar{y})}$.
	In the new coordinates, therefore, we have some point $(x',y')$ within the ball of radius $1/1000$ around $(0,0)$ where $f(x',y')<\epsilon+\inf_{H_2}f$.
	
	By the interior gradient estimate Proposition, \ref{PropInteriorInitial}, we have the bound $f(x,1)>e^{-D|x|}f(0,1)$ on the line $\{y=1\}$.
	Therefore, by Lemma \ref{LemmaISubfunction} the function
	\begin{eqnarray}
		\underline{f}(x,y)
		\;=\;f(0,1)\cdot{}e^{-4D\Lambda}\cdot\psi_1(x,y)
	\end{eqnarray}
	is a lower barrier for $f$, on the domain $\{y<1\}$ where $\psi_1$ is positive.
	Examining this lower barrier, we see $\underline{f}(x,y)>f(0,1)\cdot{}e^{-4D\Lambda}\frac{1}{9}$ on the ball of radius $1/1000$.
	Therefore
	\begin{eqnarray}
		f(x',y')
		\;\ge\;\underline{f}(x',y')
		\;\ge\;
		f(0,1)\cdot{}e^{-4D\Lambda}\frac{1}{9}
	\end{eqnarray}
	and we obtain the result that $f(0,1)\le9e^{-4D\Lambda}\left(\epsilon+\inf_{H^2}f\right)$.
	Applying the Harnack inequality again therefore gives
	\begin{eqnarray}
		f(0,y)\le\delta^{-1}f(0,1)\;\le\;9\delta^{-1}e^{4D\Lambda}\,\left(\epsilon+\inf_{H^2}f\right) \label{IneqHarnackResultFuff}
	\end{eqnarray}
	Sending $\epsilon\searrow0$ provides the conclusion (\ref{IneqFunctionLimitAlmost}).
	
	To prove the special result for the case $c=0$, note that we can add or subtract any value from $f$ and still retain $L(f)\le0$.
	Therefore we can assume $\inf_{H_2}f=0$, and from (\ref{IneqHarnackResultFuff}) conclude that
	\begin{eqnarray}
		\begin{aligned}
		0\;\le\;\liminf_{y\rightarrow\infty}f(0,y)
		\;\le\;\limsup_{y\rightarrow\infty}f(0,y)
		\;\le\;\delta^{-1}\inf_{H^2}f
		\;=\;0
		\end{aligned}
	\end{eqnarray}
\end{proof}

\section{The Liouville Theorem}

Our foundational result, Proposition \ref{PropInteriorInitial}, provides interior growth estimates that are polynomial in $y$ and exponential in $x$.
The first aim of this section is to improve this to polynomial growth/decay in $x$.

In attempting to prove the Liouville theorem, the idea is to try to construct an upper barrier on some strip that rises more quickly to infinity in the $x$-direction than any solution $f$.
Such an upper barrier could be used to crush down the value of $f$ to zero in the regions of moderate $x$-values.
The difficulty in this strategy is that close to $\{y=0\}$ we lose control over the growth of $f$, and on the boundary itself we have no restrictions whatever on growth: possibly $x\mapsto{}f(x,0)$ has extreme growth like $e^{e^x}$, or oscillates wildly; this makes any kind of upper barrier argument simply impossible.

We remedy this by taking advantage of the scale-invariance of the operator $L$ and using a blow-up style argument in order to capture some region of very large, but controlled growth.
But we can only ever reduce the situation to exponential growth bounds in $x$ this way.
This is insufficient, because the barriers available to us themselves have fixed exponential growth bounds, and it doesn't seem possible to force the exponential rate $D$ from Proposition \ref{PropInteriorInitial} to be smaller than the exponential rate available to us in the barriers.

The next proposition helps remedy this by improving on the exponential growth bounds from Proposition \ref{PropInteriorInitial} to interior polynomial bounds in the $x$ direction.

\begin{proposition}[Polynomial bounds in $x$] \label{PropPolyXBounds}
	Assume $|b_1|,|b_2|,|c|<\Lambda$, and $b_2\ge1$, $c\ge0$.
	Assume $f\ge0$ satisfies $L(f)\le0$ and $L(f)\ge-\Lambda$ on the open half-plane $H^2$.
	There exists a constant $\delta=\delta(\Lambda)>0$ so that for any two values $x,x'\in\mathbb{R}$ we have the growth/decay bounds
	\begin{eqnarray}
		\begin{aligned}
		&f(x,y)
		\;\le\;f(x',y)\frac{1}{\delta}\left(\frac{|x-x'|}{y}+1\right)^{D} \\
		&f(x,y)
		\;\ge\;f(x',y)\delta\left(\frac{|x-x'|}{y}+1\right)^{-D}
		\end{aligned}
	\end{eqnarray}
	where $D=D(\Lambda)$ is from the interior gradient estimate, Proposition \ref{PropInteriorInitial}.
\end{proposition}
\begin{proof}
	We may shift the $x$-coordinate and simultaneously scale the $x$ and $y$ coordinates so that without loss of generality we may assume $x'=0$ and $y=1$.
	Multiplying $f$ by a constant if necessary, we may assume $f(0,1)=1$.
	
	The proof will require construction of a lower barrier.
	Taking cues from the separation of variables technique our barrier will have the form
	\begin{eqnarray}
		\varphi(x,y)\;=\;g(y)\,e^{-x}.
	\end{eqnarray}
	Plugging in to the operator $L$, after elementary simplification we obtain
	\begin{eqnarray}
		\begin{aligned}
		&L(g)
		\;=\;\left[y^2g''(y)+b_2yg'(y)+\left(c-b_1y+y^2\right)g(y)\right]e^{-x}
		\end{aligned}
	\end{eqnarray}
	We are only concerned with the region where $\psi\ge0$, so using $c\ge0$ and $b_2>-\Lambda$ we find
	\begin{eqnarray}
		\begin{aligned}
		&L(g)
		\;\ge\;\left[y^2g''(y)+b_2yg'(y)+\left(-\Lambda{}y+y^2\right)g(y)\right]
		e^{-x}
		\end{aligned}
	\end{eqnarray}
	Unfortunately it may be the case that $g'(y)$ have either a positive or negative sign; indeed at $y=0$ it is certainly the case the $g'(y)>0$, as the ODE is approximately $yg''+b_2g'-\Lambda{}g\ge0$ which is almost $g'>\Lambda/b_2$ for small $y$; recalling that $b_2\in[1,\Lambda]$, so in particular $b_2>0$, we have $g'(0)>0$.
	We therefore split the inequality into the cases where $g'\ge0$ and $g'<0$:
	\begin{eqnarray}
		\begin{aligned}
		L(g)
		&\ge\left[y^2g''(y)+yg'(y)+\left(-\Lambda{}y+y^2\right)g(y)\right]
		e^{-x}\quad \text{for} \quad g'(y)\ge0 \\
		L(g)
		&\ge\left[y^2g''(y)+\Lambda{}yg'(y)
		+\left(-\Lambda{}y+y^2\right)g(y)\right]
		e^{-x}\quad \text{for} \quad g'(y)<0.
		\end{aligned}
	\end{eqnarray}
	Looking for non-negative solutions of
	\begin{eqnarray}
		\begin{aligned}
		&y^2g''(y)+yg'(y)+\left(-\Lambda{}y+y^2\right)g(y)\;=\;0
		\;\;\,\quad \text{if} \quad g'(y)\ge0 \\
		&y^2g''(y)+\Lambda{}yg'(y)+\left(-\Lambda{}y+y^2\right)g(y)\;=\;0
		\quad \text{if} \quad g'(y)<0
		\end{aligned} \label{EqnSubfunctionDiffEQ}
	\end{eqnarray}
	we find the following:
	\begin{eqnarray}
		\begin{aligned}
		&g(y)\;=\; \\
		&
		\begin{cases}
		e^{-\sqrt{-1}y}
		{}_1F_1\left(\frac12\left(1-\sqrt{-1}\Lambda\right);1;\,2\sqrt{-1}y\right), & y\in[0,\bar{y}] \\
		\\
		C_1e^{-\sqrt{-1}y}
		{}_1F_1\left(\frac12\left(\Lambda-\sqrt{-1}\Lambda\right);\Lambda;\,2\sqrt{-1}y\right)
		 & \multirow{2}*{$y\in(\bar{y},y_0]$} \\
		\;+C_2Re\left[e^{-\sqrt{-1}y}
		U\left(\frac12\left(\Lambda-\sqrt{-1}\Lambda\right),\Lambda,\,2\sqrt{-1}y\right)
		\right], & 
		\end{cases}
		\end{aligned} \label{EqnSubfunForY}
	\end{eqnarray}
	where ${}_1F_1(a;b;y)$, resp. $U(a,b,y)$, is the confluent hypergeometric function of the first kind, resp. second kind---see, for instance, Appendix A of \cite{Figu1} for a derivation.
	The constants $C_1$, $C_2$ are chosen so that $g(y)$ remains $C^{1,1}$; for a depiction see Figure \ref{SubFigBasicG}.
	One can prove that the expression 
	$$e^{-\sqrt{-1}y}
	{}_1F_1\left(\frac12\left(A-\sqrt{-1}B\right);A;\,2\sqrt{-1}y\right)$$
	is actually real-valued when $A$ and $B$ are real-valued, although we shall not pursue this tedious verification; one can certainly just take the real-valued part of this expression and not worry if it is complex-valued or not.
	
	In (\ref{EqnSubfunForY}) the break point $\bar{y}$ occurs at the maximum of $g$ which we have labeled $\bar{y}$. The value $y_0$ is the first zero of $g$, and the coefficients $C_1$, $C_2$ are chosen so that both $g(\bar{y}^-)=g(\bar{y}^+)$ and $g'(\bar{y}^-)=g'(\bar{y}^-)$.
	See figure \ref{FigFunxBar} for a depiction.
	
	In fact only two aspects of the solution (\ref{EqnSubfunForY}) are important for our proof.
	The first is that $g(y)=1+\Lambda^{-1}y+O(y^2)$ and the second is that $g(y)$ has zeros.
	
	This author is unaware of any treatment of the locations of zeros for solutions of (\ref{EqnSubfunctionDiffEQ}), but nevertheless we can show that zeros must exist, for using $-\Lambda{}y<0$ and $-\Lambda{}y>-\frac12y^2-\frac12\Lambda^2$ we can see that solutions of (\ref{EqnSubfunctionDiffEQ}) are sandwiched between a Bessel function $J_0(y)$ and a function of the form $y^{-\Lambda/2}J_{\Lambda/2}(y)$.
	Both of these have zeros, so solution of (\ref{EqnSubfunctionDiffEQ}) are also forced to have zeros.
	Certainly as $y\rightarrow\infty$ the $(-\Lambda{}y+y^2)g(y)$ term in (\ref{EqnSubfunctionDiffEQ}) is nearly $y^2g(y)$, so solutions must be Bessel-like for large $y$.
	
	With $y_0=y_0(\Lambda)$ being the first zero of $g(y)$, then for any parameter $y_d>0$ consider the function
	\begin{eqnarray}
		\begin{aligned}
			\psi_{y_d}(x,y)
			\;=\;C_1g\left(\frac{y_0}{y_d}y\right)\,e^{-\frac{y_0}{y_d}x}
		\end{aligned} \label{EqnXBarrierFunction}
	\end{eqnarray}
	where we shall choose the constant $C_1=C_1(y_d,\Lambda)$ below.
	Due to the simultaneous scaling in both coordinates, we retain $L(\psi_{y_d})\ge0$.
	By design, we have that $\psi_{y_d}(x,y_d)=0$ for any $x$.
	
	Now choose a value $y_c>0$; using $y_c$ choose values $y_d$ and $C_1$ so that the function $y\mapsto{}\psi_{y_d}(0,y)$ has point of tangency with $y\mapsto{}y^{-D}$ at the point $y_c$ (where $D$ is the value from Proposition \ref{PropInteriorInitial}).
	Assuming $y_c$ is sufficiently large, then also $\varphi_{y_c}<\delta$, where $\delta$ is the value from the Harnack inequality, Theorem \ref{ThmBoundHarnackGlobal}.
	From these choices it follows that, on the line $\{x=0\}$, we have $\psi_{y_d}(0,y)\le{}f(0,y)$.
	
	\begin{figure}[h!] 
		\centering
		\noindent\begin{subfigure}{0.4\textwidth}
			\includegraphics[scale=0.3]{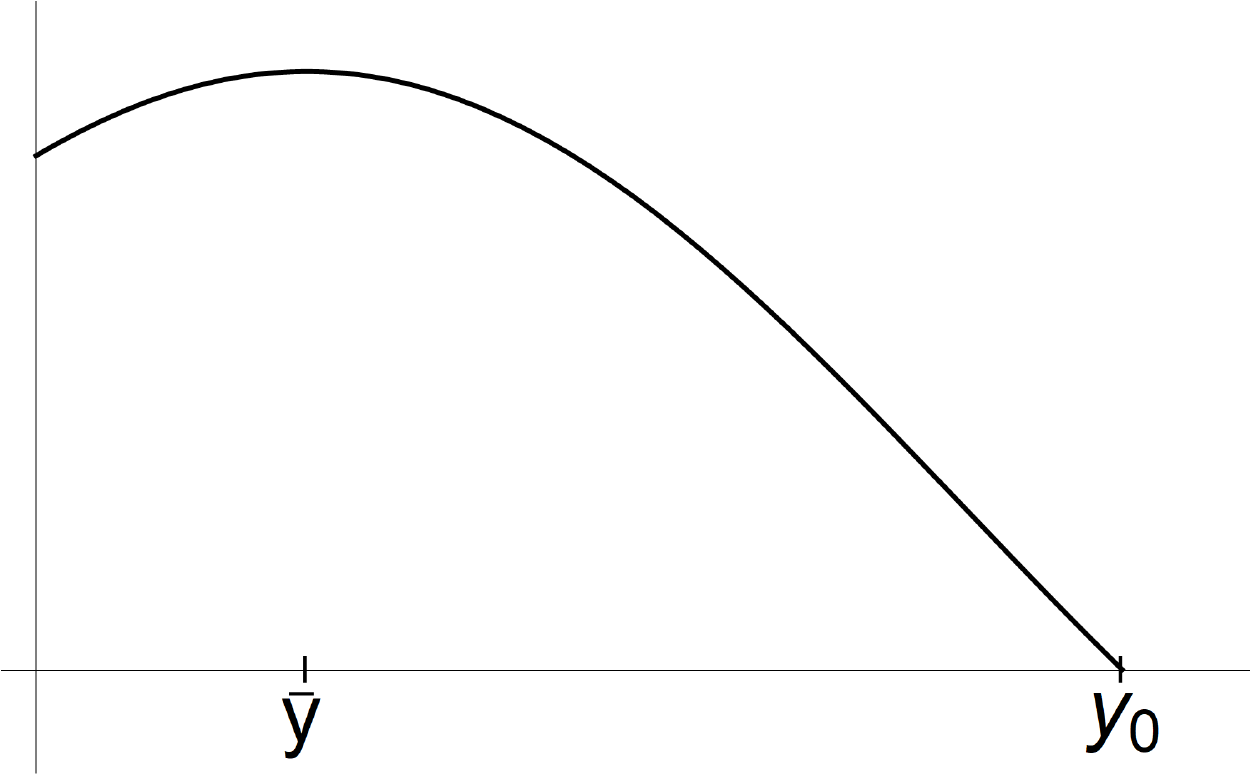}
			\caption{Depiction of the function $g(y)$, along with the ``break point'' $\bar{y}$. \\
				\vspace{0.45in}}\label{SubFigBasicG}
		\end{subfigure}
		\hspace{0.2in}
		\begin{subfigure}{0.4\textwidth}
			\includegraphics[scale=0.3]{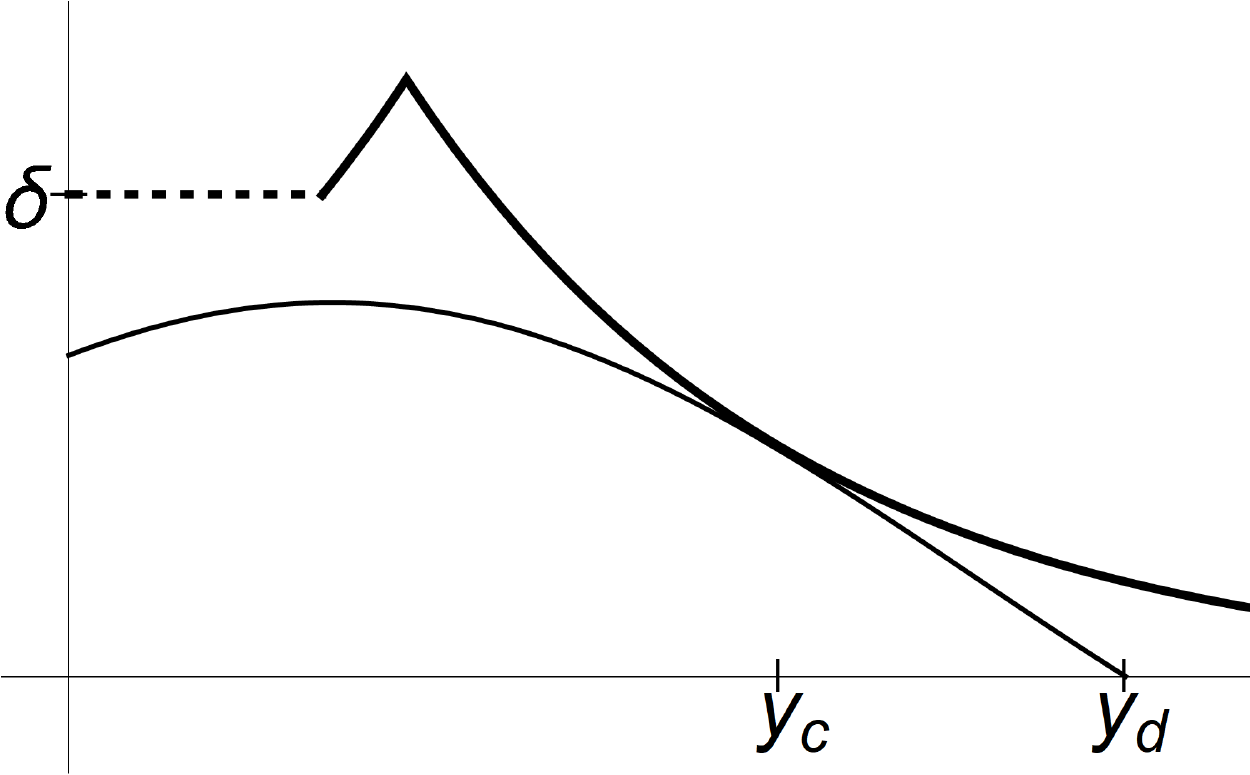}
			\caption{
				Lower bounds for $f$ depicted: the dashed line is from the Harnack inequality, the thick line is from the polynomial bound $y^{-D}$.
				After $y_c$ is chosen, then $y_d$ and $C_1$ are found so $\psi_{y_d}(0,y)<f(0,y)$.
			} \label{SubFigBoundedG}
		\end{subfigure}
		\caption{The function $g(y)$ and its interaction with bounds on $f$. \label{FigFunxBar}}
	\end{figure}

	Indeed more is true.
	Having bounded $\psi_{y_d}\le{}f$ on $\{x=0\}$ we can prove that $\psi_{y_d}\le{}f$ on the entire region $y\in[0,y_d]$, $x\ge0$.
	To see this, just subtract some small value $\epsilon$ from $\psi_{y_d}$ and note that $\psi_{y_d}-\epsilon\ge0$ on a compact subset of $\{y\ge0\}\cap\{x\ge0\}$.
	On the boundary of this compact subset we either have $\psi_{y_d}=0$ or else $\{y=0\}$ or $\{x=0\}$.
	On $\{x=0\}$ we have already seen $\psi_{y_d}<f$.
	On $\{y=0\}$ we need not even check whether $\psi_{y_d}<f$ or not; this is because we can always add a tiny multiple of $-\log{y}$, which forces $\psi_{y_d}<f$ near $y=0$, and then send this tiny multiple to $0$.
	Sending $\epsilon\searrow0$, we see $\psi_{y_d}\le{}f$ as claimed.
	
	Now having established that $\psi_{y_d}$ is a subfunction on the half-strip $\{x\ge0,\,y\in[0,y_d]\}$, we proceed to the proof of the proposition.
	
	\noindent\begin{figure}[h!]
		\vspace{-0.29in}
		\hspace{-1.9in}
		\begin{minipage}[c]{0.7\textwidth}
			\caption{
				\it Depiction of the barrier $\psi_{y_0}$ of (\ref{EqnXBarrierFunction}).
				Along the line $y$-axis, $\psi_{y_d}(0,y)=C_1g(y)$ with zero at $y_d$; we arrange it so $\psi_{y_d}(0,y)<f(0,y)$ along this axis, as depicted in Figure \ref{SubFigBoundedG}.
				We see exponential decrease in the $x$-direction.
				\label{FigBarrierX}
			}
		\end{minipage} \hspace{-0.5in}
		\begin{minipage}[c]{0.2\textwidth}
			\includegraphics[scale=0.475]{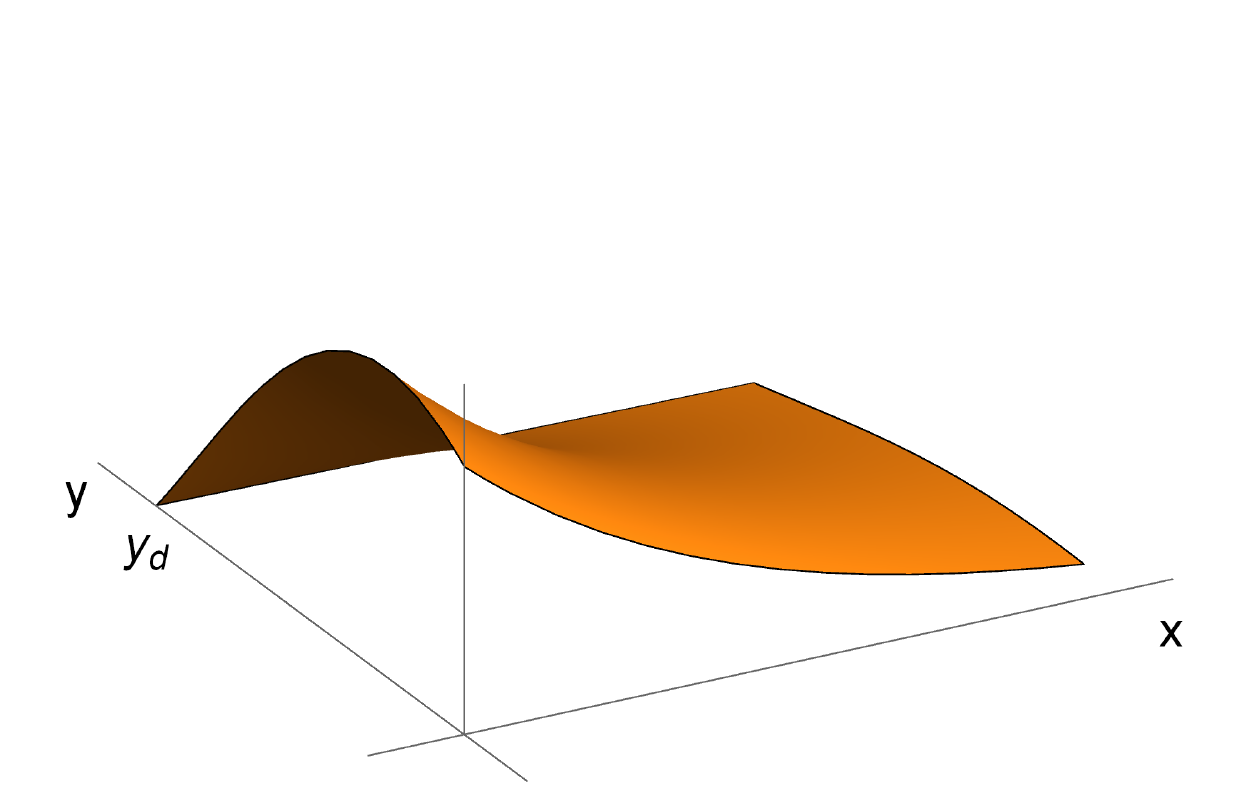}
		\end{minipage} \hfill
	\end{figure}

	Choose any sufficiently large $x_0>0$, and set $y_c=x_0$.
	Then find $y_d$ and $C_1$ so that $y\mapsto\psi_{y_d}(0,y)$ is tangent to $y\mapsto{}y^{-D}$ at $y_c$.
	For the argument, it will be sufficient to note that $C_1>\frac{1}{max_{y\in[0,y_c]}\{g(y)\}}y_c^{-D}$ and that $\frac{1}{max\,g(y)}$ is a function of $\Lambda$ only.
	We also remark that certainly $y_d>y_c$, as in Figure \ref{FigFunxBar}.
	Therefore
	\begin{eqnarray}
		\begin{aligned}
			f(x_0,1)
			&\;\ge\;\psi_{y_d}(x_0,1)
			\;=\;C_1\,g\left(\frac{y_0}{y_d}\right)\,e^{-\frac{y_0}{y_d}x_0} \\
			&\;\ge\;\frac{1}{max_y\{g(y)\}}y_c^{-D}
			\,g\left(\frac{y_0}{y_d}\right)\,e^{-\frac{y_0}{y_c}x_0}
			\quad \text{(by $y_d>y_c$)}
			 \\
			&\;=\;\frac{1}{max_y\{g(y)\}}x_0^{-D}
			\,g\left(\frac{y_0}{y_d}\right)\,e^{-y_0}
			\quad \text{(using $y_c=x_0$).}
		\end{aligned}
	\end{eqnarray}
	Finally we note that because we have chosen $x_0$ large and therefore $y_d$ large, we have $y_0/y_d\approx0$ so $g(y_0/y_d)=1+\Lambda^{-1}(y_0/y_d)+O((y_0/y_d)^2)<2$.
	Therefore
	\begin{eqnarray}
		\begin{aligned}
		f(x_0,1)
		&\;\ge\;
		\frac{2}{max\{g(y)\}}
		e^{-y_0}x_0^{-D}
		\;=\;
		\delta{}x_0^{-D}
		\end{aligned}
	\end{eqnarray}
	where $\delta=\delta(\Lambda)>0$ is defined to be $\delta=\frac{2}{max\{g(y)\}}
	e^{-y_0}$.
	
	Simultaneous scaling in both $x$ and $y$ coordinates, we see that
	\begin{eqnarray}
		f(x,y)\;\ge\;\delta\left(\frac{x}{y}\right)^{-D} \label{IneqFirstDecayIneq}
	\end{eqnarray}
	for $x/y$ sufficiently large.
	Whether $x/y$ is sufficiently large or not we always have $f(x,y)>f(0,y)e^{-D(x/y)}$ by the interior gradient bound, Proposition \ref{PropInteriorInitial}, and so by changing the constant $\delta$ if necessary we have
	\begin{eqnarray}
		f(x,y)\;\ge\;\delta\left(\frac{x}{y}+1\right)^{-D}
	\end{eqnarray}
	for all $x>0$.
	Recalling that we scaled $f$ so $f(0,y)=1$, the inequality for when $f$ has not been scaled is
	\begin{eqnarray}
		f(x,y)\;\ge\;f(0,y)\cdot{}\delta\left(\frac{x}{y}+1\right)^{-D}.
	\end{eqnarray}
	
	Ostensibly this is a decay estimate: we have shown that decay in the $x$-direction is no worse than polynomial.
	But of course a decay estimate is also a growth estimate for if, on the contrary, $f(x_0,y)>f(0,y)\frac{1}{\delta}(|x_0|/y)^{+D}$ then we simply make the coordinate transformation $x\mapsto{}x_0-x$ to obtain $f(0,y)>f(x_0,y)\frac{1}{\delta}(|x_0|/y)^{+D}$, contradicting the decay estimate (\ref{IneqFirstDecayIneq}).
\end{proof}

\begin{theorem}[The Liouville theorem] \label{ThmLiouvilleActual}
	Assume $|b_1|,|b_2|,|c|<\Lambda$, and $b_2\ge1$, $c=0$.
	Assume $f\ge0$ and $f$ solves $L(f)=0$ on the closed half-plane $\overline{H}{}^2$ (see Definition \ref{DefSolsAtDegBound}).
	Then $f$ is constant.
\end{theorem}
\begin{proof}
	For a proof by contradiction, assume $f$ is not constant.
	
	After subtracting a constant if necessary we may assume that $\inf_{\overline{H}{}^2}f=0$; because $c=0$ we retain $L(f)=0$.
	By Proposition \ref{PropAlmostMonot} for any fixed $\bar{x}$ we have $\lim_{y\rightarrow\infty}f(\bar{x},y)=0$.
	By the Harnack inequality at the boundary, Theorem \ref{ThmBoundHarnackGlobal}, we have $f(\bar{x},0)>0$.
	
	Pick some large $N$; the value $N=\frac{1}{0.11}\delta^{-1}(D!)^2\Lambda^D$ will suffice.
	Define the one-variable function $\rho(x)$ as follows:
	\begin{eqnarray}
		\rho(x)
		\;=\;\sup\left\{y>0\;\big|\; f(x,y)\,>\,\frac{1}{N}f(x,0)\right\}.
	\end{eqnarray}
	For any given $x$, $\rho(x)$ measures how long it takes $f$ to decay from what may be an extremely large value at the boundary down to values that are a small but definite fraction of this.
	We always have $0<\rho(x)<\infty$ because of two facts: continuity at the boundary ensures $\rho(x)>0$ and $\lim_{y\rightarrow\infty}f(x,y)=0$ ensures $\rho(x)<\infty$.
	Thus
	\begin{eqnarray}
		\rho:\mathbb{R}\,\longrightarrow\,(0,\infty). \label{MapRhoRange}
	\end{eqnarray}
	Continuity of $\rho$ easily follows from the continuity of $f$ on the closed half-plane, Theorem \ref{ThmContinuity}.
	
	Having defined $\rho$, we give an outline of the proof.
	First we perform a point-picking and scaling argument to create a situation where $\rho(0)=1$ and $\rho(x)>\frac12$ on some very large interval, $x\in[-R,R]$ for some very large $R=R(\Lambda)$.
	The fact that $\rho(0)=1$ means precisely that $f(0,0)=N\cdot{}f(0,1)$.
	Having done this, we observe that for all $x\in[-R,R]$, we actually have polynomial bounds at $\{y=0\}$: this is because we have polynomial bounds on $f(x,\frac12)$ along the line segment $\{(x,\frac12)\,\big|\,x\in[-R,R]\}$ and then the fact that $\rho(x)\ge\frac12$ means---by the definition of $\rho$---that $f(x,0)\le{}N{}f(x,\frac12)$.
	
	The second part of the argument is the barrier argument.
	We have uniform polynomial bounds on $f(x,y)$ on some very long strip $\mathcal{S}=\{(x,y)\,\big|\,y\in[0,1],\,x\in[-R,R]\}$.
	Then we place a barrier over top of $f(x,y)$ along the $\mathcal{S}$, and actually contradict the fact that $f(0,0)=N\cdot{}f(0,1)$.
	We can do this because the natural upper barriers available to us all have exponential growth, which vastly outstrips the polynomial growth for $f$ that we contrived with our point-picking argument.
	
	The first part of the argument is the pointpicking and re-scaling argument.
	Let $R$ be a very large number that we shall choose below.
	Choose any number $x_0$ and consider the interval $[x_0-R\rho(x_0),x_0+R\rho(x_0)]$.
	Let $x_1\in{}[x_0-R\rho(x_0),x+R\rho(x_0)]$ be a value with $\rho(x_1)<\frac12\rho(x_0)$, if such an $x_1$ exists.
	If such an $x_1$ does not exist, then we cease the process, satisfied with finding a value $x_0$ where $\rho(x)\ge\frac12\rho(x_0)$ in the interval $[x_0-R\rho(x_0),x_0+R\rho(x_0)]$.
	
	But if such an $x_1$ {\it does} exist, we set up an iteration process: assume $x_1,\dots,x_{i}$ have been chosen so $x_{j+1}\in[x_j-R\rho(x_j),x_j+R\rho(x_j)]$ and $\rho(x_{j+1})<\frac12\rho(x_j)$.
	The choose the next value $x_{i+1}$ to be any value in $[x_i-R\rho(x_i),x_i+R\rho(x_i)]$ with $\rho(x_{i+1})<\frac12\rho(x_i)$, assuming such a point exists.
	If such a point does not exist, we stop the process with the point $x_{i-1}$.
	
	This process must terminate at some {\it finite} stage.
	To see why, note that each $x_i$ must remain inside an interval of finite length around the original value $x$.
	To see this, we use $\rho(x_j)<2^{-j}\rho(x_0)$ to estimate
	\begin{eqnarray*}
		\begin{aligned}
			|x_i-x_j|
			&\;\le\;
			|x_i-x_{i+1}|+|x_{i+1}-x_{i+2}|+\dots+|x_{j-1}-x_j| \\
			&\;\le\;R\rho(x_i)+R\rho(x_{i+1})+\dots+R\rho(x_{j-1}) \\
			&\;\le\;R\left(2^{-i}+2^{-i-1}+\dots+2^{-j+1}\right)\rho(x_0)
			\;<\;2^{-i+1}R\rho(x_0).
		\end{aligned}
	\end{eqnarray*}
	Thus $x_0,x_1,\dots$ is a Cauchy sequence, and there is some value $x_\infty=\lim_{i\rightarrow\infty}x_i$.
	But $\rho$ is continuous, so $\rho(x_\infty)=\lim_i\rho(x_i)=0$.
	This is impossible by (\ref{MapRhoRange}).
	
	Therefore the point-picking process terminates at some value $x_j$, which we re-label $x'$.
	For this $x'$ we have $\rho(x)\ge\frac12\rho(x')$ for all $x\in[x'-R\rho(x'),x+R\rho(x')]$.
	
	Now re-scale the coordinate system, setting new coordinates
	\begin{eqnarray}
		\begin{aligned}
			\bar{x}\;=\;\frac{1}{\rho(x')}(x-x'), \quad
			\bar{y}\;=\;\frac{1}{\rho(x')}y.
		\end{aligned}
	\end{eqnarray}
	The function $\rho$ scales as a distance, so measured in this new system, we have $\rho(0)=1$ and $\rho(\bar{x})\ge\frac12$ for all $\bar{x}\in[-R,R]$.
	Multiplying the function $f$ by a constant, we have $f(0,1)=1$ and by the fact that $\rho(0)=1$ we have $f(0,0)=N$.
	The fact that $\rho(x)<\frac12$ on $x\in[-R,R]$ means precisely that $f(x,0)<N{}f(x,1/2)$.
	We can now verify the following facts on the very long strip $\mathcal{S}=\{(x,y)\,\big|\,x\in[-R,R],\;y\in[0,1]\}$:
	\begin{itemize}
		\item[{\it{a}})] Value at two points: $f(0,1)=1$ and $f(0,0)=N$
		\item[{\it{b}})] Bounds along the edge $\{y=1\}$: $f(x,1)<\delta^{-1}\left(1+|x|\right)^D$
		\item[{\it{c}})] Bounds along the edge $\{y=0\}$: $f(x,0)<N\delta^{-1}\left(1+2|x|\right)^D$
		\item[{\it{d}})] Bounds along the edge $\{x=\pm{}R\}$: $f(\pm{}R,y)\le{}N\delta^{-2}\left(1+2R\right)^D$.
	\end{itemize}
	Item ({\it{b}}) is due to Proposition \ref{PropPolyXBounds}.
	Item ({\it{c}}) follows from Proposition \ref{PropPolyXBounds} applied to $f(x,1/2)$ along with $f(x,0)<N{}f(x,1/2)$.
	Item ({\it{d}}) follows from the Harnack inequality, Proposition \ref{ThmBoundHarnackGlobal}.
	
	For the second part of our argument, we create an upper barrier.
	To this end, consider $G(x,y)=C_1f(x)g(y)$ where $f(x)=\cosh(x/\Lambda)$.
	Plugging in to the operator $L$ we find
	\begin{eqnarray}
		\begin{aligned}
			L(fg)
			&\;=\;y^2\left(f_{xx}g+fg_{yy}\right)
			+y\left(b_1f_xg+b_2fg_y \right) \\
			&\;\le\;\cosh(x/\Lambda)
			\left(y^2g_{yy}+yg_{y}
			+\left(y\frac{1}{\Lambda}b_1\tanh(x/\Lambda)+(y/\Lambda)^2\right)g\right) \\
			&\;\le\;\cosh(x/\Lambda)
			\left(y^2g_{yy}+yg_{y}
			+\left(y+(y/\Lambda)^2\right)g\right)
			\end{aligned}
	\end{eqnarray}
	where we have assumed $g_y<0$ and we used $|\tanh(x/\Lambda)|<1$.
	Solving $y^2g_{yy}+yg_{y}+\left(y+(y/\Lambda)^2\right)g=0$ gives
	\begin{eqnarray}
		g(y)
		\;=\;e^{-\frac{\sqrt{-1}}{\Lambda}y}\,
		{}_1F_1\left(\frac{1+\sqrt{-1}\Lambda}{2};\,1;\,\frac{2\sqrt{-1}}{\Lambda}y\right).
	\end{eqnarray}
	A quick examination, perhaps with a computer, will verify that $g$ is real valued, and is both positive and decreasing on the interval $y\in[0,1]$.
	We have $g(0)=1$ and, for all $\Lambda$ larger than about 2, $g(1)\ge0.22$.
	
	With our upper barrier being $G(x,y)=C_1f(x)g(y)$ we must choose a constant $C_1$ so that $G(x,1)>f(x,1)$.
	Considering the bound $f(x,1)<\delta^{-1}(1+|x|)^D$, we choose the value
	\begin{eqnarray}
		C_1\;=\;\frac{1}{0.22}\delta^{-1}(D!)^2\Lambda^D.
	\end{eqnarray}
	Notice this is half of our chosen value of $N$:
	\begin{eqnarray}
		N\;=\;\frac{1}{0.11}\delta^{-1}(D!)^2\Lambda^D.
	\end{eqnarray}
	Finally we choose $R=R(\Lambda)$ so big that
	\begin{eqnarray}
		4\delta^{-2}(1+2R)^D
		\;<\;0.22\cosh(R).
	\end{eqnarray}
	With these choices, we verify that our barrier $G(x,y)=C_1f(x)g(y)$ is actually larger than $f(x,y)$ on three boundary segments of the strip $S$.
	
	For the boundary segment $(x,1)$ where $x\in[-R,R]$, we have chosen the value of $C_1$ precisely so that
	\begin{eqnarray}
		G(x,1)\;>\;f(x,1).
	\end{eqnarray}
	For the two boundary segments $(\pm{}R,y)$, $y\in[0,1]$, we use $g(0)>0.22$ and our choices for $C_1$, $N$, and $R$ to compute
	\begin{eqnarray}
		\begin{aligned}
			f(\pm{}R,\,y)
			&\;\le\;N\delta^{-2}\left(1+2R\right)^D \\
			&\;\le\;\frac{0.22}{4}N\cosh(R)
			\;\le\;\frac14N\frac{1}{C_1}
			\;C_1\cosh(R)g(0) \\
			&\;\le\;\frac12C_1\cosh(R)g(0)
			\;\le\;\frac12C_1\cosh(R)g(y) \\
			&\;=\;\frac12G(\pm{}R,y).
		\end{aligned}
	\end{eqnarray}
	We have verified that $G>f$ on the three non-degenerate boundary segments of $\mathcal{S}$.
	It follows that $G$ is indeed a superfunction.
	In particular $G(0,0)\ge{}f(0,0)$.
	
	But then we see that
	\begin{eqnarray}
		\begin{aligned}
			C_1
			\;=\;G(0,0)
			\;\ge\;f(0,0)\;=\;N\;=\;2C_1.
		\end{aligned}
	\end{eqnarray}
	This contradiction established the result.
\end{proof}

\begin{corollary} \label{CorB2Less}
	Assume $\lambda>0$ is any constant, and assume $g\ge0$ solves
	\begin{eqnarray}
		y^2\triangle{}f\,+\,(1-\lambda)yf_y\;=\;0 \label{EqnB2LessThan1}
	\end{eqnarray}
	on the upper half-plane.
	Assume $g$ is continuous at $\{y=0\}$ and $f(x,0)=0$.
	Then some $C_1\ge0$ exists so that
	\begin{eqnarray}
		f(x,y)\;=\;C_1y^\lambda.
	\end{eqnarray}
\end{corollary}
\begin{proof}
	One may check that the function $F(x,y)=y^{-\lambda}f(x,y)$ satisfies the equation $y^2\triangle{}F+(1+\lambda)F_y=0$.
	If we can verify that $F=y^{-\lambda}f$ is locally bounded at $\{y=0\}$, then the Liouville theorem shows that $y^{-\lambda}f(x,y)=C_1$, as desired.
	
	To verify this local boundedness, we pinch $f(x,y)$ near $(0,0)$ by a subfunction and a superfunction, each of which has $y^\lambda$ behavior near $(0,0)$.
	We remark that this is sufficient to pinch $f$ at any boundary value $(x,0)$, by the translation-invariance of the equation (\ref{EqnB2LessThan1}).
		
	Finding a {\it sub}function with the right behavior is easy: we use
	\begin{eqnarray}
		\underline{f}(x,y)=y^{\frac{\lambda}{2}}I_{\frac{\lambda}{2}}(y)\,cos(x)
	\end{eqnarray}
	where $I_\nu$ is the usual modified Bessel function of the first kind.
	A routine check shows it satisfies $y^2\triangle\underline{f}+(1-\lambda)\underline{f}{}_y=0$.
	After multiplying $\underline{f}$ by the constant $\underline{C}=\inf_{x\in[-\pi/2,\pi/2]}f(x,1)/I_{\lambda/2}(1)$, we easily see $\underline{f}(x,y)-\epsilon<f(x,y)$ for any positive $\epsilon$, and so $\underline{f}(x,y)\le{}f(x,y)$ on $y\in[0,1]$, $x\in[-\pi/2,\pi/2]$.
	Note also that $\underline{f}$ has the correct behavior at $y=0$, namely that $\underline{f}(0,0)=0$ and $\underline{f}(0,y)=O(y^{\lambda})$.
	
	Finding a {\it super}function to complete the sandwich at $(0,0)$ is trickier.
	We must find a supersolution $\overline{f}$ that is not only larger than $f$ but also displays the correct behavior at the origin: $\overline{f}(0,y)=O(y^\lambda)$.
	We break the task into the cases $\lambda\in(0,1)$, $\lambda=1$, $\lambda\in(1,2)$, $\lambda=2$, and $\lambda>2$.
	
	{\bf Case that $\lambda\in(0,1)$}.
	Consider the function 
	\begin{eqnarray}
		\overline{f}(x,y)\;=\;y^\lambda+\frac12x^2-\frac{1}{1-\lambda}y.
	\end{eqnarray}
	One checks that $y^2\triangle\overline{f}+(1-\lambda)y\overline{f}{}_y=y^2-y$ which is non-positive on $y\in[0,1]$, so this is a supersolution.
	
	Because $\overline{f}$ has behavior $x^2$ on the line $\{y=0\}$, we can multiply $\overline{f}$ by a sufficiently large number, if necessary, to ensure that it bounds $f$ from above on the region $x\in[-1,1]$, $y\in[0,1]$.
	Now we have that $\underline{f}(x,y)\le{}f(x,y)\le{}\overline{f}(x,y)$ near $(0,0)$, so $f(x,y)=O(y^\lambda)$ near $\{y=0\}$ as desired.
	
	{\bf Case that $\lambda=1$}.
	This is the case that the operator is just $L=y^2\triangle$.
	In this case, assuming $f$ solves $\triangle{}f=0$ and $f\ge0$, the classical Hopf lemma ensures that $f(x,y)=O(y)$ at any boundary point.
	
	{\bf Case that $\lambda\in(1,2)$}.
	Consider the function
	\begin{eqnarray}
		\overline{f}(x,y)\;=\;y^\lambda\,+\,\frac12x^2\,-\,\frac{1}{2(2-\lambda)}y^2.
	\end{eqnarray}
	One checks that $y^2\triangle\overline{f}+(1-\lambda)y\overline{f}_\lambda
	=0$.
	Again we have on the boundary that $\overline{f}(x,0)=\frac12x^2$, and the $\overline{f}$ is positive at least for small values of $y$.
	So this is indeed a superfunction.
	Because $\lambda<2$ we have that $\overline{f}(0,y)=O(y^\lambda)$.
	Repeating the argument from the first case, we have $\underline{f}<f<\overline{f}$, and we conclude that $f(0,y)=O(y^\lambda)$.
	
	{\bf Case that $\lambda=2$}.
	In this case we must use the slightly more complicated barrier
	\begin{eqnarray}
		\overline{f}(x,y)\;=\;
		-1+\frac12\left(\sqrt{(x-1)^2+y^2}+\sqrt{(x+1)^2+y^2}\right). \label{EqnLambda2Barrier}
	\end{eqnarray}
	See Figure \ref{FigPiecewiseSol} for a depiction.
	It may be checked directly that $y^2\triangle\overline{f}-y\overline{f}{}_y=0$.
	On the boundary $\{y=0\}$ one may verify piecewise-linearity:
	\begin{eqnarray}
		\overline{f}(x,y)
		=\begin{cases}
		-1-x, \quad & x\in(-\infty,-1) \\
		0, & x\in[-1,1] \\
		-1+x, & x\in(1,\infty).
		\end{cases}
	\end{eqnarray}
	Arguing as in the other cases, we see that, possibly after multiplying $\overline{f}$ by a sufficiently large constant, that $\underline{f}(x,y)<f(x,y)<\overline{f}(x,y)$.
	\begin{figure}[h]
		\hspace{-1.5in}
		\begin{minipage}[c]{0.7\textwidth}
			\caption{
				\it
				Upper barrier $\overline{f}(x,y)$ of (\ref{EqnLambda2Barrier}), showing
				piecewise linearity at $\{y=0\}$.
				At $(0,0)$ the function is quadratic in $y$: $\overline{f}(0,y)=\frac12y^2+O(y^4)$.
				\\
				\\ \label{FigPiecewiseSol}
			} 
		\end{minipage} \hspace{-0.1in}
		\begin{minipage}[c]{0.2\textwidth}
			\includegraphics[scale=0.35]{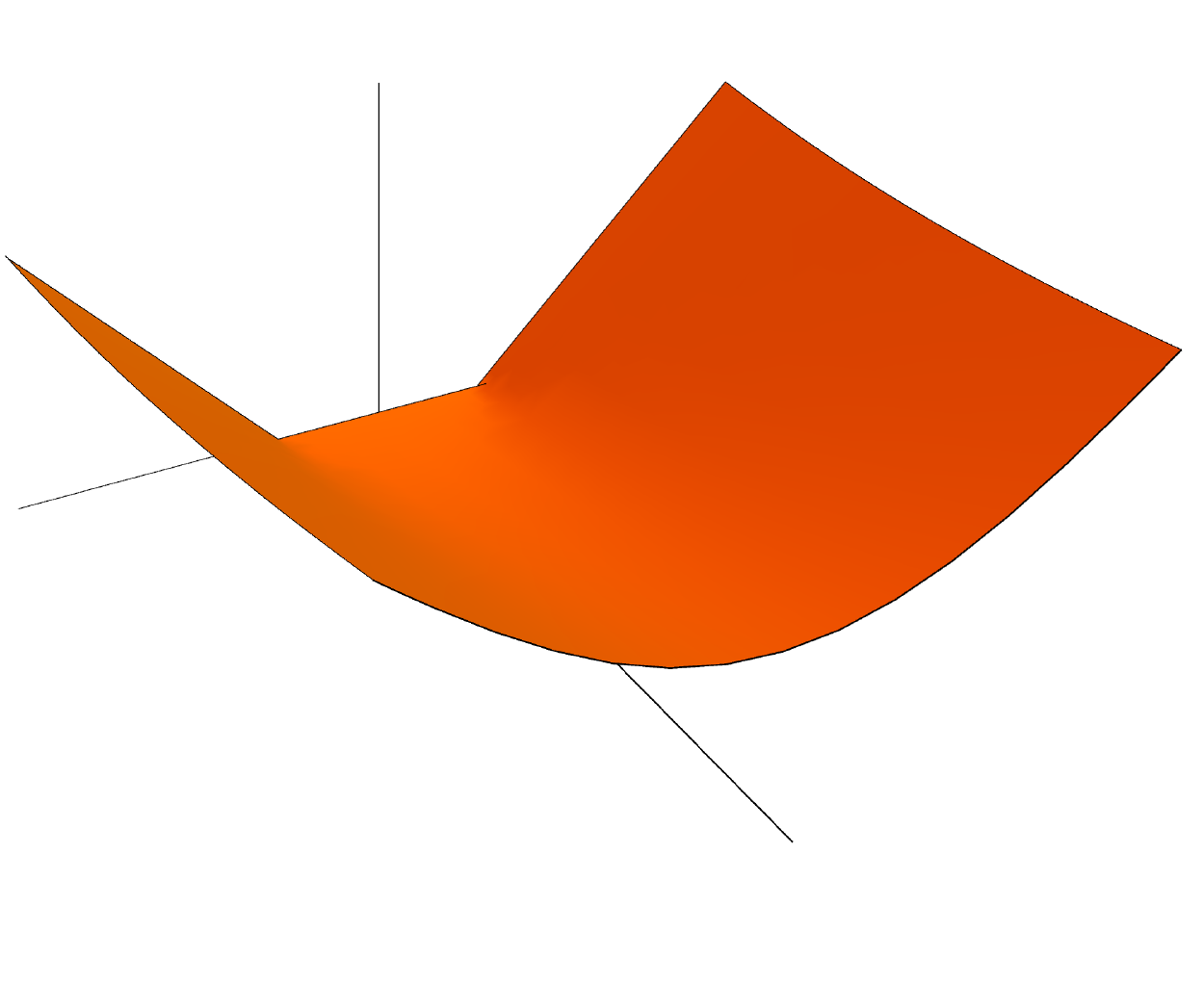}
		\end{minipage} \hfill
	\end{figure}
	After checking that, at $x=0$, we have $\overline{f}(0,y)=y^2+O(y^4)$, we conclude that indeed $f(0,y)=O(y^2)$.
	
	{\bf Case that $\lambda\in(2,\infty)$}.
	In this case we once again use the function $\overline{f}$ of (\ref{EqnLambda2Barrier}).
	This time we compute the strict inequality $y^2\triangle\overline{f}+(1-\lambda)\overline{f}{}_y<0$, so again $\overline{f}$ is a superfunction.
	As in the previous case, this allows us to conclude that $f(x,y)=O(y^2)$ at $\{y=0\}$.
	
	However, this means that $f(x,y)$ is actually $C^2$ near the boundary.
	Then we simply use $y^2f_{yy}+(1-\lambda)yf_y=-y^2f_{xx}$ which forces $f(x,y)=A(x)y^\lambda+H.O.T.$ for some one-variable function $A(x)$.
	We conclude, again, that $f(x,y)=O(y^\lambda)$.
\end{proof}

\section{Examples} \label{SecExamps}

These examples are roughly organized from most local phenomena to most global.
We start with examples showing the failure of our local results, Propositions \ref{PropFirstLocalGradEst} and \ref{PropUnspecifiabilityIntro} and Theorems \ref{ThmBoundHarnackLocalIntro} and \ref{ThmContinuityIntro}, and find out that when $b_1<1$ or if $b_1$ has no upper bound at the boundary, we have total failure: there is no local gradient estimate of the form $y|\nabla\log{}f|<D$, boundary values {\it can} be specified, the Harnack inequality at the boundary fails, and there is {\it no} continuity at $\{y=0\}$.
Example \ref{ExImps} shows that unspecifiability fails if local finiteness of $f$ is relaxed.

\subsection{Failures at $\{y=0\}$}

We fist give several examples showing how our ``local'' theorems fail if either $b_2<1$ or if the local boundedness assumption is forgotten.
When $b_2\in(0,1)$ we show it is always possible to specify boundary values at $\{y=0\}$, even though the results usually have bad differentiability at the boundary. \\

\refstepcounter{Examp}
{\bf Example \arabic{Examp}: Homogeneous solutions and steps.} \label{ExHomogAndSteps}

If $L$ has constant coefficients, meaning $b_1$, $b_2$, $c$ are constants, then we can reduce the equation $L(f)=0$ to an ordinary differential equation.
Assuming a solution of the form $f(x,y)=F(x/y)$, then $L(f)=0$ reduces to
\begin{eqnarray}
	\left(1+z^2\right)F_{zz}
	+\left(b_1+(2-b_2)z\right)F_{z}
	+cF\;=\;0
\end{eqnarray}
where $z=x/y$.
When $c\le0$ then $F$ is locally bounded, and when $b_2<1$ then $F$ is globally bounded.
\begin{figure}[h]
	\hspace{-2in}
	\begin{minipage}[c]{0.6\textwidth}
		\caption{
			\it Typical solution to $(1+z^2)F_{zz}+(2-b_2)zF_z=0$.\\
			\\ \label{FigZSol}
		} 
	\end{minipage} \hspace{-0.5in}
	\begin{minipage}[c]{0.2\textwidth}
		\includegraphics[scale=0.35]{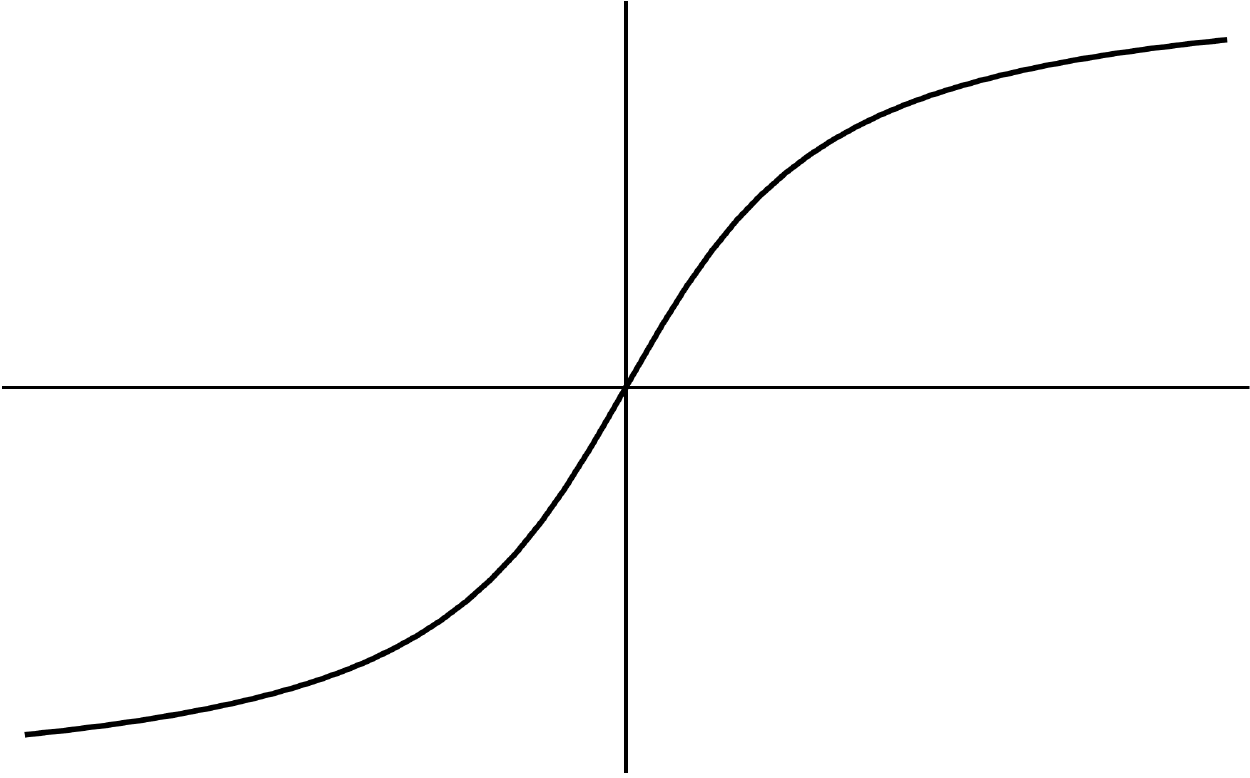}
	\end{minipage} \hfill
\end{figure}
Taking $b_1=c=0$ (for simplicity) we have the general solution
\begin{eqnarray}
	f(x,y)
	\;=\;C_1\frac{x}{y}\,
	{}_2F_1\left(\frac12, \frac12\left(2-b_2\right);\,\frac32;\,-\frac{x^2}{y^2} \right)
	+C_2
\end{eqnarray}
where ${}_2F_1$ is a hypergeometric function; see for example \S15.2({\it{i}}) of \cite{DLMF} for properties of this particular hypergeometric function.

If we also impose $b_2<1$ then actually the solution $f(x,y)$ is globally bounded, and we can easily choose $C_1$ and $C_2$ so that $\inf{}f=0$ and $\sup{}f=1$.
Then on the degenerate boundary the function $x\mapsto{}f(x,0)$ is a step function; see Figure \ref{FigSteps}.
This step-like solution has regularity $C^\infty$ on the interior; for $b_2\in(0,1)$ it is $C^{0,1-b_2}$ on $\{y=0\}$ except at the jump discontinuity at $(0,0)$, and for $b_2\le0$ it is $C^\infty$ except at $(0,0)$.
When $b_2\ge1$ this step is non-normalizable, and unbounded at $\{y=0\}$.
These examples showcase the major qualitative differences among the regimes $b_2\le0$, $b_2\in(0,1)$, and $b_2\ge1$.

\noindent\begin{figure}[h!] 
	\centering
	\noindent\begin{subfigure}[b]{0.4\textwidth} 
		\includegraphics[scale=0.5]{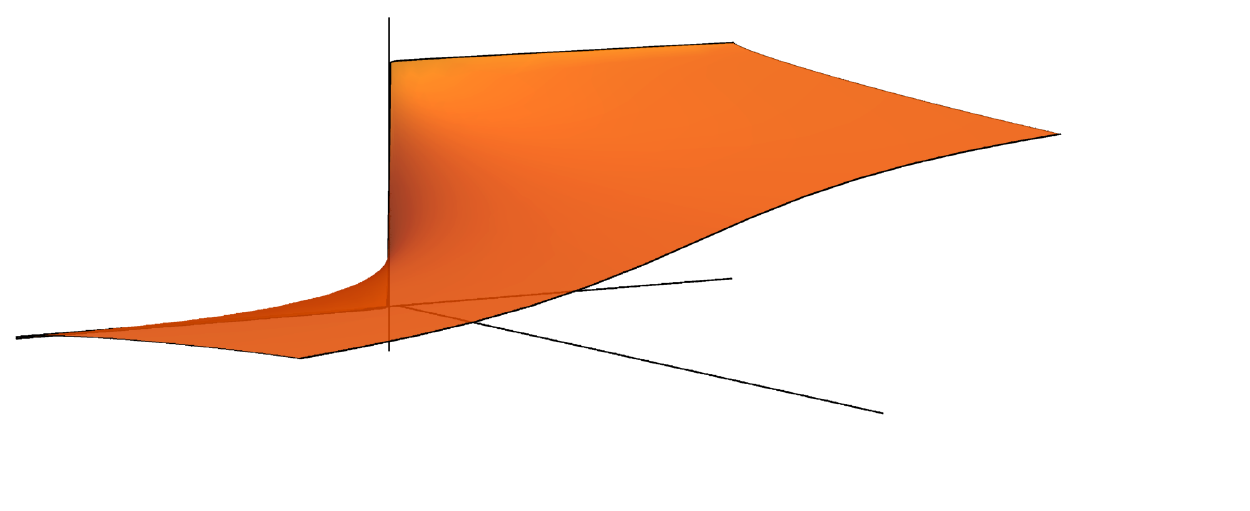}
		\caption{Step-like solution $f(x,y)=F(z/y)$ for $b_1=c=0$ and $b_2=0.5$; solution is $C^{0,1/2}$ at $\{y=0\}$ except at the discontinuity.
			\label{SubFigStepSol1}}
	\end{subfigure}
	\hspace{0.5in}
	\begin{subfigure}[b]{0.4\textwidth} 
		\includegraphics[scale=0.5]{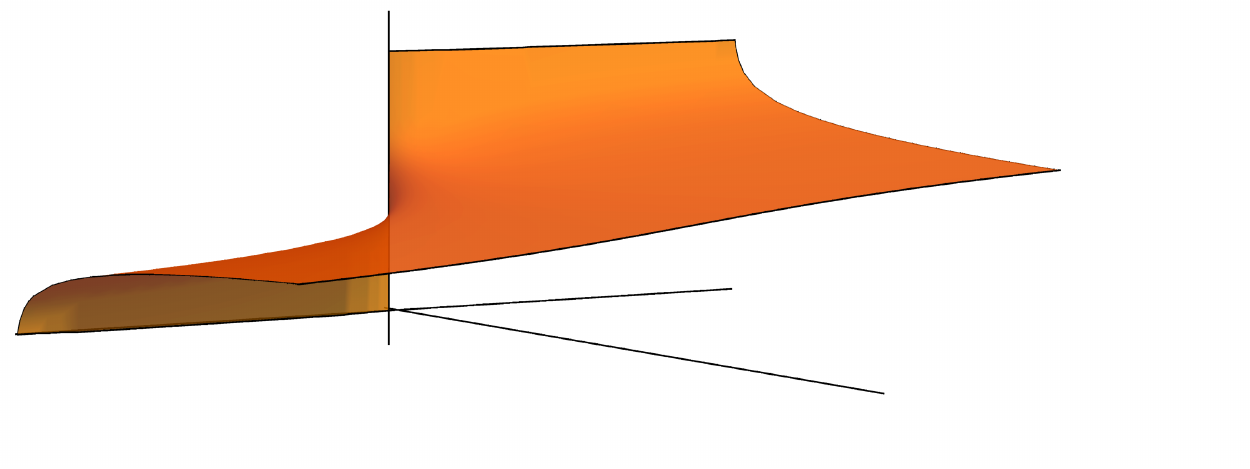}
		\caption{Step-like solution $f(x,y)=F(z/y)$ for $b_1=c=0$ and $b_2=0.9$; solution is only $C^{0,0.1}$ at $\{y=0\}$, except at the discontinuity. \label{SubFigStepSol2}}
	\end{subfigure}
	\caption{\it Solution with a step at $\{y=0\}$. }\label{FigSteps}
\end{figure}

This example shows that $f$ can be discontinuous on $\{y=0\}$ when $b_2<1$, meaning Theorem \ref{ThmContinuityIntro} is false when $b_2<1$.
This $f$ reaches its absolute minimum and absolute maximum on the boundary, contradicting the maximum principle when $b_2<1$.
Also, after adding a constant to $f$ so that $f\ge0$ but $f=0$ on a portion of the boundary, we see that the Harnack inequality, Theorem \ref{ThmBoundHarnackGlobal}, also fails for $b_2<1$. \\

\refstepcounter{Examp}
{\bf Example \arabic{Examp}: Impulses.} \label{ExImps}

With constant $b_1$, $b_2$, $c$, moving from the unit step to the unit impulse is simple: take a derivative with respect to $x$.
Assuming $c=0$ then a solution for $L(K)=0$ with a point-like singularity on the boundary is
\begin{eqnarray}
	K(x,y)\;=\;
	\frac{1}{y}\left(1+\left(\frac{x}{y}\right)^2\right)^{\frac{-2+b_2}{2}}
	Exp\left[-b_1\tan^{-1}\left(\frac{x}{y}\right)\right].
\end{eqnarray}
To justify the assertion that this is an impulse when restricted to $\{y=0\}$, note that after fixing any $y_0>0$, the integral $\int_{-\infty}^\infty{}K(x,y_0)\,dx$ gives a constant value, even as the function $\lim_{y_0\searrow0}K(\cdot,y_0)$ converges to zero everywhere except $x=0$, where it becomes unboundedly large.
One may check that in the case $b_2<1$ this value is finite, and we call the impulse {\it normalizable}: after multiplying by a constant then $x\mapsto{}K(x,0)$ is the unit Dirac-delta; see Figure \ref{SubFigImpNorm}.
If $b_2\ge1$ then $K(x,y)$ is no longer normalizable and the boundary singularity has infinite mass; see Figure \ref{SubFigImpNonNorm}.
\noindent\begin{figure}[h!]
	\centering
	\noindent\begin{subfigure}[b]{0.4\textwidth} 
		\includegraphics[scale=0.475]{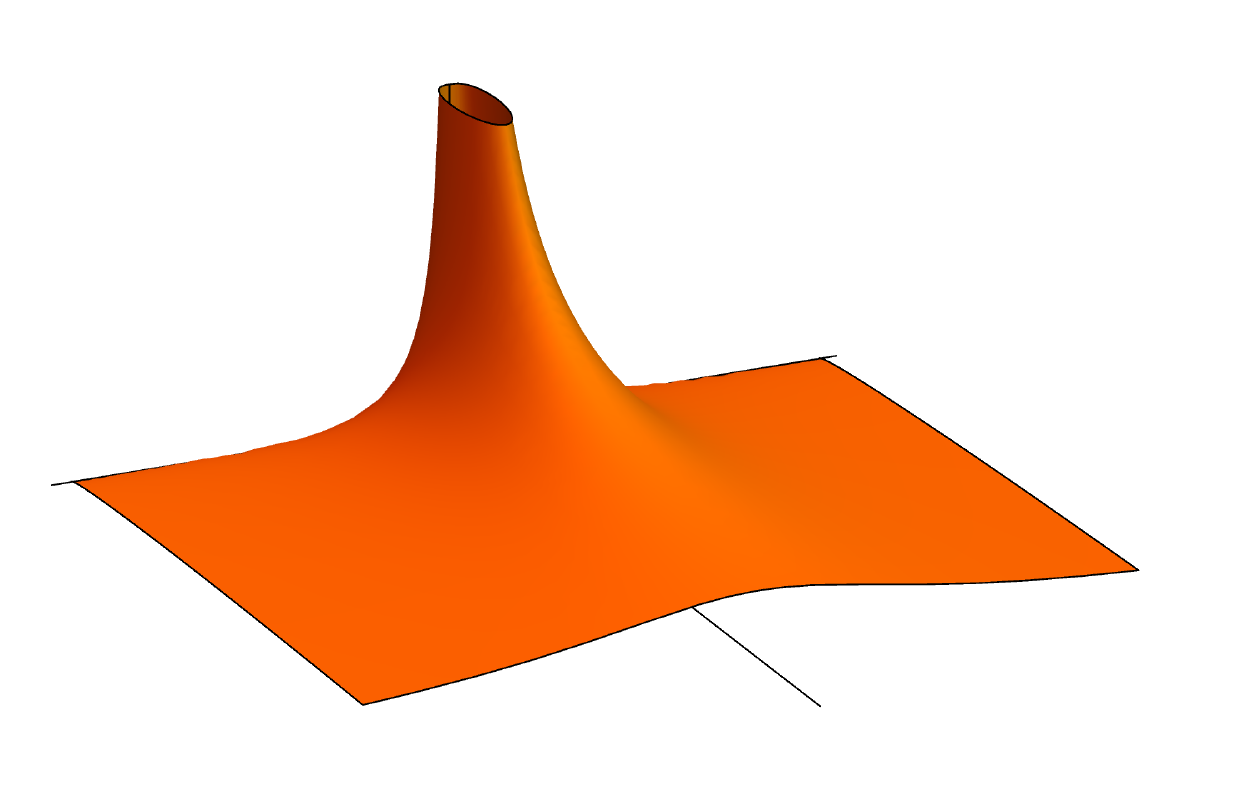}
		\caption{Solution for a normalizable impulse along $\{y=0\}$, with singular point $(0,0)$.
		Depicted is a solution for $b_2=0.5$.
		\label{SubFigImpNorm}}
	\end{subfigure}
	\hspace{0.25in}
	\begin{subfigure}[b]{0.4\textwidth} 
		\includegraphics[scale=0.475]{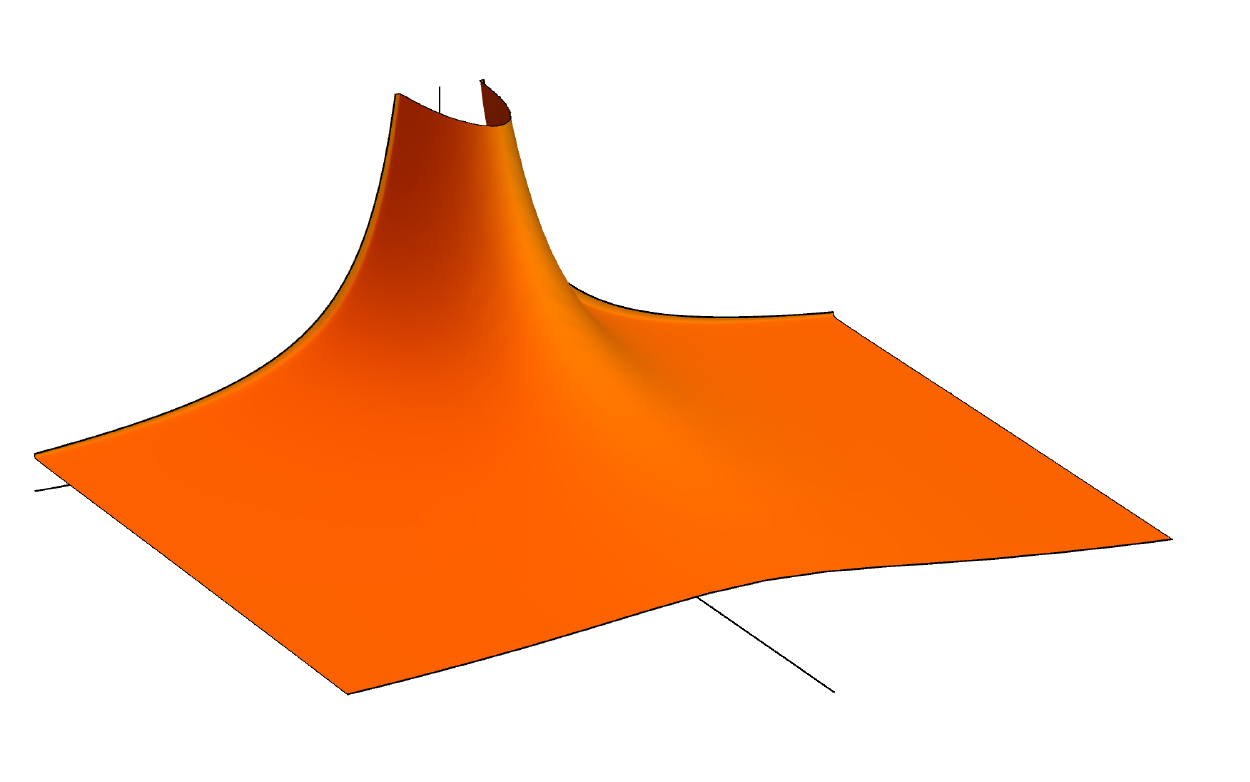}
		\caption{Solution for a non-normalizable impulse along $\{y=0\}$, with singular point $(0,0)$.
		Depicted is a solution for $b_2=1.5$. \label{SubFigImpNonNorm}}
	\end{subfigure}
	\caption{\it Solutions with an impulse at $\{y=0\}$. }\label{FigImps}
\end{figure}

This example shows, for instance, that the local finiteness conditions on Proposition \ref{PropUnspecifiabilityIntro} is indispensable.
Even further, it shows that boundary values {\it can} be specified whenever $b_2<1$, for using $K(x,y)$ as a kernel and using some function $f_0:\mathbb{R}\rightarrow\mathbb{R}$ as boundary conditions, then when $K$ is normalizable (which occurs when $b_2<1$) we have half-plane solutions
\begin{eqnarray}
	f(x,y)
	\;=\;\int_{-\infty}^\infty{}K(x-t,y)f_0(t)\,dt. \label{EqnKernelSolution}
\end{eqnarray}
Then assuming the usual conditions for convergence of (\ref{EqnKernelSolution}) indeed we have $f(x,0)=f_0(x)$.
We are therefore able to specify boundary values whenever $b_2<1$.

We can clearly observe the two ``phase changes'' in the behavior of solutions that we described in the introduction.
When $b_2\in(0,1)$ then smooth boundary values $f_0$ produce solutions $f$ that are only $C^{0,\alpha}$ near the boundary; indeed if $b_2$ is constant then at $\{y=0\}$ we ordinarily only get $f\in{}C^{0,1-b_2}$ and no better.
The second ``phase change'' occurs when $b_2\le0$, for then smooth boundary values produce smooth solutions. \\

\refstepcounter{Examp}
{\bf Example \arabic{Examp}: Failure of the maximum principle when $b_2<1$.} \label{ExMaxFailure}

For any constant $\lambda>0$ consider the equation
\begin{eqnarray}
	y^2\triangle{}f\,+\,(1-\lambda)yf_y\;=\;0.
\end{eqnarray}
On the region $\Omega=[-\pi/2,\pi/2]\times[0,1]$ we have the bounded non-negative solution
\begin{eqnarray}
	f(x,y)\;=\;y^{\frac{\lambda}{2}}K_{\frac{\lambda}{2}}(y)\,\cos(x)\;+\;C
\end{eqnarray}
where $K_\nu(y)$ is the familiar modified Bessel function of the second kind.
A maximum is reached at $(0,0)$, one example of which is depicted in Figure \ref{FigMaxPrinViolation}.
This demonstrates the failure of the maximum principle, Theorem \ref{ThmWeakMaxIntro}, when $b_1<1$.
We remark that when $\lambda\in(0,1)$, $f$ has regularity $C^{0,\lambda}$ at $\{y=0\}$.

\begin{figure}[h]
	\hspace{-1.9in}
	\begin{minipage}[c]{0.7\textwidth}
		\caption{
			\it A solution to $y^2\triangle{}f+\frac12yf_y=0$ with compact support, clearly violating the maximum principle.
			Depicted is $f(x,y)=y^{\frac14}K_{\frac14}(y)\cos(x)-0.55$.
			\\ \label{FigMaxPrinViolation}
		} 
	\end{minipage} \hspace{-0.5in}
	\begin{minipage}[c]{0.2\textwidth}
		\includegraphics[scale=0.475]{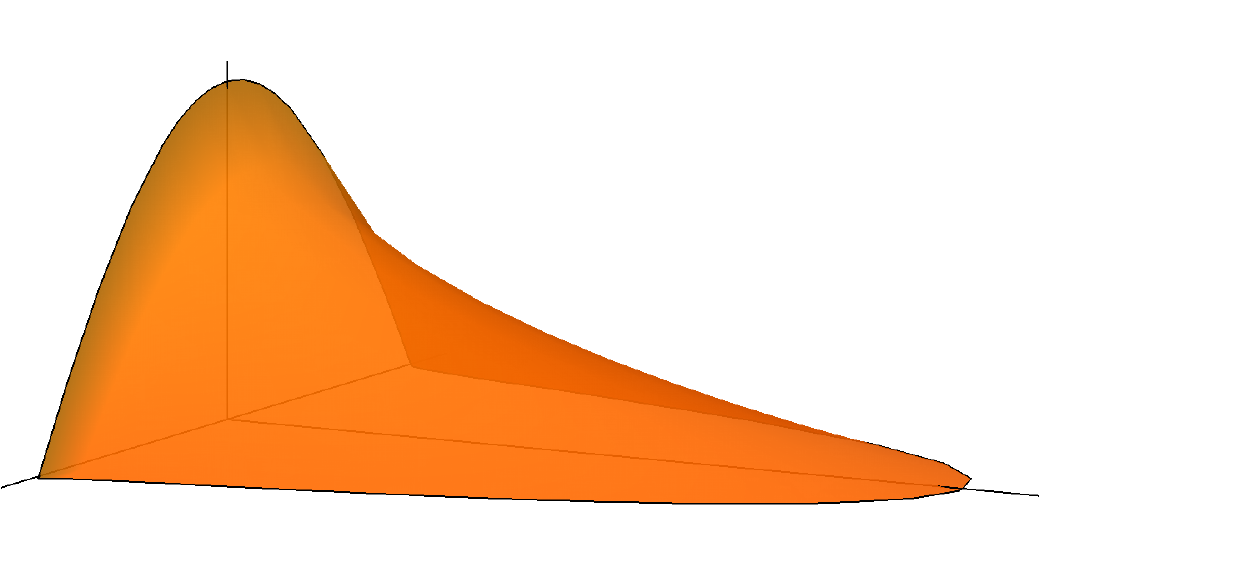}
	\end{minipage} \hfill
\end{figure}

\subsection{Examples on subdomains}

Several of our theorems are global in nature, requiring solutions $f\ge0$ to exist on the closed half-plane.
Here we look at the cases where a solution $f\ge0$ exists only on a strip $y\in[0,1]$ or on a half-plane of the form $y\in[1,\infty)$.
We see that {\it uniqueness} of solutions, with given boundary values, completely fails on such subdomains, as does the Liouville theorem. \\

\refstepcounter{Examp}
{\bf Example \arabic{Examp}: Non-Uniqueness on strips.} \label{ExNonUniqStrips}

For any constant $\lambda>0$ the equation
\begin{eqnarray}
	y^2\triangle{}f\,+\,(1+\lambda)yf_y\;=\;0 \label{EqnConstCoefNoB1}
\end{eqnarray}
has non-negative solutions on the strip $y\in[0,j_{\frac{\lambda}{2}},1]$
\begin{eqnarray}
	f(x,y)\;=\;y^{-\frac{\lambda}{2}}J_{\frac{\lambda}{2}}(y)\,e^{x}
\end{eqnarray}
where $J_{\frac{\lambda}{2}}$ is the Bessel function of the first kind and $j_{\frac{\lambda}{2},1}$ is its first zero.
The function $f$ is locally bounded but not bounded, non-negative, is $C^\infty$, and is precisely zero on the non-degenerate boundary $\{y=j_{\frac{\lambda}{2},1}\}$.
The solutions grow exponentially in the $x$-direction.
This shows non-uniqueness on the strip, even when values are specified on the non-degenerate boundary.
This also shows that the Liouville theorem, Theorem \ref{ThmLiouvilleFirstStatement}, certainly fails on subdomains of $H^2$.
(Incidentally, it also shows the necessity of some kind of growth assumption in Theorems 1.6 and 1.11 of \cite{FP}, even under strong differentiability assumptions.) \\

\refstepcounter{Examp}
{\bf Example \arabic{Examp}: Failure of almost-monotonicity and the Liouville theorem on half-planes.} \label{ExOtherHalfPlanes}

The ``almost monotonicity'' theorem is truly a global theorem and is false on subdomains, even unbounded subdomains, as we show here.
Consider the half-plane $y\in[1,\infty)$ and the equation (\ref{EqnConstCoefNoB1}).
The function
\begin{eqnarray}
	f(x,y)\;=\;1-y^{-\lambda}
\end{eqnarray}
is a positive solution to $y^2\triangle{}f+(1+\lambda)f_y=0$ on $y\in(0,\infty)$.
But we have $\lim_{y\rightarrow0}f(x,y)=1$, which is not the global minimum of $f$, and so almost-monotonicity fails.
This also shows the Liouville theorem fails on such a half-plane. \\

\subsection{The Heston-type operators}

The Heston operator, in appropriate coordinates such as in (\ref{EqnHestonModified}), is \begin{eqnarray}
	L_H\;=\;
	y\left(\partial_x\partial_x+\partial_y\partial_y\right)
	+(b_1+B_1y)\partial_{x}
	+(b_1+B_2y)\partial_{y}
	-r \label{EqnHestExEqn}
\end{eqnarray}
where we take $b_1,b_2,B_1,B_2,r$ to be constant---in the Heston model it is required that $b_2>0$, $B_2<0$, and it is normally assumed interest rates are positive, $r>0$ (although in the post-crisis world this may be questionable).
After multiplying through by $y$, we see that the operator
\begin{eqnarray}
	y\cdot{}L_H\;=\;
	y^2\triangle
	+y(b_1+B_1y)\frac{\partial}{\partial{x}}
	+y(b_1+B_2y)\frac{\partial}{\partial{y}}
	-ry
\end{eqnarray}
has the Euler-type degeneracy at $\{y=0\}$ that we study in this paper.
However the coefficients are not bounded for large $y$, and because of this, for solutions $f\ge0$ of $L_H(f)=0$, we expect our local results to hold but our global results to fail. \\

\refstepcounter{Examp}
{\bf Example \arabic{Examp}: Global solutions of $L_H(f)=0$.} \label{ExHestonGradFail}

With $L_H$ as in (\ref{EqnHestExEqn}) and taking $b_1=B_2=0$, $r=1$, and setting $b_2=1$ and $B_2=\pm1$ we see that
\begin{eqnarray}
	\begin{aligned}
		f(x,y)\;=\;1+y \quad \text{solves} \quad y^2\triangle{}f+y(1+y)f_y-yf=0, \\
		f(x,y)\;=\;e^y \quad \text{solves} \quad y^2\triangle{}f+y(1-y)f_y-yf=0.
	\end{aligned}
\end{eqnarray}
Both examples are $C^\infty$ and non-negative.
The example with $B_2=+1$ obeys the interior gradient estimate Proposition \ref{PropIntGradFirst}, but the $B_2=-1$ solution violates it.
These examples both violate the two Harnack inequalities Theorems \ref{ThmBoundHarnackLocalIntro} and \ref{ThmBoundHarnackGlobal}, the almost-monotonicity theorem Proposition \ref{PropAlmostMonoIntro}.
Both violate the Liouville theorem, Theorem \ref{ThmLiouvilleFirstStatement}.

But one might notice that $r>0$ in both cases, and $B_2\ge0$ is forbidden in the Heston financial model.
We have been unable to find an entire, locally finite solution $f\ge0$ to the Heston equation with both $r\le0$ and $B_2<0$.
This motivates Conjecture 2 of the Introduction, which we restate here for convenience.

{\bf Conjecture: The Liouville theorem for the time-independent Heston equation.}
{\it A non-negative, locally finite solution $L_H(f)=0$ for the operator $L_H$ of (\ref{EqnHestExEqn}) on the closed half-plane $\overline{H}{}^2$ with $b_1>0$, $B_2<-\epsilon^2$ for some non-zero $\epsilon$ and with non-positive interest rate $r\le0$, is necessarily constant.
Such a solution is zero if $r<0$.}

The absurdity of negative rates appears frequently in markets now---in some cases throughout the entire term structure\footnote{For the first time in July 2016, yields on Swiss government debt were negative out to 50 years.} and in some cases on private debt.\footnote{At the time of writing in autumn 2019, Barron's magazine has reported that \$600B in private debt is trading at negative interest rates.}
Our conjecture, if true, may have consequences for financial market modeling with negative interest rates.

\subsection{Global solutions and supersolutions} \label{SubSecExampsNonzeroC}

Most of our theorems require $c\ge0$.
Our Liouville theorem requires $c=0$.
We show that these restrictions are indeed necessary. \\

\refstepcounter{Examp}
{\bf Example \arabic{Examp}: Failure of almost-monotonicity and Liouville when $c<0$.} \label{ExNegC}

Taking $c=-y^2/(1+y^2)$ we see that $c$ is bounded and negative on the half-plane.
Then the equation
\begin{eqnarray}
	y^2\triangle{f}\,+\,yf_y\,-\,\frac{y^2}{1+y^2}f\;=\;0
\end{eqnarray}
has solution $f(x,y)=E(-y^2)$, where $E$ is the elliptic integral of the second kind.
This solution is positive, smooth, bounded at $y=0$, and unbounded at $y=\infty$ where it grows like a multiple of $y$.
Therefore it violates the strong constraints on behavior at infinity that almost-monotonicity imposes in the $c\ge0$ case.

The interior gradient estimate Proposition \ref{PropIntGradFirst} remains valid, as it must.
But in addition to the failure of almost-monotonicity, we see the failure of the Liouville theorem, Theorem \ref{ThmLiouvilleFirstStatement}. \\

\refstepcounter{Examp}
{\bf Example \arabic{Examp}: Failure of the Liouville theorem when $c\ge0$.} \label{ExPosC}

Consider the functions
\begin{eqnarray}
	\begin{aligned}
	&c\;=\;\begin{cases}
		0, \quad & 0\le{}y\le1 \\
		\frac14, &  1<y<\infty,
	\end{cases} \\
	&F(y)\;=\;\begin{cases}
		1, \quad & 0\le{}y\le1 \\
		y^{-\frac12}\,+\,\frac12{}y^{-\frac12}\log(y), &  1<y<\infty.
	\end{cases}
	\end{aligned}
\end{eqnarray}
Setting $f(x,y)=F(y)$ we see that $f\in{}C^{1,1}$, $f>0$, $f$ is uniformly bounded, and weakly solves $y^2\triangle{}f+2yf_y+cf=0$ on the entirety of the closed half-plane $\overline{H}{}^2$.
Yet this function is not constant, showing the Liouville theorem can fail when $c\ge0$. \\

\refstepcounter{Examp}
{\bf Example \arabic{Examp}: Failure of the Liouville theorem for superfunctions.} \label{ExSupers}
The classical Liouville theorem holds for supersolutions: if $\triangle{}f\le0$ weakly on $\mathbb{R}^2$ and $f$ is entire and non-negative, then $f$ is constant.
One may wonder if the Liouville theorem of this paper is similarly true when $L(f)\le0$.
But it is not true.
Depicted in Figure \ref{FigSuperFun} is the function
\begin{eqnarray}
	f(x,y)\;=\;y^{-2}\left(\sqrt{2}-\sqrt{1-x^2-y^2+\sqrt{(1-x^2-y^2)^2+4y^2}}\right)
\end{eqnarray}
which is uniformly bounded and satisfies $y^2\triangle{}f+3yf_y=0$ everywhere except along a singular ray $\{x=0,\,y\ge1\}$, where $L(f)\le0$ in the weak or the viscosity sense.
\begin{figure}[h]
	\hspace{-2in}
	\begin{minipage}[c]{0.65\textwidth}
		\caption{
			\it A uniformly bounded superfunction for $L=y^2\triangle+3y\partial_y$.\\
			\\ \label{FigSuperFun}
		} 
	\end{minipage} \hspace{-0.5in}
	\begin{minipage}[c]{0.2\textwidth}
		\includegraphics[scale=0.475]{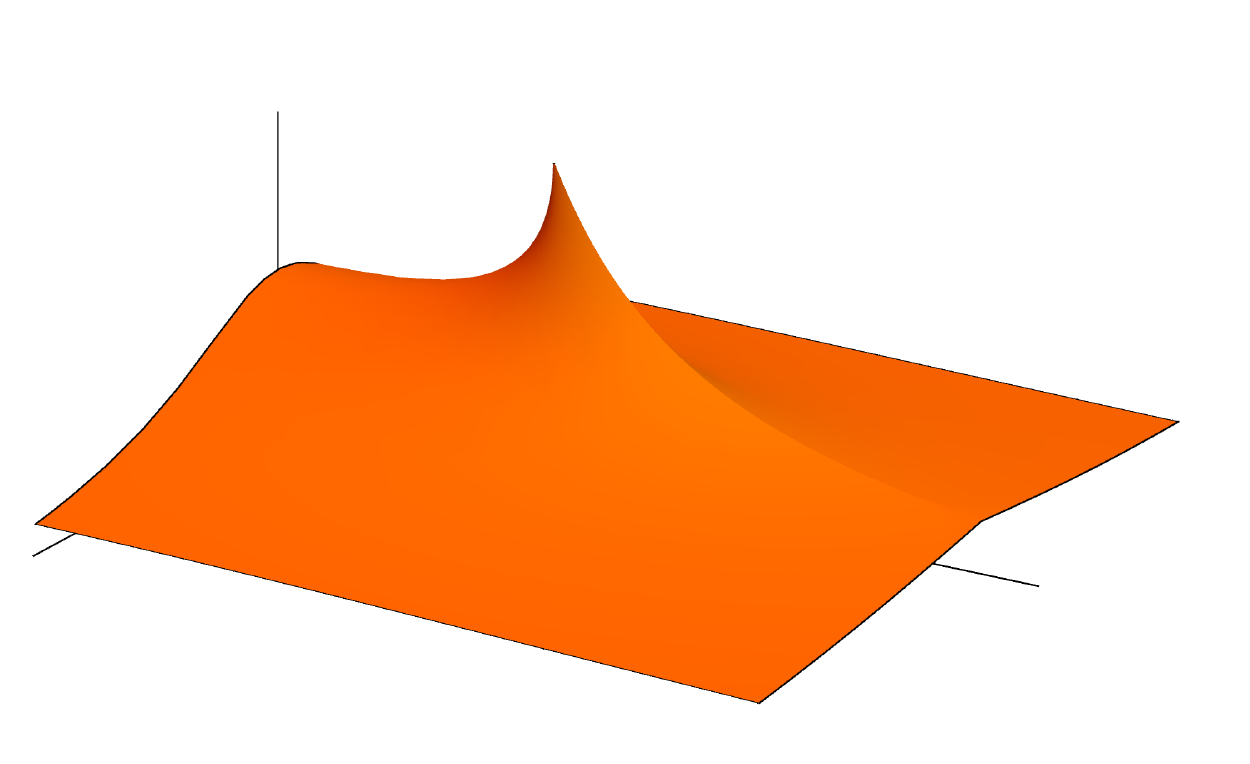}
	\end{minipage} \hfill
\end{figure}
This superfunction violates even our most basic result, the interior gradient bound, Proposition \ref{PropInteriorInitial}, for at the salient the function is $C^{0,\frac12}$ but no better.

\end{document}